\documentclass[10pt]{article}

\usepackage[margin=1in]{geometry}
\usepackage[latin1]{inputenc}
\usepackage[english]{babel}
\usepackage[babel]{csquotes}
\usepackage[noadjust]{cite}
\usepackage{amssymb, amsmath, amsthm, amsfonts}
\usepackage[toc,page]{appendix}
\usepackage{enumerate}
\usepackage[shortlabels]{enumitem}
\usepackage[colorlinks=true, pdfstartview=FitV, linkcolor=colorLink, citecolor=colorCite, urlcolor=colorLink, linktocpage=true]{hyperref}
\usepackage{tikz, graphicx, pgfplots, floatrow}
\usepackage{tocloft}
\usepackage{upgreek}
\usepackage{bbm}
\usepackage{longtable}
\usepackage{booktabs}
\usepackage{float}
\usepackage{caption}
\makeatletter
\renewcommand{\paragraph}{%
\@startsection{paragraph}{4}%
{\z@}{1.5ex \@plus 1.5ex \@minus .2ex}{-0.7em}%
{\normalfont\normalsize\bfseries}%
}
\makeatother

\setlist[itemize]{leftmargin=5mm}

\theoremstyle{plain}
\newtheorem{theorem}{Theorem}[section]

\newtheorem{lemma}[theorem]{Lemma}
\newtheorem{proposition}[theorem]{Proposition}

\theoremstyle{definition}
\newtheorem{definition}[theorem]{Definition}

\newtheorem{remark}[theorem]{Remark}
\newtheorem*{acknowledgements}{Acknowledgements}

\numberwithin{equation}{section}

\makeatletter
\newcommand\RedeclareMathOperator{%
\@ifstar{\def\rmo@s{m}\rmo@redeclare}{\def\rmo@s{o}\rmo@redeclare}%
}
\newcommand\rmo@redeclare[2]{%
\begingroup \escapechar\m@ne\xdef\@gtempa{{\string#1}}\endgroup
\expandafter\@ifundefined\@gtempa
{\@latex@error{\noexpand#1undefined}\@ehc}%
\relax
\expandafter\rmo@declmathop\rmo@s{#1}{#2}}
\newcommand\rmo@declmathop[3]{%
\DeclareRobustCommand{#2}{\qopname\newmcodes@#1{#3}}%
}
\@onlypreamble\RedeclareMathOperator
\makeatother
\def\R{\mathbb{R}}
\def\C{\mathbb{C}}
\def\S{\mathbb{S}}
\def\P{\mathbb{P}}
\def\Z{\mathbb{Z}}
\def\E{\mathbb{E}}
\def\N{\mathbb{N}}
\def\G{\mathbb{G}}

\def\HH{\mathcal{H}}
\def\CC{\mathcal{C}}

\def\FF{\mathcal{F}}
\def\VV{\mathcal{V}}
\def\EE{\mathcal{E}}

\def\MM{\mathcal{M}}
\def\GG{\mathcal{G}}
\def\HH{\mathcal{H}}
\def\TT{\mathcal{T}}
\def\SS{\mathcal{S}}
\def\MM{\mathcal{M}}
\def\PPP{\mathcal{P}}
\def\dotSS{\dot{\mathcal{S}}}

\def\frkf{\mathfrak{f}}
\def\frkh{\mathfrak{h}}

\def\frkw{\mathfrak{w}}

\def\e{\mathbf{e}}
\def\x{\mathbf{x}}
\def\y{\mathbf{y}}
\def\v{\mathbf{v}}
\def\X{\mathbf{X}}

\def\PP{\mathbf{P}}

\def\he{\hat{\mathbf{e}}}
\def\hx{\hat{\mathbf{x}}}
\def\hy{\hat{\mathbf{y}}}

\def\dotX{\dot{\mathbf{X}}}

\def\sfQ{\mathrm{Q}}
\def\sfR{\mathrm{R}}
\def\sfI{\mathrm{I}}
\def\sfV{\mathrm{V}}
\def\sfH{\mathrm{H}}
\def\sfS{\mathrm{S}}

\def\SLE{\mathrm{SLE}}

\def\eps{\upvarepsilon}
\def\tau{\uptau}
\def\eta{\upeta}
\def\sigma{\upsigma}
\def\delta{\updelta}
\def\rho{\uprho}
\def\xi{\upxi}

\def\d{\mathrm{d}}
\def\theta{\upvartheta}
\def\zeta{\upzeta}

\renewcommand{\tilde}[1]{\widetilde{#1}}
\renewcommand{\hat}[1]{\widehat{#1}}
\renewcommand{\bar}[1]{\overline{#1}}
\renewcommand{\mod}[1]{\;\mathrm{mod}(#1)}
\renewcommand{\l}[1]{\bigl{#1}}
\renewcommand{\r}[1]{\bigr{#1}}

\DeclareMathOperator\wind{wind}

\DeclareMathOperator\diam{diam}
\DeclareMathOperator\dH{\d^{\mathrm{H}}}
\DeclareMathOperator\dTV{\d^{\mathrm{TV}}}
\DeclareMathOperator\dCMP{\d^{\mathrm{CMP}}}

\DeclareMathOperator\dCMPloc{\d_{\mathrm{loc}}^{\mathrm{CMP}}}
\RedeclareMathOperator\Im{Im}
\RedeclareMathOperator\Re{Re}

\DeclareMathOperator\length{length}
\DeclareMathOperator\Out{Outrad}
\DeclareMathOperator\Area{Area}
\DeclareMathOperator\midpoint{mid}
\RedeclareMathOperator\S{S}
\DeclareMathOperator\Var{\mathbb{V}ar}
\DeclareMathOperator\Cov{\mathbb{C}ov}
\RedeclareMathOperator\deg{deg}
\allowdisplaybreaks

\definecolor{colorLink}{RGB}{0,100,162}
\definecolor{colorCite}{RGB}{8,124,100}
\title{Scaling limits of planar maps under the Smith embedding}
\author{
	\begin{tabular}{c} Federico Bertacco\\ [-5pt] \small Imperial College London \end{tabular}
	\begin{tabular}{c} \\[-5pt]\small  \end{tabular}
	\begin{tabular}{c} \\[-5pt]\small  \end{tabular}
	\begin{tabular}{c} Ewain Gwynne\\ [-5pt] \small University of Chicago \end{tabular}
	\begin{tabular}{c} \\[-5pt]\small  \end{tabular}
	\begin{tabular}{c} \\[-5pt]\small  \end{tabular}
	\begin{tabular}{c} Scott Sheffield\\ [-5pt] \small MIT \end{tabular}
}

\date{ }
\begin{document}	

\maketitle

%%%%%%%%%%%%%%%%%%%%%%%%%%%%%%%%%%%%%%%%%%
\begin{abstract}
\noindent
The Smith embedding of a finite planar map with two marked vertices, possibly with conductances on the edges, is a way of representing the map as a tiling of a finite cylinder by rectangles. In this embedding, each edge of the planar map corresponds to a rectangle, and each vertex corresponds to a horizontal segment. Given a sequence of finite planar maps embedded in an infinite cylinder, such that the random walk on both the map and its planar dual converges to Brownian motion modulo time change, we prove that the a priori embedding is close to an affine transformation of the Smith embedding at large scales. By applying this result, we prove that the Smith embeddings of mated-CRT maps with the sphere topology converge to $\gamma$-Liouville quantum gravity ($\gamma$-LQG).
\end{abstract}
%%%%%%%%%%%%%%%%%%%%%%%%%%%%%%%%%%%%%%%%%%

\tableofcontents

%%%%%%%%%%%%%%%%%%%%%%%%%%%%%%%%%%%%%%%%%%
\section{Introduction}

%%%%%%%%%%%%%%%%%%%%%%%%%%%%%%%%%%%%%%%%%%
\subsection{Motivation}
Over the past few decades, there has been a large amount of interest in the study of random planar maps, i.e., graphs embedded in the plane viewed modulo orientation-preserving homeomorphisms. Since the foundational work of Polyakov in the context of bosonic string theory \cite{Pol81}, it has been believed that various types of random planar maps converge, in various topologies, to limiting random surfaces called \emph{Liouville quantum gravity} (LQG) surfaces. The rigorous mathematical study of LQG has been explored, e.g., in works by Duplantier and Sheffield \cite{DS11} and Rhodes and Vargas \cite{RV_KPZ}.
Roughly speaking, LQG surfaces can be thought of as random two-dimensional Riemannian manifolds parameterized by a fixed underlying Riemann surface, indexed by a parameter $\gamma\in (0,2]$. These surfaces are too rough to be Riemannian manifolds in the literal sense, but one can still define, e.g., the associated volume form (area measure) and distance function (metric) via regularization procedures \cite{Kahane, DS11, RV_KPZ, Berestycki_Elementary, DDDF_first, Weak_Metric, GM_Conf, GM_Uniq}. Many properties of the $\gamma$-LQG area measure are well-known \cite{RVReview, BerMulti, BerPow}.

\medskip
\noindent
One way of formulating the convergence of random planar maps toward LQG surfaces is to consider so-called \emph{discrete conformal embeddings} of the random planar maps. Here, a discrete conformal embedding refers to a particular way of drawing the map in the plane, which is in some sense a discrete analog of the Riemann mapping. Suppose we have a random planar map with $n$ vertices, along with a discrete conformal embedding of the map that maps each vertex to a point in $\C$. This embedding creates a natural measure on the plane, with each vertex given a mass of $1/n$. In many settings, it is natural to conjecture that as $n$ tends to infinity, the measure should converge weakly to the $\gamma$-LQG area measure, with the parameter $\gamma$ depending on the particular planar map model under consideration. Additionally, the random walk on the embedded map is expected to converge in law to two-dimensional Brownian motion modulo time parameterization (more precisely, the parameterized walk should converge to the so-called \emph{Liouville Brownian motion} \cite{Ber_LBM, GRV_LBM, GB_LBM}).
Several precise scaling limit conjectures for random planar maps toward LQG surfaces were formulated, e.g., in \cite{DS11, LQG_Sphere, She_Zipper, DMS21}. However, this very general convergence ansatz has only been rigorously proven in a few specific settings (see below). 

\medskip
\noindent
One of the challenges in formulating a general scaling limit result for the embedding of random planar maps is the existence of numerous discrete conformal embeddings that could be regarded as natural in some sense. We collect here some of the most commonly employed discrete conformal embeddings.
\begin{itemize}
\item The \emph{circle packing} (see \cite{Nac20} for a review), which represents the map as the tangency graph of a collection of non-overlapping circles\footnote{We refer to \cite{RS87} for a proof of the fact that the circle packing for lattice approximations of planar domains gives an approximation of the Riemann mapping.}.
\item The \emph{Smith embedding} (a.k.a.\ \emph{rectangle packing}), which will be the focus of the present paper, was introduced by Brooks, Smith, Stone, and Tutte in \cite{BSST40}. It is another popular method of embedding planar graphs, and it is defined by means of a rectangle tiling of either a cylinder or a rectangle, and in which vertices of the planar map correspond to horizontal segments in the Smith embedding, and edges of the planar map correspond to rectangles in the Smith embedding\footnote{We refer to \cite{GP_square} for a proof of the fact that the Smith embedding for fine-mesh lattice graphs gives an approximation of the Riemann mapping.}. 
Several papers have studied properties of the Smith embedding of planar maps \cite{BS96, Geo16, Hutch_Smith, Carm_Smith}.
\item Other examples of discrete conformal embeddings include the \emph{Tutte embedding} \cite{GMS_Tutte}, the \emph{Cardy--Smirnov embedding} \cite{HS_Cardy}, and the \emph{Riemann uniformization embedding}, obtained by viewing the planar map as a piecewise flat two-dimensional Riemannian manifold where the faces are identified with unit side length polygons. 
\end{itemize}

\noindent
A few cases of the conjecture that LQG describes the scaling limit of random planar maps under discrete conformal embeddings have been proven. For example, in \cite{GMS_Tutte}, Gwynne, Miller, and Sheffield established the convergence to $\gamma$-LQG under the Tutte embedding for a one-parameter family of random planar maps defined using pairs of correlated Brownian motions, known as the \emph{mated-CRT maps} (see below for a definition of this family of random planar maps). Moreover, in \cite{GMS_Voronoi}, the same authors proved that the Tutte embedding of the Poisson Voronoi tessellation of the Brownian disk converges to $\sqrt{8/3}$-LQG. In \cite{HS_Cardy}, Holden and Sun proved that the scaling limit of uniformly sampled triangulations under the Cardy--Smirnov embedding converges to $\sqrt{8/3}$-LQG. Finally, let us also mention that in \cite{Nach_Circle}, the authors studied the circle-packing of the mated-CRT map and showed that there are no macroscopic circles in the circle packing of this random planar map. Roughly speaking, the main goal of this paper is to provide a general convergence result for the Smith embedding of planar maps, which works whenever random walk on both the map and its dual approximate Brownian motion.

%%%%%%%%%%%%%%%%%%%%%%%%%%%%%%%%%%%%%%%%%%
\subsection{Main result}
The main result of this paper concerns the scaling limit of general (random) planar maps under the Smith embedding. More precisely, we consider a sequence of finite planar maps, each of which comes with an embedding into an infinite cylinder, referred to as the \emph{a priori embedding}. Under the assumption that both the sequence of maps and their duals satisfy an invariance principle, we show that the a priori embedding is close to an affine transformation of the Smith embedding at large scales. We then apply this result to prove the convergence of the Smith embeddings of mated-CRT maps to $\gamma$-LQG.

\medskip
\noindent
One advantage of the version of the Smith embedding considered in this paper is that its definition is particularly natural for random planar maps without boundary. This is in contrast to other embeddings under which random planar maps have been shown to converge to LQG (such as the Tutte embedding \cite{GMS_Tutte} and the Cardy--Smirnov embedding \cite{HS_Cardy}) which are most naturally defined for planar maps with boundary.

\medskip
\noindent
Throughout the paper, we always write \emph{weighted planar map} for a planar map with edge conductances.  Moreover, throughout the article, we allow all of our planar maps to have multiple edges and self-loops. In order to state our main theorem, we need to introduce some notation. Given a planar graph $\GG$, we denote the sets of vertices, edges, and faces of $\GG$ by $\VV\GG$, $\EE\GG$, and $\FF\GG$, respectively. We consider a doubly marked finite weighted planar map $(\GG, c, v_0, v_1)$, where $v_0$, $v_1 \in \VV\GG$ are the two marked vertices, and $c=\{c_e\}_{e\in\EE\GG}$ is a collection of positive unoriented weights called conductances. We assume that we are given a \emph{proper embedding} of the map in the infinite cylinder $\CC_{2\pi} := \R/2\pi\Z \times \R$ \label{pag_inf_cyl} in the sense of the following definition.
\begin{definition}
\label{def_proper_embedd}
An embedding of the quadruple $(\GG, c, v_0, v_1)$ in the infinite cylinder $\CC_{2\pi}$ is said to be \emph{proper} if:
\begin{enumerate}[(a)]
	\item the edges in $\EE\GG$ are continuous and do not cross;
	\item the graph $\GG$ is connected;
	\item the two marked vertices $v_0$ and $v_1$ are mapped to $-\infty$ and $+\infty$, respectively.
\end{enumerate}
\end{definition}

\noindent
We observe that, if $(\GG, c, v_0, v_1)$ is properly embedded in $\CC_{2\pi}$, then the set of unmarked vertices $\VV\GG$ is contained in $\CC_{2\pi}$; each edge in $\EE\GG$ is a curve in $\CC_{2\pi}$ that does not cross any other edge, except possibly at its endpoints; and each face in $\FF\GG$ is a connected component in $\CC_{2\pi}$ of the complement of the embedded graph $\GG$. Since the two marked vertices are mapped to $\pm \infty$, this implies that there is an infinite face at each end of the cylinder $\CC_{2\pi}$. In what follows, we use the convention to identify each vertex in $\VV\GG$ with its a priori embedding, i.e., if $x \in \VV\GG$ then we view $x$ as a point in $\CC_{2 \pi}$. Furthermore, for a set $K\subset \CC_{2\pi}$, we write
\begin{equation*}
 \VV\GG(K) := \l\{ x\in \VV\GG :  x\in  K \r\} .
\end{equation*}

\medskip
\noindent
We denote by $\smash{(\hat{\GG}, \hat{c})}$ the dual weighted planar graph associated to $\smash{(\GG, c)}$, where the conductance $\smash{\hat{c}_{\hat{e}}}$ of a dual edge $\smash{\hat{e}\ \in \EE\hat{\GG}}$ is equal to the resistance of the corresponding primal edge $e \in \EE\GG$, i.e., we set $\smash{\hat{c}_{\hat{e}} := 1/c_e}$. We assume that $\smash{(\hat{\GG}, \hat{c})}$ is \emph{properly embedded} in the infinite cylinder $\CC_{2\pi}$ in the sense of the following definition.
\begin{definition}
\label{def_proper_embedd_dual}
An embedding of the dual weighted planar graph $(\hat{\GG}, \hat{c})$ associated to $(\GG, c)$ in the infinite cylinder $\CC_{2\pi}$ is said to be \emph{proper} if:
\begin{enumerate}[(a)]
	\item every vertex of $\hat{\GG}$ is contained in a face of $\GG$;
	\item every edge $e$ of $\GG$ is crossed exactly at one point by a single edge $\hat{e}$ of $\hat{\GG}$ which joins the two faces incident to $e$.
\end{enumerate}
\end{definition}

\noindent
If an edge $e \in \EE\GG$ is oriented, the orientation of the corresponding dual edge $\hat{e} \in \EE\hat{\GG}$ can be obtained by rotating $e$ counter-clockwise. As for the primal graph, given a set $K \subset \CC_{2\pi}$, we write
\begin{equation*}
 \VV\hat{\GG}(K) := \l\{\hat{x} \in \VV\hat{\GG} :  \hat{x} \in  K \r\} .	
\end{equation*}

\subsubsection{Smith embedding}
We are now ready to provide a somewhat informal description of the Smith embedding of a given doubly marked finite weighted planar map $(\GG, c, v_0, v_1)$. The precise definition will be given in Section~\ref{sec_background_setup}. As mentioned earlier, the Smith embedding of a planar map was first introduced by Brooks, Smith, Stone, and Tutte in \cite{BSST40}, and later generalized to infinite planar graphs by Benjamini and Schramm in \cite{BS96}.

\medskip
\noindent
The Smith embedding of the quadruple $(\GG, c, v_0, v_1)$ is constructed by means of a tiling by rectangles of a finite cylinder $\CC_{\eta} := \R/\eta\Z \times [0, 1]$\label{pag_finite_cyl}, where $\eta$ is a positive number which depends on $(\GG,c)$. Each vertex $x \in \VV\GG$ is represented by a horizontal line segment $\sfH_x$, each edge $e \in \EE\GG$ by a rectangle $\sfR_e$, and each dual vertex $\hat{x} \in \VV\hat{\GG}$ by a vertical line segment $\sfV_{\hat{x}}$. In particular, since each edge $e \in \EE\GG$ corresponds to a rectangle in the tiling, we need to specify four coordinates for each edge. This is done by means of the voltage function $\frkh: \VV\GG \to [0, 1]$ and its discrete harmonic conjugate $\frkw: \VV\hat{\GG} \to \R/\eta\Z$. 

\medskip
\noindent
The function $\frkh $ is the unique function on $\VV\GG$ which is discrete harmonic on $\VV\GG\setminus\{v_0,v_1\}$ (with respect to the conductances $c$) with boundary conditions given by $\frkh(v_0) = 0$ and $\frkh(v_1) = 1$, i.e.,
\begin{equation*}
\frkh(x) = \P_x\l(\text{$X$ hits $v_1$ before $v_0$}\r), \quad \forall x \in \VV\GG,
\end{equation*}
where $\P_x$ denotes the law of the (weighted) random walk $X^x$ on $(\GG,c)$ started from $x$. We refer to Subsection~\ref{sub_rand_elect} for more details. 

\medskip
\noindent
The function $\frkw$ is the function on $\VV\hat{\GG}$ that satisfies the discrete Cauchy--Riemann equation associated to $\frkh$, i.e., it is the function on the set of dual vertices $\VV\hat{\GG}$ whose difference at the endpoints of each edge of $\hat{\GG}$ is equal to the difference of $\frkh$ at the endpoints of the corresponding primal edge times its conductance.  As we will see in Subsection~\ref{sub_dis_har_conj}, the function $\frkw$ is only defined modulo $\eta\Z$ and modulo an additive constant that can be fixed by imposing that $\frkw$ is equal to zero on a chosen dual vertex. In particular, the choice of the additive constant of $\frkw$ fixes the rotation of the cylinder in which the tiling takes place.

\medskip
\noindent
Now, we can specify the various objects involved in the definition of the Smith embedding. 
\begin{itemize}
	\item For each edge $e \in \EE\GG$, the rectangle $\sfR_e$ corresponds to the rectangle on $\CC_{\eta}$ such that the height coordinates of the top and bottom sides are given by the values of $\frkh$ at the endpoints of $e$, and the width coordinates of the left and right sides of $\sfR_e$ are given by the values of $\frkw$ at the endpoints of the corresponding dual edge $\hat{e}$. 
	\item For each vertex $x \in \VV\GG$, the horizontal segment $\sfH_x$ corresponds to the maximal horizontal segment which lies on the boundaries of all the rectangles corresponding to the edges incident to $x$.
	\item For each dual vertex $\hat{x} \in \VV\hat{\GG}$, the vertical segment $\sfV_{\hat{x}}$ corresponds to the maximal vertical segment which is tangent with all rectangles corresponding to primal edges surrounding $\hat{x}$.
\end{itemize}

\noindent
We call the map $\SS: \EE\GG \cup \VV\GG \cup \VV\hat{\GG} \to \CC_{\eta}$ such that $\SS(e) := \sfR_e$, $\SS(x) := \sfH_x$, and $\SS(\hat{x}) := \sfV_{\hat{x}}$ the \emph{tiling map} associated to the quadruple $(\GG, c, v_0, v_1)$. We refer to Figure~\ref{fig_smith_tiling} for a diagrammatic illustration of the tiling map associated to a given quadruple $(\GG, c, v_0, v_1)$. We define the Smith embedding $\dotSS$ associated to the quadruple $(\GG, c, v_0, v_1)$ as the function from $\VV\GG$ to $\R/\eta\Z \times [0, 1]$ given by
\begin{equation}
\label{eq_Smith_intro}
\dotSS(x) := \midpoint(\sfH_x), \quad \forall x \in \VV\GG,
\end{equation}
where $\midpoint(\sfH_x)$ corresponds to the middle point of the horizontal line segment $\sfH_x$\footnote{This definition of the Smith embedding is somewhat arbitrary. Indeed, for each $x \in \VV\GG$, one can define $\dotSS(x)$ to be any arbitrary point inside the horizontal segment $\sfH_x$.}. We refer to Subsection~\ref{sub_def_tiling_Smith} for precise definitions.

\begin{figure}[h]
	\centering
	\includegraphics[scale=1]{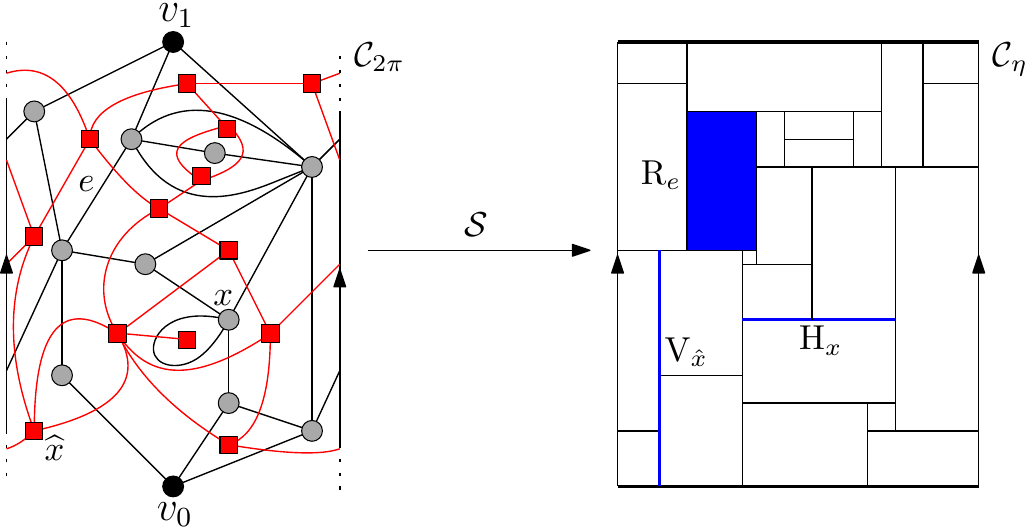}
	\caption[short form]{\small \textbf{Left:} A doubly marked finite weighted planar graph $(\GG, c, v_0, v_1)$, drawn in gray, properly embedded in the infinite cylinder $\CC_{2\pi}$ with the two marked vertices $v_0$ and $v_1$ drawn in black. Drawn in red is the corresponding weighted dual planar graph $(\hat{\GG}, \hat{c})$, which is also properly embedded in $\CC_{2\pi}$. \textbf{Right:} The Smith diagram associated to $\smash{(\GG, c, v_0, v_1)}$ constructed via the tiling map $\SS$. The blue horizontal segment $\sfH_x$ corresponds to the vertex $\smash{x \in \VV\GG}$. The blue vertical segment $\smash{\sfV_{\hat{x}}}$ corresponds to the dual vertex $\smash{\hat{x} \in \VV\hat{\GG}}$. The blue rectangle $\smash{\sfR_e}$ corresponds to the edge $\smash{e \in \EE\GG}$. In both the left and right figure, the two vertical lines with an arrow are identified with each other.}
	\label{fig_smith_tiling}
\end{figure}

\subsubsection{Assumptions and statement of the main result}
To state our main result, we need to consider a sequence of doubly marked finite weighted planar maps 
\begin{equation*}
\l\{(\GG^n, c^n, v_0^n, v_1^n)\r\}_{n \in \N},
\end{equation*}
and the sequence of associated weighted dual planar graphs $\{(\hat{\GG}^n, \hat{c}^n)\}_{n \in \N}$. We make the following assumptions.
\begin{enumerate}[start=1,label={(H\arabic*})]
\item \label{it_embedding} (\textbf{Cylindrical embedding}) For each $n \in \N$, the quadruple $(\GG^n, c^n, v_0^n, v_1^n)$ is properly embedded in the infinite cylinder $\CC_{2\pi}$ in the sense of Definition~\ref{def_proper_embedd}. Furthermore, the associated weighted dual planar graph $(\hat{\GG}^n, \hat{c}^n)$ is also properly embedded in $\CC_{2\pi}$ in the sense of Definition~\ref{def_proper_embedd_dual}.
\item \label{it_invariance} (\textbf{Invariance principle on the primal graphs}) 
For each $n \in \N$, view the embedded random walk on $(\GG^n, c^n)$, stopped when it hits either $v_0^n$ or $v_1^n$, as a continuous curve in $\CC_{2\pi}$ obtained by piecewise linear interpolation at constant speed. For each compact subset $K \subset \CC_{2\pi}$ and for any $z \in K$, the law of the random walk on $(\GG^n, c^n)$ started from the vertex $x_z^n \in \VV\GG^n$ nearest to $z$ weakly converges as $n \to \infty$ to the law of the Brownian motion on $\CC_{2\pi}$ started from $z$ with respect to the local topology on curves viewed modulo time parameterization specified in Subsection \ref{subsub_metric_CMP}, uniformly over all $z \in K$. 
\item \label{it_invariance_dual} (\textbf{Invariance principle on the dual graphs}) For each $n \in \N$, view the embedded random walk on $(\hat{\GG}^n, \hat{c}^n)$ as a continuous curve in $\CC_{2\pi}$ obtained by piecewise linear interpolation at constant speed. For each compact subset $K \subset \CC_{2\pi}$ and for any $z \in K$, the law of the random walk on $(\hat{\GG}^n, \hat{c}^n)$ started from the vertex $\hat{x}_z^n \in \VV\hat{\GG}^n$ nearest to $z$ weakly converges as $n \to \infty$ to the law of the Brownian motion on $\CC_{2\pi}$ started from $z$ with respect to the local topology on curves viewed modulo time parameterization specified in Subsection~\ref{subsub_metric_CMP}, uniformly over all $z \in K$. 
\end{enumerate}

\begin{remark}
It is natural to ask whether assumptions \ref{it_invariance} and \ref{it_invariance_dual} are actually equivalent, i.e., if one implies the other. We don't have any strong reason to either believe or dismiss the equivalence of these two assumptions. However, we emphasize that the local structure of the dual graph can differ significantly from the primal one. For example, the Tutte embedding of a graph and its dual can look very different, particularly in cases where one allows vertices of very high degree (which corresponds to having a very large face in the dual graph). Therefore, without any restrictions on the sequences of planar maps, it might be possible to construct examples where the walk and the dual walk behave very differently.
\end{remark}

\noindent
In what follows, given a point $x \in \CC_{2\pi}$, we write $\Re(x) \in [0, 2\pi)$ for its horizontal coordinate and $\Im(x) \in \R$ for its height coordinate. Similarly, if $x \in \CC_{\eta}$, then $\Re(x) \in [0, \eta)$ denotes its horizontal coordinate and $\Im(x) \in [0, 1]$ denotes its height coordinate. We are now ready to state our main theorem.

\begin{theorem}[Main theorem]
\label{th_main_1}
Consider a sequence $\smash{\{(\GG^n, c^n, v_0^n, v_1^n)\}_{n \in \N}}$ of doubly marked finite weighted planar maps, and let $\smash{\{(\hat{\GG}^n, \hat{c}^n)\}_{n \in \N}}$ be the corresponding sequence of dual weighted planar graphs. Assume that assumptions \ref{it_embedding}, \ref{it_invariance}, \ref{it_invariance_dual} are satisfied. For each $n \in \N$, let $\dotSS_n : \VV\GG^n \to \CC_{\eta_n}$ denote the Smith embedding associated with the quadruple $(\GG^n, c^n, v_0^n, v_1^n)$ as specified in \eqref{eq_Smith_intro}. There exist sequences $\smash{\{c^{\frkh}_n\}_{n \in \N}}$, $\smash{\{b^{\frkh}_n\}_{n \in \N}}$, $\smash{\{b^{\frkw}_n\}_{n \in \N} \subset \R}$ such that, if we define the affine transformation $\smash{T_n : \CC_{\eta_n} \to \CC_{2\pi}}$ by
\begin{equation*}
\Re(T_n x) := \left(\frac{2\pi}{\eta_n} \Re(x) + b_n^{\frkw}\right) \mod{2\pi} \quad \text{ and } \quad \Im(T_n x) := c_n^{\frkh} \Im(x) + b^{\frkh}_n,  \qquad \forall x \in \CC_{\eta_n}, 
\end{equation*}
then, for all compact sets $K \subset \CC_{2\pi}$, it holds that 
\begin{equation*}
\lim_{n \to \infty} \sup_{x \in \VV\GG^{n}(K)} \d_{2\pi}\l(T_n \dotSS_n(x), x\r) = 0,
\end{equation*}
where $\d_{2\pi}$ denotes the Euclidean distance on the cylinder $\CC_{2\pi}$. 
\end{theorem}

\begin{figure}[!h]
	\centering
	\includegraphics[scale=0.75]{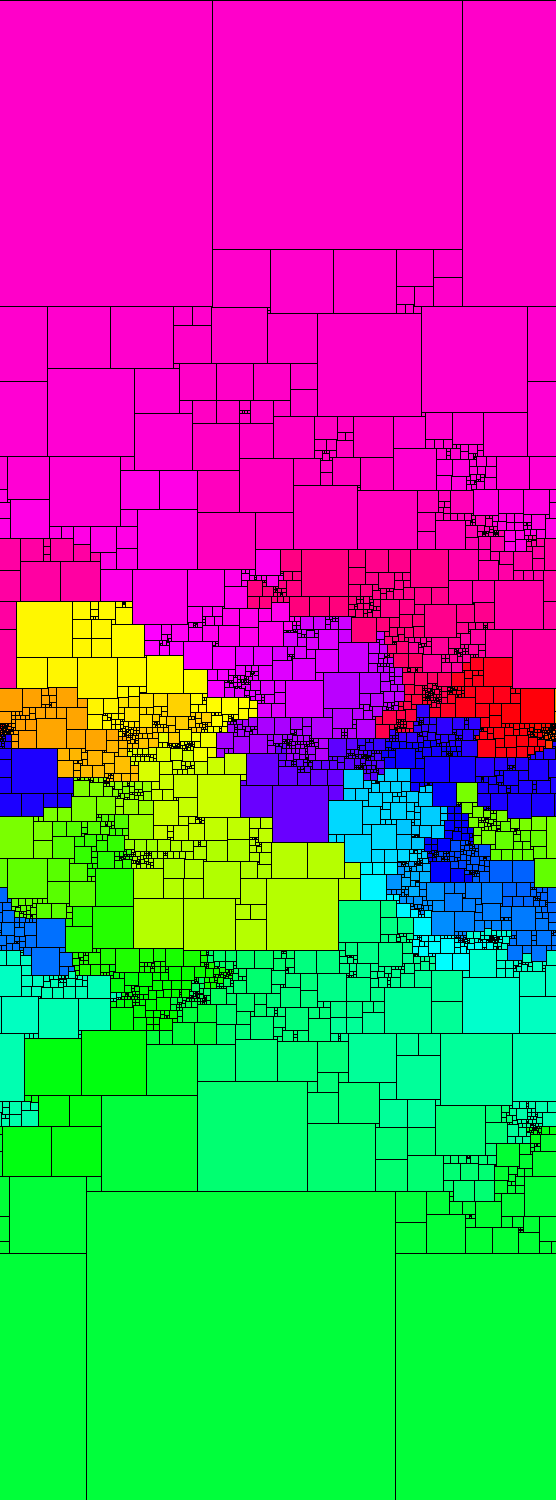} 
	\hspace*{3cm}
	\includegraphics[scale=0.75]{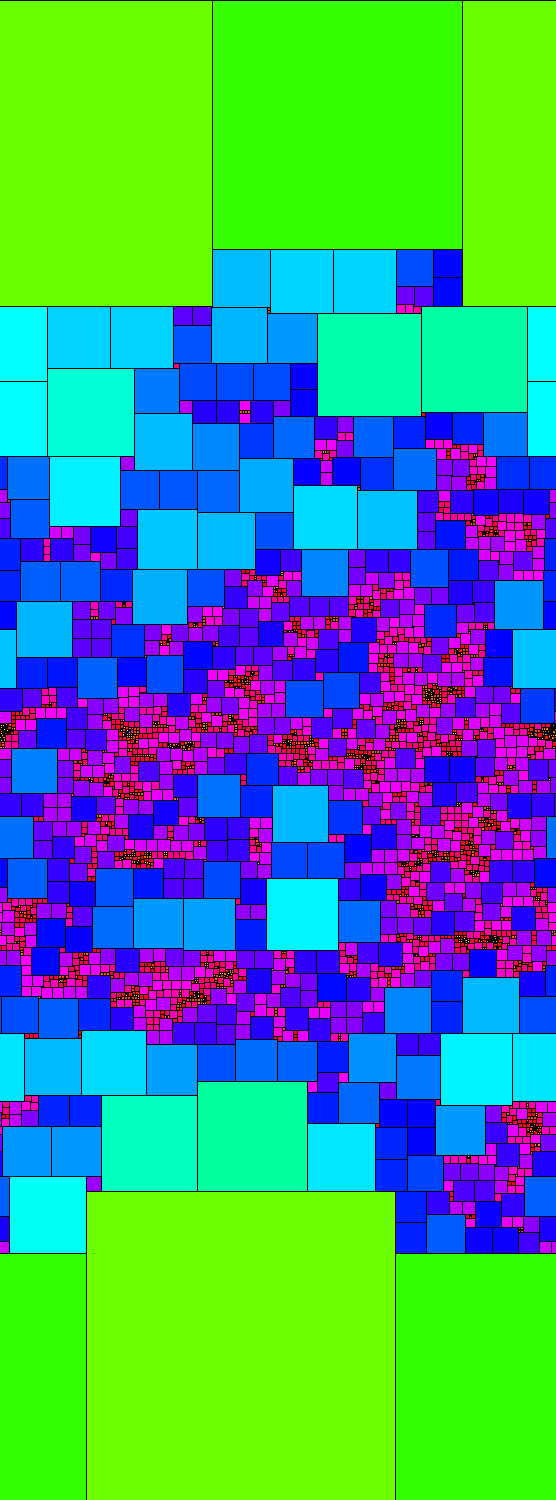}
	\caption[short form]{\small Consider a planar map $\MM$ with unit edge conductances and $n$ edges, and a distinguished spanning tree $\TT$. Then $(\MM,\TT)$ determines a quadrangulation $\GG$ with $n$ quadrilaterals obtained by replacing each edge of $\MM$ by a quadrilateral. Moreover, the path that ``snakes'' between $\TT$ and its dual crosses the edges of $\GG$ in order. Shown here is the Smith diagram of an instance of $\GG$ corresponding to a uniformly random $(\MM,\TT)$ pair with two random points chosen as roots. \textbf{Left:} The squares in the figure are colored according to their order with respect to the space-filling path that snakes between $\TT$ and its dual, i.e., based on their position in the cyclic ordering. \textbf{Right:} The squares in the figure are colored according to their Euclidean size. It is interesting to compare the figure on the right with the square subdivisions associated to the $\gamma$-LQG measure appearing, e.g., in \cite[Figures~1,~2,~3]{DS11}. The interested reader can find the Mathematica code used to generate the above simulations at the following link: \href{https://github.com/federico-bertacco/smith-embedding.git}{https://github.com/federico-bertacco/smith-embedding.git}. Furthermore, at the same link, there is also a PDF file available that provides further explanation on how these simulations were produced.}
	\label{fig_smith_sim}
\end{figure}

\medskip
\noindent
Theorem \ref{th_main_1} tells us that in order to say that the Smith embedding of $(\GG^n, c^n, v_0^n, v_1^n)$ is close to a given a priori embedding (up to translation and scaling), we only need to know a certain invariance principle for random walk under the a priori embedding. This result is in some ways not surprising, since it is natural to expect that if a simple random walk (and its dual) approximate Brownian motion, then discrete harmonic functions (and their conjugate duals) should approximate continuum harmonic functions. However, showing that this is actually true in the limit, and that the convergence is to the right continuum harmonic functions, will require some new coupling tricks, which we hope may prove useful in other settings as well. More precisely, the particular statement we obtain is far from obvious a priori, for two main reasons.
\begin{itemize}
\item Our hypotheses only concern the \emph{macroscopic} behavior of the random walk on $(\GG^n,c^n)$ in the \emph{bulk} of the cylinder. We do not need any hypotheses about how the random walk behaves when it gets close to the marked vertices $v_0^n$ and $v_1^n$. This may seem surprising at first glance since one could worry, e.g., that the structure of $(\GG^n,c^n)$ in small neighborhoods of $v_0^n$ and $v_1^n$ makes it much easier for random walk to hit $v_0^n$ than for it to hit $v_1^n$, and so the height coordinate function $\frkh^n$ is close to zero on all of $\VV\GG^n$. What allows us to get around this is the scaling and translation sequences $\smash{\{c^{\frkh}_n\}_{n \in \N}}$ and $\smash{\{b^{\frkh}_n\}_{n \in \N}}$. We refer to Subsection~\ref{sub_height} for more details. 
\item The width coordinate $\frkw^n$ is discrete harmonic but does not admit a simple direct description in terms of the random walk on $(\hat{\GG}^n,\hat c^n)$. For this reason, a fair amount of work is required to get from the invariance principle for this random walk to a convergence statement for $\frkw^n$. We refer to Subsection~\ref{sub_width} for more details.
\end{itemize}

\noindent
We remark that the Smith embedding can be very far from the identity near the ends of the cylinder: what is interesting, and perhaps surprising, is the generality in which we can show that the ``bad behavior'' gets ``smoothed out'' in the middle of the cylinder. This is apparent in the simulations presented in Figure~\ref{fig_smith_sim}.

\medskip
\noindent
As we will discuss in the next section, one application of Theorem~\ref{th_main_1} is the convergence of the mated-CRT map with the sphere topology to LQG under the Smith embedding. More generally, Theorem~\ref{th_main_1} reduces the problem of proving the convergence to LQG under the Smith embeddings for other types of random planar maps to the problem of finding some a priori embeddings of the map and its dual under which the counting measure on vertices converges to the LQG measure and the random walk on the map converges to Brownian motion modulo time parameterization.
 
%%%%%%%%%%%%%%%%%%%%%%%%%%%%%%%%%%%%%%%%%%
\subsection{Application to the mated-CRT map}
\label{sec_mated-CRT}
Mated-CRT maps are a one-parameter family of random planar maps constructed and studied, e.g., in \cite{GHS_mated, GMS_Harmonic, DMS21, GMS_Tutte}. The mated-CRT maps are parameterized by a real parameter $\gamma \in (0, 2)$ and are in the universality class of $\gamma$-LQG. 
In this paper, we will be interested in mated-CRT maps with the sphere topology. For each $n \in \N$ and $\gamma \in (0, 2)$, the $n$-\emph{mated-CRT map with the sphere topology} is the random planar triangulation $\GG^n$ with vertex set given by 
\begin{equation*}
	\VV\GG^n := \frac{1}{n} \Z \cap (0, 1],
\end{equation*} 
and an edge set defined by means of a condition involving a pair of linear Brownian motions. More precisely, consider a two dimensional Brownian motion $(L, R)$ with covariance matrix given by
\begin{equation}
\label{eq_var_cov_mated}
	\Var(L_t) = \Var(R_t) = |t|, \qquad \Cov(L_t, R_t) = -\cos\left(\frac{\pi \gamma^2}{4}\right)|t|,
\end{equation}
and conditioned to stay in the first quadrant for one unit of time and end up at $(0, 0)$, i.e., $(L, R)$ is an excursion.

\medskip
\noindent
Then, two vertices $x_1$, $x_2 \in \VV\GG^n$ are connected by an edge if and only if either
\begin{equation}
\label{eq_construction_mated_CRT}
\begin{alignedat}{1}
\max\left\{\inf_{t \in [x_1 - 1/n, x_1]} L_t, \inf_{t \in [x_2 - 1/n, x_2]} L_t \right\} &\leq \inf_{t \in [x_1, x_2 - 1/n]} L_t, \quad \text{ or }   \\
\max\left\{\inf_{t \in [x_1 - 1/n, x_1]} R_t, \inf_{t \in [x_2 - 1/n, x_2]} R_t \right\} &\leq \inf_{t \in [x_1, x_2 - 1/n]} R_t.
\end{alignedat}
\end{equation}

\noindent
The vertices in $\VV\GG^n$ are connected by two edges if $|x_1 - x_2| \neq 1/n$ and both the conditions in \eqref{eq_construction_mated_CRT} hold. We observe that the condition for $L$ in \eqref{eq_construction_mated_CRT} is equivalent to the existence of a horizontal line segment below the graph of $L$ whose end points are of the form $(t_1,L_{t_1})$ and $(t_2,L_{t_2})$ for $t_1 \in [x_1 - 1/n, x_1]$ and $t_2 \in [x_2 - 1/n,x_2]$, and similarly for $R$. This allows us to give an equivalent, more geometric, version of the definition of $\GG^n$. In particular, this procedure assigns a natural planar map structure to the mated-CRT map $\GG^n$, under which it is a triangulation. We refer to Figure~\ref{fig_mated} for a diagrammatic explanation of this procedure.

\begin{figure}[h]
\centering
\includegraphics[scale=0.92]{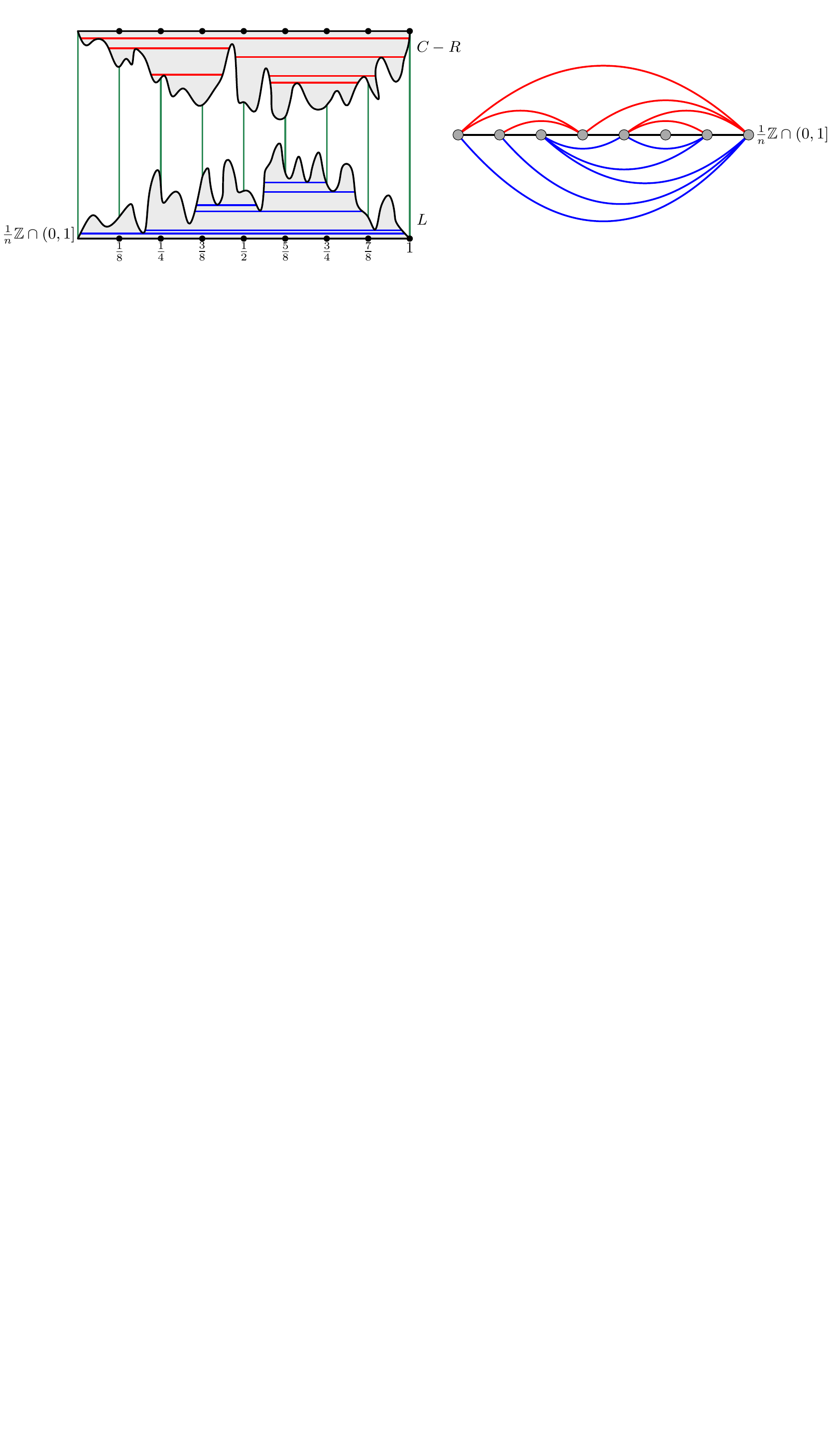}
\caption[short form]{\small \textbf{Left:} A diagram showing the construction of the mated-CRT map with sphere topology and $n = 8$ vertices. To geometrically construct the mated-CRT map $\GG^n$, we draw the graphs of $L$ and $C-R$ with a chosen large constant $C > 0$ to ensure that the graphs do not intersect. The region between the graphs is then divided into vertical strips. Each strip corresponds to the vertex $x \in \VV\GG^n$ which is the horizontal coordinate of its rightmost point. Two vertices $x_1$, $x_2 \in \VV\GG^n$ are connected by an edge if and only if their respective vertical strips are connected by a horizontal line segment that is either below the graph of $L$ or above the graph of $C-R$. For each pair of vertices for which the condition holds for $L$ (resp.\ $C-R$), we have drawn the lowest (resp.\ highest) segment which joins the corresponding vertical strips. We note that consecutive vertices are always connected by an edge. \textbf{Right:} The graph $\GG^n$ can be represented in the plane by connecting two vertices $x_1$, $x_2 \in \VV\GG^n$ with an arc below (resp.\ above) the real line if their vertical strips are connected by a horizontal segment below (resp.\ above) the graph of $L$ (resp.\ $C-R$). Additionally, each pair of consecutive vertices in $\VV\GG^n$ is connected by an edge. This representation gives $\GG^n$ a planar map structure under which it is a triangulation. A similar illustration was shown in \cite{GMS_Tutte}.}
\label{fig_mated}
\end{figure}

\medskip
\noindent
In \cite{GMS_Tutte}, Gwynne, Miller, and Sheffield proved that the Tutte embeddings of mated-CRT maps with the disk topology converge to $\gamma$-LQG. Thanks to our main theorem, we can prove an analogous result for the Smith embeddings of mated-CRT maps with the sphere topology. More precisely, for each $n \in \N$, pick two marked vertices $v_0^n$, $v_1^n \in \VV\GG^n$\footnote{We refer to Section~\ref{sec_mated} for an explanation of how these two marked vertices are selected.}. Then, we can conformally map the sphere into the infinite cylinder $\CC_{2\pi}$ so that the marked points are mapped to $\pm \infty$, and $(\GG^n, v_0^n, v_1^n)$\footnote{Here, each edge in $\EE\GG^n$ has unit conductance and so we do not specify the sequence of weights $c$ as in the general case.} is properly embedded in $\CC_{2\pi}$.

\begin{theorem}[Convergence of mated-CRT map]
\label{th_main_2}
Fix $\gamma \in (0, 2)$ and let $\{(\GG^n, v_0^n, v_1^n)\}_{n \in \N}$ be the sequence of doubly marked $n$-mated CRT map with the sphere topology embedded in $\CC_{2\pi}$ as specified above. There exists a sequence of random affine transformations $\{T_n\}_{n \in \N}$ from $\CC_{\eta_n}$ to $\CC_{2\pi}$ of the form specified in the statement of Theorem~\ref{th_main_1} such that, if we let $\mu_n$ be the push-forward with respect to the mapping $z \mapsto T_n z$ of the counting measure on the set $\dotSS_n(\VV\GG^n)$ scaled by $1/n$, then we have the following convergences in probability as $n \to \infty$.
\begin{enumerate}[(a)]
	\item \label{mated_a} On each compact subset of $\CC_{2\pi}$, the measure $\mu_n$ weakly converges to the $\gamma$-LQG measure associated to a doubly marked unit area quantum sphere parameterized by $\CC_{2 \pi}$ in such a way that its marked points are at $\pm\infty$, as defined in \cite[Definition~4.21]{DMS21}.
	\item \label{mated_b} On each compact subset of $\CC_{2\pi}$, the image under the mapping $z \mapsto T_n z$ of the space-filling path on the Smith embedded mated-CRT map on $\CC_{\eta_n}$ obtained from the left-right ordering of the vertices converges uniformly with respect to the two-point compactification topology on the cylinder to space-filling SLE$_\kappa$ on $\CC_{2\pi}$, with $\kappa = 16/\gamma^2$, parameterized by $\gamma$-LQG mass. 
	\item \label{mated_c} For $z \in \CC_{2\pi}$, let $x_z^n \in \VV\GG^n$ be the vertex nearest to $z$. The conditional law given $\GG^n$ of the image under the mapping $z \mapsto T_n z$ of the simple random walk on the Smith-embedded mated-CRT map started from $\dotSS^n(x_z^n)$ and stopped when it hits one of the horizontal segments associated to the marked vertices weakly converges to the law of Brownian motion on $\CC_{2 \pi}$ started from $z$, modulo time parameterization and uniformly over all $z$ in a compact subset of $\CC_{2\pi}$.
\end{enumerate}
\end{theorem}

\noindent
To conclude, let us point out that item~\ref{mated_a} of Theorem~\ref{th_main_2} solves \cite[Question~1]{Berry_square} for the case of mated-CRT maps.

%%%%%%%%%%%%%%%%%%%%%%%%%%%%%%%%%%%%%%%%%%
\subsection{Outline}
\label{sec_outline}
Most of the paper is dedicated to proving Theorem~\ref{th_main_1}, and it is organized as follows. In the first part of Section~\ref{sec_background_setup}, we provide some background material on weighted planar graphs and the theory of electrical networks. We then move on to the precise construction of the tiling map and the definition of the Smith embedding in Subsection~\ref{sub_def_tiling_Smith}. As mentioned earlier, this is achieved by introducing two harmonic maps: one on the planar map itself and one on the associated dual planar map.

\medskip
\noindent
In Section~\ref{sec_prop_Smith}, we establish and prove several properties of the Smith embedding. The most significant result in this section is Lemma~\ref{lm_winding_embed_walk}, which heuristically asserts that the conditional expected horizontal winding of a Smith-embedded random walk, given the vertical coordinate, is zero. This property is crucial for proving our main theorem. To prove this intermediate result, we rely on Lemma~\ref{lm_hitting_hor}, which essentially states that the conditional probability, given the vertical coordinate, that the Smith-embedded random walk hits a certain horizontal line segment is proportional to the segment's width. We note that a similar result, though without conditioning on the vertical component, was previously obtained by Georgakopoulos \cite[Lemma~6.2]{Geo16} in the context of infinite weighted planar graphs.

\medskip
\noindent
Section~\ref{sec_assumption_main} forms the core of this article and contains the proof of Theorem~\ref{th_main_1}, which is divided into two main parts. In Subsection~\ref{sub_height}, we examine the height coordinate function, and in Subsection~\ref{sub_width}, we focus on the width coordinate function.
Specifically, the main result of Subsection~\ref{sub_height} is Proposition~\ref{pr_main_height}, which roughly states that the height coordinate of the a priori embedding is asymptotically close to an affine transformation of the height coordinate of the Smith embedding. Similarly, the main result of Subsection~\ref{sub_width} is Proposition~\ref{pr_main_witdh}, which states the analogous fact for the width coordinate. We refer to Subsections~\ref{sub_height}~and~\ref{sub_width} for the proof outlines of the height and width coordinate results, respectively. Finally, in Subsection~\ref{sub_proof_main}, we show how to combine the results for the height coordinate and width coordinate to prove Theorem~\ref{th_main_1}. 

\medskip
\noindent
In Section~\ref{sec_mated}, we provide a brief introduction to the relationship between mated-CRT maps and LQG. We then prove in Subsection~\ref{sub_mated_assumptions} that this family of random planar maps satisfies the assumptions of our main result. Specifically,  in Subsection~\ref{sub_conv_LQG}, we apply our result to show that the scaling limit of mated-CRT maps is $\gamma$-LQG, thereby proving Theorem~\ref{th_main_2}.

\begin{acknowledgements}
F.B.\ is grateful to the Royal Society for financial support through Prof.\ M.~Hairer's research professorship grant RP\textbackslash R1\textbackslash 191065. 
E.G.\ was partially supported by a Clay research fellowship.
Part of this work was carried out during the \textit{Probability and Mathematical Physics} ICM satellite conference at Helsinki in Summer 2022. We thank the organizers of this conference for their hospitality.
\end{acknowledgements}
%%%%%%%%%%%%%%%%%%%%%%%%%%%%%%%%%%%%%%%%%%

%%%%%%%%%%%%%%%%%%%%%%%%%%%%%%%%%%%%%%%%%%
\section{Background and setup}
\label{sec_background_setup}

%%%%%%%%%%%%%%%%%%%%%%%%%%%%%%%%%%%%%%%%%%
\subsection{Basic definitions}

%%%%%%%%%%%%%%%%%%%%%%%%%%%%%%%%%%%%%%%%%%
\subsubsection{Basic notation}
We write $\N$ for the set of positive integers and $\N_0$ for the set of non-negative integers. Given $n \in \N$ and $j \in \N_0$, we let $[n] := \{1, \ldots, n\}$ and $[n]_j := \{j, \ldots, n\}$, Furthermore, for $n \in \N$ and $j \in \N_0$, we write $[x_n]$ to denote the collection of objects $\{x_1, \ldots, x_n\}$ and $[x_n]_j$ to denote the collection of objects $\{x_j, \ldots, x_n\}$. If $a$ and $b$ are two real numbers, we write $a \lesssim b$ if there is a constant $C > 0$, independent of the values of $a$ or $b$ and certain other parameters of interest, such that $a \leq C b$, and we highlight the dependence of the implicit constants when necessary. If $a$ and $b$ are two real numbers depending on a variable $x$, we write $a = o_x(b)$ if $a/b$ tends to $0$ as $x \to \infty$. We use the convention to identify $\R^2$ with $\C$. In particular, given a point $x \in \R^2$, we write $\Re(x)$ (resp. $\Im(x)$) for its horizontal (resp. vertical) coordinate.

\subsubsection{Metric on curves modulo time parameterization}
\label{subsub_metric_CMP}
For $T_1$, $T_2 > 0$, let $P_1 : [0,T_1] \to \R^2$ and $P_2 : [0,T_2] \to \R^2$ be two continuous curves defined on possibly different time intervals. We define 
 \begin{equation}
 \label{eq_metric_CMP}
 \dCMP\l(P_1,P_2\r) := \inf_{\phi } \sup_{t\in [0,T_1]} \l|P_1(t) - P_2(\phi(t))\r|,
 \end{equation}
where the infimum is taken over all increasing homeomorphisms $\phi : [0,T_1]  \to [0,T_2]$. It is known that $\dCMP$ induces a complete metric on the set of curves viewed modulo time parameterization (see\ \cite[Lemma~2.1]{AB99}). 

\medskip
\noindent
For curves defined for infinite time, it is convenient to have a local variant of the metric $\dCMP$. Assume that $P_1 : [0,\infty) \to \R^2$ and $P_2 : [0,\infty) \to \R^2$ are two such curves. Then, for $r > 0$, let $T_{1,r}$ (resp.\ $T_{2,r}$) be the first exit time of $P_1$ (resp.\ $P_2$) from the ball $B(0, r)$ centred at $0$ with radius $r$, or $0$ if the curve starts outside $B(0, r)$. 
We define 
\begin{equation}
\label{eq_metric_CMP_loc}
\dCMPloc\l(P_1,P_2\r) := \int_1^\infty e^{-r} \l(1 \wedge \dCMP\left(P_1|_{[0,T_{1,r}]} , P_2|_{[0,T_{2,r}]} \r) \right)\d r.
\end{equation}
Moreover, we observe that given a sequence $\{P_n\}_{n \in \N}$ of continuous curves defined for infinite time, then $\lim_{n \to \infty} \dCMPloc(P_n , P) = 0$, if and only if, for Lebesgue almost every $r > 0$, $P_n$ stopped at its first exit time from $B(0, r)$ converges to $P$ stopped at its first exit time from $B(0, r)$ with respect to the metric~\eqref{eq_metric_CMP}.

\begin{remark}
In the remaining part of the article, we also need to consider curves taking values in the infinite cylinder $\CC_{2 \pi}$. We equip the spaces specified above, but with $\CC_{2\pi}$ in place of $\R^2$, with the same metrics. It will be clear from the context whether the metric under consideration refers to curves in $\R^2$ or in $\CC_{2\pi}$.
\end{remark}

\subsubsection{Graph notation}
Given a finite planar graph $\GG$, besides the notation related to $\GG$ specified in the introduction, we need to introduce some further nomenclature. In particular, in what follows, we use $e \in \EE\GG$ to denote both oriented and unoriented edges. An oriented edge $e \in \EE\GG$ is oriented from its tail $e^{-}$ to its head $e^{+}$. Furthermore, given a vertex $x \in \VV\GG$, we write $\VV\GG(x)$ for the set of vertices $y$ adjacent to $x$, i.e., such that there exists an edge connecting $x$ to $y$. For a vertex $x \in \VV\GG$, we denote by $\EE\GG(x)$ the set of edges in $\EE\GG$ incident to $x$. For a fixed orientation of the edges in $\EE\GG(x)$, we let $\EE\GG^{\downarrow}(x)$ (resp.\ $\EE\GG^{\uparrow}(x)$) be the set of edges in $\EE\GG(x)$ with heads (resp.\ tails) equal to $x$. Similar notation will also be used for the dual planar graph $\hat{\GG}$.

\paragraph{Metric graph.}
We will need to consider the metric space associated to a planar graph $\GG$ which can be canonically built as follows. For each edge $e \in \EE\GG$, we choose an arbitrary orientation of $e$ and we let $I_e$ be an isometric copy of the real unit interval $[0, 1]$. We define the metric space $\G$ associated with $\GG$ to be the quotient of $\cup_{e \in \EE\GG} I_e$ where we identify the endpoints of $I_e$ with the vertices $e^{-}$ and $e^{+}$, and we equip it with the natural path metric $\d^{\G}$. More precisely, for two points $x$, $y$ lying on an edge of $\GG$, we define $\d^{\G}(x, y)$ to be the Euclidean distance between $x$ and $y$. For points $x$, $y$ lying on different edges, we use the metric given by the length of the shortest path between the two points, where distances are measured along the edges using the Euclidean distance. We can also define the dual metric graph $\smash{\hat{\G}}$, and the associated metric $\smash{\d^{\hat{\G}}}$, in a similar way. 

%%%%%%%%%%%%%%%%%%%%%%%%%%%%%%%%%%%%%%%%%%
\subsection{Universal cover}
\label{sub_universal_cover}
The concept of universal cover of a graph will play an important role in our analysis. If $\GG$ is a graph embedded in the infinite cylinder $\CC_{2\pi}$, then there is a canonical way to define its lift $\GG^{\dag}$ to the universal covering space of $\CC_{2\pi}$. More precisely, consider the universal cover $(\R^2, \sigma_{2\pi})$ of $\CC_{2\pi}$, where the covering map $\sigma_{2\pi} : \R^2 \to \CC_{2\pi}$ is defined by
\begin{equation}
\label{eq_covering_map}
\sigma_{2 \pi}(t, x) := \l(e^{i t}, x\r) , \quad \forall (t, x) \in \R^2.
\end{equation}
Then, the lifted graph $\GG^{\dag}$ can be constructed by taking every lift of every vertex and every edge of $\GG$ in $\CC_{2\pi}$ to the covering space $\R^2$. We denote by $\VV\GG^{\dag}$ and $\EE\GG^{\dag}$ the set of vertices and edges of the lifted graph $\GG^{\dag}$, respectively. Moreover, we can also construct the lift of the dual graph $\hat{\GG}$ to the universal covering space $\R^2$ in a similar way, and we denote it by $\hat{\GG}^{\dag}$. We adopt the following notational convention: if $x \in \VV\GG$ is a vertex, then we denote by $\x \in \VV\GG^{\dag}$ a lift of $x$; if $e \in \EE\GG$ is an edge, then we denote by $\e \in \EE\GG^{\dag}$ a lift of $e$; if $\hat{x} \in \VV\hat{\GG}$ is a dual vertex, then we denote by $\hx \in \VV\hat{\GG}^{\dag}$ a lift of $\hat{x}$.

\begin{figure}[h]
\centering
\includegraphics[scale=0.85]{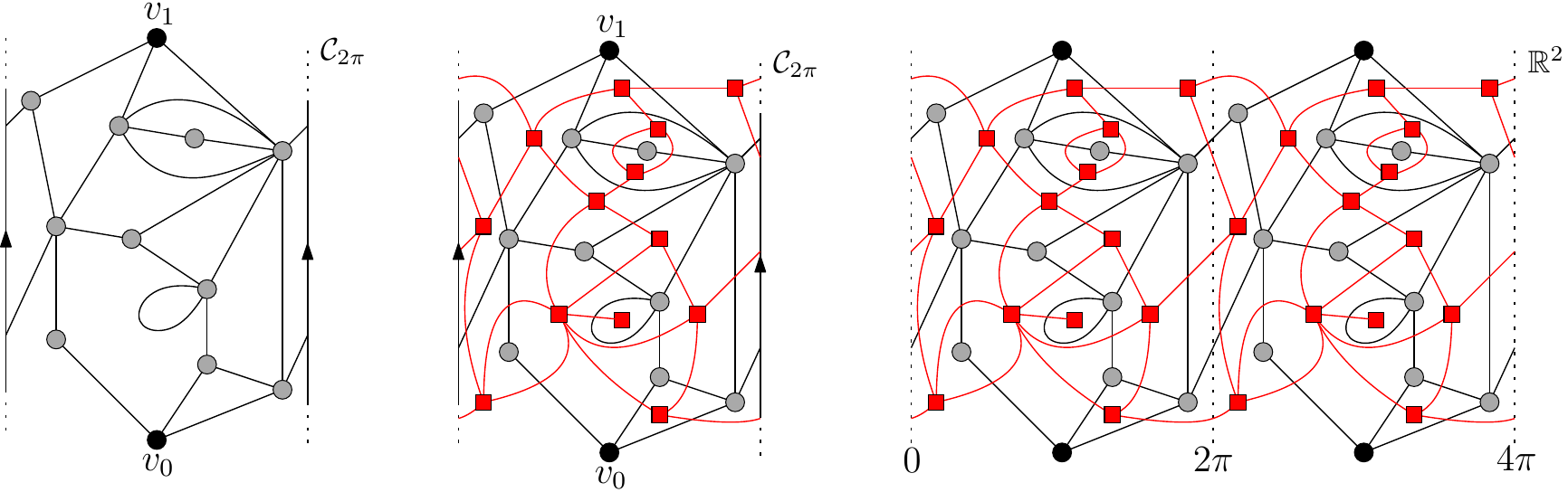}
\caption[short form]{\small \textbf{Left:} A doubly marked finite weighted planar graph $(\GG, c, v_0, v_1)$, drawn in gray, properly embedded in the infinite cylinder $\CC_{2\pi}$ with the two marked vertices $v_0$ and $v_1$ drawn in black. \textbf{Middle:} The same doubly marked finite weighted planar graph as in the left figure together with its dual planar weighted graph $\smash{(\hat{\GG}, \hat{c})}$, drawn in red, properly embedded in $\CC_{2\pi}$. \textbf{Right:} A portion of the lifted weighted graph $\smash{(\GG^{\dag}, c^{\dag})}$, drawn in grey, and the associated dual lifted graph $\smash{(\hat{\GG}^{\dag}, \hat{c}^{\dag})}$, drawn in red. In both the left and middle figure, the two vertical lines with an arrow are identified with each other.}
\end{figure}

\noindent
Moreover, if $(\GG, c)$ is a finite weighted planar graph embedded in $\CC_{2\pi}$, we can naturally assign to each lifted edge $\e$ the conductance $c^{\dag}_{\e} := c_{e}$, and we denote by $(\GG^{\dag}, c^{\dag})$ the lifted weighted graph. By definition, the lifted graph $\GG^{\dag}$ is periodic in the sense that if $\x_1$, $\x_2 \in \VV\GG^{\dag}$ are two points in $\R^2$ such that $\Im(\x_1) = \Im(\x_2)$ and $|\Re(\x_1) - \Re(\x_2)| \in \N_0$, then $\sigma_{2\pi}(\x_1) = \sigma_{2\pi}(\x_2)$. Finally, for a set $K \subset \R^2$, we write
\begin{equation*}
\VV\GG^{\dag}(K) := \l\{\x \in \VV\GG^{\dag} \, : \, \x \in K\r\}, \qquad \VV\hat{\GG}^{\dag}(K) := \l\{\hx \in \VV\GG^{\dag} \, : \, \hx \in K\r\}.
\end{equation*}

\medskip
\noindent
Before proceeding, we recall the following simple result. An oriented path in $\GG$ is a collection of oriented edges $e_1 \cdots e_n$ in $\EE\GG$ such that $\smash{e_{j}^{+} = e_{j+1}^{-}}$, for all $\smash{j \in [n-1]}$. Furthermore, if also $\smash{e_n^{+} = e_1^{-}}$, then $e_1 \cdots e_n$ is called an oriented loop.
\begin{lemma} 
\label{lm_path_lifting}
Let $e_1 \cdots e_n$ be an oriented path in $\GG$. Let $\e_1$ be a lift of $e_1$ to the lifted graph $\GG^{\dag}$, then there exists a unique path $\e_1 \cdots \e_n$ in $\GG^{\dag}$ such that $\e_{j}$ is a lift of $e_j$, for all $j \in [n]_2$.
\end{lemma}

\noindent
The main advantage of working in the universal cover of the cylinder is that we can keep track of the winding of paths.  
\begin{definition}
Let $0 \leq t_1 < t_2$, consider a path $P: [t_1, t_2] \to \CC_{2 \pi}$, and let $\PP:[t_1, t_2]\to\R^2$ be a lift of $P$ to the universal cover. We define the winding of $P$ by letting
\begin{equation*}
\wind_{2 \pi}(P) := \frac{\Re(\PP(t_2)) -  \Re(\PP(t_1))}{2 \pi} .
\end{equation*}
We say that $P$ winds around the cylinder if $|\wind_{2 \pi}(P)| \geq 1$. We say that $P$ does a noncontractible loop around the cylinder if there exist times $t_1 \leq s_1 < s_2 \leq t_2$ such that $P|_{[s_1, s_2]}$ winds around the cylinder and $P(s_1) = P (s_2)$.
\end{definition}

%%%%%%%%%%%%%%%%%%%%%%%%%%%%%%%%%%%%%%%%%%
\subsection{Random walks and electrical networks}
\label{sub_rand_elect}
In this subsection, we briefly recall the main concepts in the theory of electrical networks and we refer to \cite{LP16, Nac20} for a complete introduction. Let $(\GG, c, v_0, v_1)$ be a doubly marked finite weighted planar graph properly embedded in the infinite cylinder $\CC_{2\pi}$ in the sense of Definition~\ref{def_proper_embedd} . The conductance of a vertex $x \in \VV\GG$ is denoted by $\pi(x)$ and it is defined to be the sum of the conductances of all the edges incident to $x$, i.e.,
\begin{equation*}
\pi(x) := \sum_{e \in \EE\GG(x)} c_{e} , \quad \forall x \in \VV\GG .
\end{equation*}

\paragraph{Random walk.}
The random walk on $(\GG, c)$ is the discrete time Markov chain $X = \{X_n\}_{n \in \N_0}$ with state space $\VV\GG$ such that, for all $n \in \N_0$,
\begin{equation*}
\P\l(X_{n+1} = y \mid X_n = x\r) = \begin{cases}
c_{xy}/\pi(x) , & y \in \VV\GG(x),  \\
0 , & \text{otherwise}.
\end{cases}
\end{equation*}
Given a vertex $x \in \VV\GG$, we write $\P_x$ and $\E_x$ for the law and expectation of $X$ started from $x$. Moreover, we may write $X^x$ in order to emphasize that the random walk $X$ is started from the vertex $x \in \VV\GG$. With a slight abuse of notation, we will also denote with $X = \{X_t\}_{t\geq0}$ the continuous time version of the random walk, where the continuous path is generated by piecewise linear interpolation at constant speed. If the conductance on every edge of the graph is equal to one, we call the random walk in this case simple random walk.

\medskip
\noindent
We emphasize that, given a random walk $X$ on $(\GG, c)$, we can canonically lift it to the lifted weighted planar graph $(\GG^{\dag}, c^{\dag})$, and we denote the resulting walk by $\X$. If $X^{x}$ is started from a point $x \in \VV\GG$, then we need to specify the lift $\x \in \sigma_{2\pi}^{-1}(x)$ of $x$ from which the lifted walk $\X^{\x}$ is started from. Similar notation will be also adopted for the random walk on the dual graph.

\paragraph{Estimate on the total variation distance.}
We now state and prove an elementary lemma for general weighted planar graphs which allows to compare the total variation distance of the exit positions from a set for two random walks started from two distinct points.
\begin{lemma}
\label{lm_tot_variation}
Let $(\GG, c)$ be a finite weighted planar graph and let $W \subset \VV\GG$. For $x \in \VV\GG$, let $X^x$ be the random walk on $(\GG, c)$ started from $x$ and let $\tau_x$ be the first time that $X^x$ hits $W$. Then, for $x$, $y \in \VV\GG \setminus W$, it holds that
\begin{equation*} 	
\dTV\l(X^x_{\tau_{x}}, X^y_{\tau_{y}}\r) \leq \P\l(X^x|_{[0, \tau_x]} \text{ does not disconnect } y \text{ from } W\r),
\end{equation*}
where $\dTV$ denotes the total variation distance.
\end{lemma}
\begin{proof}
The proof is a variant of \cite[Lemma~3.12]{GMS22} with the difference that one should consider a weighted spanning tree instead of a uniform spanning tree of the finite weighted planar graph $(\GG, c)$. For the reader's convenience, we gather here a proof. The lemma is a consequence of Wilson's algorithm. Consider the weighted spanning tree $\TT$ of the finite weighted planar graph $(\GG, c)$, where all vertices of $W$ are wired to a single point. We recall that the weighted spanning tree $\TT$ is chosen randomly from among all the spanning trees with probability proportional to the product of the conductances along the edges of the tree. For $x\in \VV\GG$, let $L^x$ be the unique path in $\TT$ from $x$ to $W$.  For a path $P$ in $\GG$, write $\operatorname{LE}(P) $ for its chronological loop erasure. By Wilson's algorithm (see \cite[Theorem~4.1]{LP16}), we can generate the union $L^x\cup L^y$ by using the following procedure. 
\begin{enumerate}[(a)]
\item Run $X^y$ until time $\tau_y$ and generate the loop erasure $ \operatorname{LE}(X^y|_{[0,\tau^y]})$.
\item Conditional on $X^y|_{[0,\tau_y]}$, run $X^x$ until the first time $\tilde{\tau}_x$ that it hits either $\operatorname{LE}(X^y|_{[0,\tau_y]})$ or $W$. 
\item Set $L^x \cup L^y = \operatorname{LE}(X_y|_{[0,\tau_y]}) \cup \operatorname{LE}(X^x|_{[0,\tilde{\tau}_x]})$. 
\end{enumerate}
Note that $L_y =  \operatorname{LE}(X^y|_{[0,\tau_y]})$ in the above procedure. Interchanging the roles of $x$ and $y$ in the above procedure shows that $L^x$ and $\operatorname{LE}(X^x|_{[0,\tau_x]})$ have the same distribution. When constructing $L^x\cup L^y$ as described above, the points at which $L^x$ and $L^y$ hit $W$ coincide if $X^x$ hits $\operatorname{LE}(X^y|_{[0,\tau_y]})$ before reaching $W$. In particular, this occurs when $X^x|_{[0, \tau_x]}$ disconnects $y$ from $W$. Thus, there is a coupling between $\operatorname{LE}(X^x|_{[0,\tau_x]})$ and $\operatorname{LE}(X^y|_{[0,\tau_y]})$, where the probability that these two loop erasures hit $W$ at the same point is at least $\P_x(X|_{[0, \tau_x]} \text{ disconnects } y \text{ from } W)$. Now, by observing that $X^x_{\tau_x}$ corresponds to the point at which $\operatorname{LE}(X^x|_{[0,\tau_x]})$ first hit $W$, and similarly for $y$ in place of $x$, we obtain the desired result.
\end{proof}

\paragraph{Electrical network.} 
There is an extremely useful correspondence between random walks and Kirchhoff's theory of electric networks. Let $(\GG, c, v_0, v_1)$ be as above and suppose that every edge $e \in \EE\GG$ is made of conducting wires with conductance equals to $c_e$. Connect a battery between $v_1$ and $v_0$ so that the voltage at $v_1$ is equal to one and the voltage at $v_0$ is equal to zero. Then certain currents will flow along the edges of the graph establishing the voltage at each vertex $x \in \VV\GG \setminus \{v_0, v_1\}$. An immediate consequence of physical laws is that the voltage function is harmonic on $\VV\GG$ except at $v_0$ and $v_1$. More formally, we have the following definition.
\begin{definition}
\label{def_voltage_fct}
The voltage function associated to the quadruple $(\GG, c, v_0, v_1)$ is the unique function $\frkh: \VV\GG \to [0,1]$ such that $\frkh(v_0) = 0$, $\frkh(v_1) = 1$, and 
\begin{equation*}
\frkh(x) = \frac{1}{\pi(x)} \sum_{y \in \VV\GG(x)} c_{xy} \frkh(y) , \quad \forall x \in \VV\GG \setminus\{v_0, v_1\}.
\end{equation*}
In view of the role that $\frkh$ will play in the construction of the Smith embedding, we will also call $\frkh$ the \emph{height coordinate function}.
\end{definition}

\noindent
Given an edge $e \in \EE\GG$, we say that $e$ is \emph{harmonically oriented} if $\frkh(e^+) \geq \frkh(e^{-})$. In what follows, unless otherwise specified, we always consider the harmonic orientation of the edges in $\EE\GG$. It is a remarkable fact that the voltage function $\frkh$ admits a representation in terms of a random walk $X$ on $(\GG, c)$. More precisely, if for all $v \in \VV\GG$, we define $\tau_v$ to be the first hitting time of $v$ for $X$, then one can easily check that  
\begin{equation*}
\frkh(x) = \P_x\l(\tau_{v_1} < \tau_{v_0}\r), \quad \forall x \in \VV\GG ,
\end{equation*}
Moreover, since the voltage function $\frkh$ is harmonic on $\VV\GG \setminus \{v_0, v_1\}$, if $X^x$ is a random walk on $(\GG, c)$ started from $x \in \VV\GG$ and killed upon reaching the set $\{v_0, v_1\}$, then the process $\frkh(X^x)$ is a martingale with respect to the filtration generated by $X^x$. 

\begin{remark}
We note that we can canonically lift the voltage function $\frkh$ to the lifted weighted graph $(\GG^{\dag}, c^{\dag})$ by setting $\frkh^{\dag} : \VV\GG^{\dag} \to [0, 1]$ as follows
\begin{equation}
\label{eq_harmo_lift}
\frkh^{\dag}(\x) : = \frkh(\sigma_{2 \pi}(\x)) , \quad \forall \x \in \VV\GG^{\dag}.
\end{equation} 
\end{remark}

\begin{remark}
\label{rm_extened_harm}
Note that we can naturally extend the definition of the voltage function to a function on the metric graph $\G$ associated to $\GG$, i.e., we can define the function $\frkh : \G \to [0, 1]$. This extension can be done by linearly interpolating the values at the endpoints of every edge $e \in \EE\GG$. More precisely, if $x \in \G$ is a point lying on the harmonically oriented edge $e \in \EE\GG$, then we set 
\begin{equation*}
\frkh(x):=(\frkh(e^{+})-\frkh(e^{-}))\d^{\G}(e^{-}, x) + \frkh(e^{-}) .
\end{equation*}
\end{remark}

\paragraph{The flow induced by the voltage function.}
We finish this subsection by introducing the flow across oriented edges induced by the voltage function $\frkh$. We denote this flow by $\nabla \frkh: \EE\GG \to \R$ and we define it as follows
\begin{equation*}
\nabla \frkh(e) := c_e \l(\frkh(e^{+}) - \frkh(e^{-})\r), \quad \forall e \in \EE\GG . 
\end{equation*}
The flow $\nabla \frkh$ satisfies the following well-known properties (see \cite[Section~2.2]{Nac20}).
\begin{enumerate}[(a)]
	\item (\emph{Antisymmetry}) For every oriented edge $e\in\EE\GG$, it holds that 
	\begin{equation*}
		\nabla \frkh(-e) = - \nabla\frkh(e),
	\end{equation*}
	where $-e$ stands for the edge $e$ endowed with opposite orientation.
	\item (\emph{Kirchhoff's node law}) For all $x \in \VV\GG \setminus \{v_0, v_1\}$, it holds that
	\begin{equation}
	\label{eq_kirch_node}	
	\sum_{e \in \EE\GG(x)} \nabla \frkh(e) = 0,
	\end{equation}
	where here the orientation of each $e \in \EE\GG(x)$ is fixed by letting $e^{-} = x$.
	\item (\emph{Kirchhoff's cycle law}) For every directed cycle $e_1 \cdots e_n$, it holds that
 	\begin{equation}
	\label{eq_kirch_cycle}
	\sum_{i = 1}^{n} \frac{1}{c_{e_i}} \nabla \frkh(e_i) = 0 .
	\end{equation}
\end{enumerate}

\noindent
We denote the strength of the flow $\nabla \frkh$ induced by $\frkh$ by setting
\begin{equation}
\label{eq_flow_strength}
\eta := \sum_{e \in \EE\GG^{\uparrow}(v_0)} \nabla \frkh (e),
\end{equation}
where $\smash{\EE\GG^{\uparrow}(v_0)}$ denotes the set of harmonically oriented edges in $\EE\GG$ with tails equal to $v_0$. Furthermore, thanks to the harmonicity of $\frkh$, a simple computation yields that $\smash{\eta = \sum_{e \in \EE\GG^{\downarrow}(v_1)} \nabla \frkh(e)}$, where $\smash{\EE\GG^{\downarrow}(v_1)}$ denotes the set of harmonically oriented edges in $\EE\GG$ with heads equal to $v_1$.

%%%%%%%%%%%%%%%%%%%%%%%%%%%%%%%%%%%%%%%%%%
\subsection{Discrete harmonic conjugate}
\label{sub_dis_har_conj}
Let $(\GG, c, v_0, v_1)$ be a doubly marked finite weighted planar graph properly embedded in the infinite cylinder $\CC_{2\pi}$ according to Definition~\ref{def_proper_embedd}, and let $\smash{(\hat{\GG}, \hat{c})}$ be the associated weighted dual planar graph properly embedded in $\CC_{2\pi}$ according to Definition~\ref{def_proper_embedd_dual}. Moreover, let $\smash{\frkh : \VV\GG \to [0, 1]}$ be the voltage function associated with $\smash{(\GG, c, v_0, v_1)}$ as defined in Definition~\ref{def_voltage_fct}. We want to define the discrete harmonic conjugate function of $\frkh$, i.e., the function $\frkw$ defined on the set of dual vertices $\smash{\VV\hat{\GG}}$ that satisfies the discrete Cauchy--Riemann equation. More formally, for every directed edge $e \in \EE\GG$ and its corresponding oriented dual edge $\smash{\hat{e} \in \EE\hat{\GG}}$, the function $\frkw$ should satisfy the following identity
\begin{equation}
\label{eq_Cauchy_Riemann}
\nabla \frkw(\hat{e}) := \hat{c}_{\hat{e}} \l(\frkw(\hat{e}^{+}) - \frkw(\hat{e}^{-})\r) = \frkh(e^{+}) - \frkh(e^{-}),
\end{equation}
where we recall that $\hat{c}_{\hat{e}} = 1/c_e$. 

\medskip
\noindent
To precisely define the function $\frkw$ specified above, it will be more convenient to work with the lifted weighted graph $\smash{(\GG^{\dag}, c^{\dag})}$ and its dual $\smash{(\hat{\GG}^{\dag}, \hat{c}^{\dag})}$. More precisely, we consider the lifted voltage function $\smash{\frkh^{\dag} : \VV\GG^{\dag} \to [0, 1]}$. We fix an arbitrary vertex $\smash{\hx_0 \in \VV\hat{\GG}^{\dag}}$ on the lifted dual graph, and for every $\smash{\hx \in \VV\hat{\GG}^{\dag}}$ we consider a directed path of lifted dual edges $\smash{\he_1 \cdots \he_n}$ connecting $\smash{\hx_0}$ to $\smash{\hx}$.
\begin{remark}
We emphasise that the lifted dual graph $\smash{\hat{\GG}^{\dag}}$ is always connected and so we can find a path connecting $\smash{\hx_0}$ to any $\smash{\hx \in \VV\hat{\GG}^{\dag}}$.
\end{remark}

\noindent
We define the function $\frkw^{\dag}:\VV\hat{\GG}^{\dag} \to \R$ by setting
\begin{equation}
\label{eq_conj_harmo_lift}
\mathfrak{w}^{\dag}(\hx) : =  \sum_{j=1}^{n} \nabla \frkh^{\dag}(\e_j), \quad \forall \hx \in \VV\hat{\GG}^{\dag}.
\end{equation}
where $\e_j \in \EE\GG^{\dag}$ is the oriented primal edge associated to $\he_j$. We call the function $\frkw^{\dag}$ defined in this way the \emph{lifted discrete harmonic conjugate} function associated to $\frkh^{\dag}$ with base vertex $\hx_0$. The following lemma guarantees that $\frkw^{\dag}$ is actually well-defined. 
\begin{lemma}
\label{lm_harm_conj}
For all $\hx \in \VV\hat{\GG}^{\dag}$, the value $\frkw^{\dag}(\hx)$ defined in \eqref{eq_conj_harmo_lift} does not depend on the choice of the directed path from $\hx_0$ to $\hx$. Moreover, for any $\hx_1$, $\hx_2 \in \VV\hat{\GG}^{\dag}$ such that $\sigma_{2\pi}(\hx_1) = \sigma_{2\pi}(\hx_2)$, the following relation holds
\begin{equation*}
\frac{\frkw^{\dag}(\hx_1) - \frkw^{\dag}(\hx_2)}{\eta} = \frac{\Re(\hx_1) - \Re(\hx_2)}{2 \pi},
\end{equation*}
where we recall that $\eta$ denotes the strength of the flow induced by $\frkh$ as defined in \eqref{eq_flow_strength}.

\end{lemma}
\begin{proof}
For the first part of the lemma, the proof is similar to that of \cite[Lemma~3.2]{BS96}. In particular, it is sufficient to prove that for any oriented loop $\he_1 \cdots \he_n$ in $\hat{\GG}^{\dag}$, it holds that
\begin{equation}
\label{eq_well_defined_disc_conj}
\sum_{j = 1}^{n} \l(\frkw^{\dag}(\he^{+}_j)  - \frkw^{\dag}(\he^{-}_j)\r)= 0 .
\end{equation} 
The key observation in \cite[Lemma~3.2]{BS96} is that every oriented loop $\he_1 \cdots \he_n$ in $\hat{\GG}^{\dag}$ can be written as the disjoint union of simple closed loops and of paths of length two consisting of a single dual edge traversed in both directions. Here, by a simple closed path we mean that $\he_j^{+} \neq \he_k^{+}$ for distinct $j$, $k \in [n]$ when $n > 2$, while when $n = 2$, we mean that $\he_1 \neq - \he_2$. Therefore, since \eqref{eq_well_defined_disc_conj} obviously holds if the path consists of a single dual edge traversed in both directions, we can assume without loss of generality that $\he_1 \cdots \he_n$ is a simple counter-clockwise oriented closed loop. Let $K \subset \R^2$ be the bounded connected component of $\R^2 \setminus \he_1 \cdots \he_n$. Then, thanks to \eqref{eq_Cauchy_Riemann}, it holds that
\begin{equation}
\label{eq_well_defined_disc_conj_step}
\sum_{j = 1}^{n} \l(\frkw^{\dag}(\he^{+}_j) - \frkw^{\dag}(\he^{-}_j)\r) =  \sum_{j = 1}^{n} \nabla \frkh^{\dag}(\e_j) = \sum_{\x \in \VV\GG^{\dag}(K)} \sum_{\e \in \EE\GG^{\dag}(\x)} \nabla \frkh^{\dag}(\e),
\end{equation}
where here the orientation of each edge $\e \in \EE\GG^{\dag}(\x)$ is fixed by letting $\e^{-} = \x$. The second equality in \eqref{eq_well_defined_disc_conj_step} follows from the following argument. Fix $\x \in \VV\GG^{\dag}(K)$ and consider $\y \in \VV\GG^{\dag}(\x)$. If $\y \not \in K$, then $\x\y = \e_j$ for some $j \in [n]$, while if $\y \in K$ then $\nabla \frkh^{\dag}(\x\y)$ cancels out with $\nabla \frkh^{\dag}(\y\x)$ thanks to the antisymmetry of $\nabla \frkh^{\dag}$. The term on the right-hand side of \eqref{eq_well_defined_disc_conj_step} is equal to zero thanks to the boundedness of $K$ and Kirchhoff's node law \eqref{eq_kirch_node}. 

\medskip
\noindent
Concerning the second part of the lemma, we can proceed as follows. Let $\hx_1$, $\hx_2 \in \VV\hat{\GG}^{\dag}$ be such that $\sigma_{2\pi}(\hx_1) = \sigma_{2\pi}(\hx_2)$, and let $k \in \Z$ be such that $k = (\Re(\hx_2) - \Re(\hx_1))/(2\pi)$.
Then we need to prove that $\frkw^{\dag}(\hx_2) - \frkw^{\dag}(\hx_1) = \eta k$. In particular, thanks to the first part of the lemma, it is sufficient to consider an arbitrary directed path $\he_1 \cdots \he_n$ in $\hat{\GG}^{\dag}$ connecting $\hx_1$ to $\hx_2$ and show that $\sum_{j = 1}^n (\frkw^{\dag}(\he^{+}_j) - \frkw^{\dag}(\he^{-}_j)) = \eta k$. We assume first that $k = 1$. We can choose the lifted dual edges $[\he_n]$ in such way that $\he_1 \cdots \he_n$ is a simple path oriented from left to right. Now, using an argument similar to the one used above, it is not difficult to see that $\sum_{j = 1}^n \nabla \frkh^{\dag}(\e_j) = \eta$, and so the conclusion follows in this case. Finally, the general case can be obtained easily: we can just ``glue'' together, by eventually changing the orientation if $k$ is negative, $k$ copies of the path used in the case $k=1$.  
\end{proof}

\noindent
From the definition~\eqref{eq_conj_harmo_lift} of the function $\frkw^{\dag}$, it follows that for every oriented edge $\e \in \EE\GG^{\dag}$ and for the associated dual edge $\he \in \EE\hat{\GG}^{\dag}$, it holds that 
\begin{equation*}
\nabla \frkw^{\dag}(\he) := \hat{c}^{\dag}_{\he}(\frkw^{\dag}(\he^{+}) - \frkw^{\dag}(\he^{-})) = \frkh^{\dag}(\e^{+}) - \frkh^{\dag}(\e^{-}),
\end{equation*}
i.e., $\frkw^{\dag}$ satisfies the discrete Cauchy--Riemann equation. Moreover, an immediate application of Kirchhoff's cycle law \eqref{eq_kirch_cycle} implies that the function $\frkw^{\dag}$ is harmonic on $\VV\hat{\GG}^{\dag}$. Thanks to Lemma~\ref{lm_harm_conj}, we can define the discrete harmonic conjugate function of $\frkh$ as follows.
\begin{definition}
\label{def_width_fct}
The discrete harmonic conjugate function of $\frkh$ with base vertex $\hat{x}_0 \in \VV\hat{\GG}$ is the unique function $\frkw : \VV\hat{\GG} \to \R/ \eta\Z$ such that $\frkw(\hat{x}_0) = 0$ and 
\begin{equation*}
\frkw(\hat{x}) = \frkw^{\dag}(\hx) \mod{\eta},\quad  \forall \hat{x} \in \VV\hat{\GG} ,
\end{equation*}
where $\frkw^{\dag}: \VV\GG^{\dag} \to \R$ is the function defined in \eqref{eq_conj_harmo_lift} with base vertex an arbitrary lift of $\hat{x}_0$. In view of the role that $\frkw$ will play in the construction of the Smith embedding, we will also call $\frkw$ the \emph{width coordinate function}.
\end{definition}

\begin{remark}
\label{rm_extened_discrte_harm}
As for the case of the voltage function $\frkh$, we can naturally extend the definition of $\frkw$ to a function from the dual metric graph $\hat{\G}$, i.e., $\frkw : \hat{\G} \to \R/\eta\Z$. To be precise, if $\hat{x} \in \hat{\G}$ is a point on the edge $\hat{e} \in \EE\hat{\GG}$ and $\frkw(\hat{e}^{+}) \geq \frkw(\hat{e}^{-})$, then we set
\begin{equation*}
\frkw(\hat{x}):=(\frkw(\hat{e}^{+})-\frkw(\hat{e}^{-}))\d^{\hat{\G}}(\hat{e}^{-}, \hat{x}) + \frkw(\hat{e}^{-}) .
\end{equation*}
\end{remark}

%%%%%%%%%%%%%%%%%%%%%%%%%%%%%%%%%%%%%%%%%%
\subsection{Construction of the Smith embedding}
\label{sub_def_tiling_Smith}
Let $(\GG, c, v_0, v_1)$ be a doubly marked finite weighted planar graph properly embedded in the infinite cylinder $\CC_{2\pi}$ according to Definition~\ref{def_proper_embedd}, and let $\smash{(\hat{\GG}, \hat{c})}$ be the associated weighted dual planar graph properly embedded in $\CC_{2\pi}$ according to Definition~\ref{def_proper_embedd_dual}. The aim of this subsection is to precisely define the \emph{Smith embedding} of $(\GG, c, v_0, v_1)$. As we have already explained in the introduction, the Smith embedding is built in terms of a tiling of a finite cylinder with rectangles in which every edge $e \in \EE\GG$ corresponds to a rectangle in the tiling, every vertex $x \in \VV\GG$ corresponds to the maximal horizontal segment tangent with all rectangles corresponding to the edges incident to $x$, and every dual vertex $\hat{x} \in \VV\hat{\GG}$ corresponds to the maxiaml vertical segment tangent with all rectangles corresponding to primal edges surrounding $\hat{x}$. The existence of such tiling was first proven in \cite{BSST40} and then successively extended in \cite{BS96}.

\paragraph{The main objects.} 
To precisely define the Smith embedding, it will be more convenient to work with the lifted weighted graph $\smash{(\GG^{\dag}, c^{\dag})}$ and its dual $\smash{(\hat{\GG}^{\dag}, \hat{c}^{\dag})}$. More precisely, we need to consider the lifted voltage function $\frkh^{\dag}:\VV\GG^{\dag} \to [0, 1]$ and its lifted discrete harmonic conjugate function $\frkw^{\dag}:\VV\hat{\GG}^{\dag} \to \R$.
For every edge $\smash{\e \in \EE\GG^{\dag}}$, consider its harmonic orientation and let $\smash{\he \in \EE\hat{\GG}^{\dag}}$ be the corresponding oriented dual edge. We define the intervals 
\begin{equation*}
\sfI_{\e} := \l[\frkw^{\dag}(\he^{-}), \frkw^{\dag}(\he^{+})\r], \qquad \hat{\sfI}_{\he} := \l[\frkh^{\dag}(\e^{-}), \frkh^{\dag}(\e^{+})\r].
\end{equation*}
Then, we define the rectangle $\sfR_{\e}$  associated to the edge $\e\in \EE\GG^{\dag}$ by letting
\begin{equation}
\label{eq_rec}
\sfR_{\e} := \sfI_{\e} \times \hat{\sfI}_{\he} \subset \R \times [0,1], \quad \forall \e \in \EE\GG^{\dag}.
\end{equation}
Recalling the definition \eqref{eq_conj_harmo_lift} of the lifted discrete harmonic conjugate function $\frkw^{\dag}$, it holds that 
\begin{equation*}
\frkw^{\dag}(\he^{+}) - \frkw^{\dag}(\he^{-}) = c_{\e} (\frkh^{\dag}(\e^{+}) - \frkh^{\dag}(\e^{-})).
\end{equation*} 
Therefore, the aspect ratio of the rectangle $\sfR_{\e}$ is equal to the conductance $c_{\e}$ of the edge $\e\in \EE\GG^{\dag}$. In particular, this implies that if an edge $\e\in \EE\GG^{\dag}$ has unit conductance, then $\sfR_{\e}$ is a square. For a vertex $\x \in \VV\GG^{\dag}$, we define the closed horizontal line segment $\sfH_{\x}$ by setting
\begin{equation}
\label{eq_hor}
\sfH_{\x} := \bigcup_{\e\in \EE\GG^{\dag, \downarrow}(\x)} \sfI_{\e} \times \{\frkh^{\dag}(\x)\} \subset \R \times [0,1], \quad \forall \x \in \VV\GG^{\dag},
\end{equation}
where $\EE\GG^{\dag, \downarrow}(\x)$ denotes the set of harmonically oriented lifted edges with heads equal to $\x$. Finally, for a dual vertex $\hx \in \VV\hat{\GG}^{\dag}$, we define the closed vertical line segment $\sfV_{\hx}$ by letting
\begin{equation}
\label{eq_ver}
\sfV_{\hx} := \bigcup_{\he \in \EE\hat{\GG}^{\dag, \downarrow}(\hx)} \{\frkw^{\dag}(\hx)\} \times \hat{\sfI}_{\he} \subset \R \times [0,1], \quad \forall \hx \in \VV\hat{\GG}^{\dag},
\end{equation}
where $\EE\hat{\GG}^{\dag, \downarrow}(\hx)$ denotes the set of harmonically oriented lifted dual edges with heads equal to $\hx$. Thanks to the harmonicity of the lifted height coordinate function, we observe that in the definition of $\sfH_{\x}$, one can replace $\EE\GG^{\dag, \downarrow}(\x)$ with $\EE\GG^{\dag, \uparrow}(x)$, and similarly for $\sfV_{\hat{\x}}$.

\paragraph{Construction of the tiling.}
We recall that $\eta$ denotes the strength of the flow induced by $\frkh$ as defined in \eqref{eq_flow_strength}. We consider the cylinder 
\begin{equation*}
\CC_{\eta} := \R/\eta\Z \times [0, 1],
\end{equation*}
where $\R/\eta \Z$ denotes the circle of length $\eta$. We let $(\R \times [0, 1], \sigma_{\eta})$ be the universal cover of $\CC_{\eta}$, where the covering map $\sigma_{\eta} : \R \times [0, 1] \to \CC_{\eta}$ is defined by
\begin{equation}
\label{eq_covering_eta}
\sigma_{\eta}(t, x) := \l(e^{i 2 \pi t/\eta}, x\r), \quad \forall (t, x) \in \R \times [0, 1] .
\end{equation}

\medskip
\noindent
We are now ready to define the tiling map. For each $e \in \EE\GG$, $x \in \VV\GG$, and $\hat{x} \in \VV\hat{\GG}$, we define the following objects
\begin{equation}
\label{eq_def_main_objects_Smith}
\sfR_e := \sigma_{\eta}(\sfR_{\e}), \qquad \sfH_x := \sigma_{\eta}(\sfH_{\x}), \qquad \sfV_{\hat{x}} := \sigma_{\eta}(\sfV_{\hx}) ,
\end{equation}
where $\e \in \EE\GG^{\dag}$, $\x \in \VV\GG^{\dag}$, and $\hx \in \VV\hat{\GG}^{\dag}$ are lifts of $e$, $x$, and $\hat{x}$, respectively. An immediate consequence of Lemma~\ref{lm_harm_conj} is that $\sfR_e$, $ \sfH_x$, and $\sfV_{\hat{x}}$ are well-defined, i.e., they do not depend on the particular choice of the lifts $\e$, $\x$ and $\hx$. 

\medskip
\noindent
The following properties are well-known (see~\cite[Theorem~3.1]{BS96}).
\begin{enumerate}[(a)]
\item  The collection of rectangles $\{\sfR_e\}_{e \in \EE\GG}$ constitutes a tiling of $\R/\eta\Z \times [0, 1]$, i.e., for each pair of distinct edges $e_1$, $e_2 \in \EE\GG$, the interiors of the rectangles $\sfR_{e_1}$ and $\sfR_{e_2}$ are disjoint and $\cup_{e \in \EE\GG} \sfR_e =\R/\eta\Z \times [0, 1]$.
\item For each two distinct edges $e_1$, $e_2 \in \EE\GG$, the interiors of the vertical sides of the rectangles $\sfR_{e_1}$ and $\sfR_{e_2}$ have a non-trivial intersection only if $e_1$ and $e_2$ both lie in the boundary of some common face of $\GG$.
\item Two rectangles intersect along their horizontal (resp.\ vertical) boundaries if and only if the corresponding primal (resp.\ dual) edges share an endpoint.
\end{enumerate}

\noindent
We note that if $e \in \EE\GG$ is such that no current flows through it, i.e., $\frkh(e^{-}) = \frkh(e^{+})$, then the corresponding rectangle $\sfR_e$ is degenerate and consists only of a single point. We also remark that the existence of the aforementioned tiling was proven by Benjamini and Schramm in \cite{BS96}. Originally, their proof was stated specifically for the case of edges with unit conductance, however, it can be readily extended to our setting.

\begin{definition}[Tiling map]
The tiling map associated to the quadruple $(\GG, c, v_0, v_1)$ is the map 
\begin{equation*}
\SS : \EE\GG  \cup \VV\GG \cup \VV\hat{\GG} \to \R/\eta\Z \times [0,1]
\end{equation*}
such that 
\begin{equation*}
\SS(e): = \sfR_{e},  \quad \forall e\in \EE\GG; \qquad \SS(x) := \sfH_{x}, \quad \forall x \in \VV\GG; \qquad \SS(\hat{x}): = \sfV_{\hat{x}}, \quad \forall \hat{x} \in \VV\hat{\GG},
\end{equation*}
where $\sfR_e$, $ \sfH_x$, and $\sfV_{\hat{x}}$ are as defined in \eqref{eq_def_main_objects_Smith}. The image of the tiling map $\SS$ is called the Smith diagram associated to $(\GG, c, v_0, v_1)$.
\end{definition}

\noindent
We refer to Figure~\ref{fig_smith_tiling} for an illustration of the Smith diagram associated to a given quadruple $(\GG, c, v_0, v_1)$ with unit conductances.

\begin{remark}
\label{rem_general_metric}
Since the height coordinate function $\frkh$ can be extended to the metric graph $\G$, we can view each rectangle $\smash{\sfR_e}$ of the tiling as being foliated into horizontal segments, one for each inner point of the corresponding edge $\smash{e \in \EE\GG}$. Similarly, since the width coordinate function $\frkw$ can be extended to the dual metric graph $\smash{\hat{\G}}$,  we can also view each rectangle $\sfR_e$ of the tiling as being foliated into vertical segments, one for each inner point of the corresponding dual edge $\smash{\hat{e} \in \EE\hat{\GG}}$.
\end{remark}

\noindent
It will be also convenient to introduce the lifted tiling map associated to $(\GG, c, v_0, v_1)$ which is the map
\begin{equation*}
\SS^{\dag} : \EE\GG^{\dag} \cup \VV\GG^{\dag} \cup \VV\hat{\GG}^{\dag} \to \R \times [0,1]
\end{equation*}
such that $\SS^{\dag}(\x) := \sfH_{\x}$ for each $\x \in \VV\GG^{\dag}$, $\SS^{\dag}(\e): = \sfR_{\e}$ for each $\e\in \EE\GG^{\dag}$, and $\SS^{\dag}(\hx): = \sfV_{\hx}$ for each $\hx \in \VV\hat{\GG}^{\dag}$. We emphasize that, since the collection $\{R_e\}_{e \in \EE\GG}$ forms a tiling of the cylinder $\R/\eta\Z \times [0, 1]$, the collection of rectangles $\{R_{\e}\}_{\e\in \EE\GG^{\dag}}$ forms a periodic tiling of $\R \times [0,1]$ of period $\eta$. We are now ready to precisely define the Smith embedding.
\begin{definition}[Smith embedding]
\label{def_dotted_Smith}
The Smith embedding associated to the quadruple $(\GG, c, v_0, v_1)$ is the function $\dotSS : \VV\GG \to \R/\eta\Z \times [0,1]$ such that
\begin{equation*}
	\dotSS(x) =\midpoint(\sfH_x), \quad \forall x \in \VV\GG,
\end{equation*}
where $\midpoint(\sfH_x)$ denotes the middle point of the horizontal line segment $\SS(x)$. Moreover, we define the lifted Smith embedding $\dotSS^{\dag} :\VV\GG^{\dag} \to \R \times [0, 1]$ as the map that assigns to each $\x \in \VV\GG^{\dag}$ the middle point of the horizontal line segment $\SS^{\dag}(\x)$.
\end{definition}

\noindent
We emphasize once again that the choice to define the Smith embedding by picking the middle point of each horizontal line segment is rather arbitrary. Indeed, the main result of this paper holds also if one chose an arbitrary point inside each horizontal segment. Finally, for technical reasons, we also need to introduce the following map.

\begin{definition}
\label{def_Smith_rand}
We define the map $\dotSS^{\dag, \rm{rand}}$ that assigns to each vertex $\x \in \VV\GG^{\dag}$ the random variable $\dotSS^{\dag, \rm{rand}}(\x)$ which is uniformly distributed on the horizontal line segment $\SS^{\dag}(\x)$.
\end{definition}
%%%%%%%%%%%%%%%%%%%%%%%%%%%%%%%%%%%%%%%%%%

%%%%%%%%%%%%%%%%%%%%%%%%%%%%%%%%%%%%%%%%%%
\section{Some properties of the Smith embedding}
\label{sec_prop_Smith}
In this section, we collect some results that follow directly from the construction of the Smith embedding. We fix throughout this section a doubly marked finite weighted planar graph $(\GG, c, v_0, v_1)$ properly embedded in the infinite cylinder $\CC_{2\pi}$ according to Definition~\ref{def_proper_embedd}, and we also consider the associated weighted dual planar graph $\smash{(\hat{\GG}, \hat{c})}$ properly embedded in $\CC_{2\pi}$ according to Definition~\ref{def_proper_embedd_dual}. In what follows, we consider the metric graph $\G$ associated to $\GG$, and we let $\frkh : \G \to [0, 1]$ be the extended height coordinate function as specified in Remark~\ref{rm_extened_harm}. Furthermore, we also consider the dual metric graph $\hat{\G}$, and we let $\frkw : \hat{\G} \to \R/\eta\Z$ be the extended width coordinate function as specified in Remark~\ref{rm_extened_discrte_harm}.

%%%%%%%%%%%%%%%%%%%%%%%%%%%%%%%%%%%%%%%%%%
\subsection{Adding new vertices}
\label{subsec_aux_height}
In this subsection, we examine how declaring a finite number of interior points on certain edges of the graph as vertices affects both the height coordinate function and the random walk on the graph. Notably, as proved in Lemma~\ref{lm_harm_coincides} below, this process of adding new vertices does not alter the height coordinate function on the set of original vertices. This technique proves to be useful, and we will employ it in the proof of our main result (see, for instance, the setups for the proofs of Propositions~\ref{pr_main_height} and~\ref{pr_main_witdh} in Subsections~\ref{sub_height} and~\ref{sub_width}, respectively).

\medskip
\noindent
In order to make this precise, we start with the following definition. 

\begin{definition}
\label{def_new_graph}
Let $W \subset \G$ be a finite subset of the metric graph. We define the weighted planar graph $(\GG', c')$ associated to $(\GG, c)$ and $W$ as follows:
\begin{enumerate}[(a)]
\item The set of vertices $\VV\GG'$ is given by $\VV\GG \cup W$;
\item If the interior of an edge $e \in \EE\GG$ contains $n \in \N$ points of $W$, then $e$ is split into $n+1$ new edges $\smash{[e'_{n+1}]}$ according to the points in the interior of $e$. The edge $e$ remains unchanged otherwise.
\item If the interior of an edge $e \in \EE\GG$ is split into $\smash{[e'_{n+1}]}$ new edges, for some $n \in \N$, then we set
\begin{equation*}
c'_{e'_i} := \frac{c_e}{\d^{\G}\l(e_i^{\prime, -}, e_i^{\prime, +}\r)}, \quad \forall i \in [n+1].
\end{equation*}
The conductance of $e$ remain unchanged otherwise.
\end{enumerate}
The weighted dual graph $(\hat{\GG}', \hat{c}')$ can be naturally constructed from $(\GG', c')$. 
\end{definition}

\begin{remark}
At the level of the Smith diagram, adding new vertices to the interior of some edges of the graph according to the procedure described above corresponds to horizontally dissecting the rectangles associated to such edges. More precisely, let us assume for simplicity that only one point is added to the interior of an edge $e$, and let $e_1'$ and $e_2'$ be the new edges in which $e$ is split into. Suppose that $\smash{e_1^{\prime, -} = e^{-}}$ and $\smash{e_2^{\prime, +} = e^{+}}$. Let $\SS'$ be the tiling map associated to new weighted graph. Then it is immediate to check that $\SS(e) = \SS'(e_1') \cup \SS'(e_2')$. In particular the rectangles $\SS'(e_1')$ and $\SS'(e_2')$ have the same width of the rectangle $\SS(e)$, and the height of $\SS'(e_1')$ is proportional to $\smash{\d^{\G}(e_1^{\prime, -}, e_1^{\prime, +})}$, while that of $\SS'(e_2')$ is proportional to $\smash{\d^{\G}(e_2^{\prime, -}, e_2^{\prime, +})}$. We refer to Figure~\ref{fig_smith_tiling_add_new} for a diagrammatic representation of this procedure.
\end{remark}

\begin{figure}[h]
	\centering
	\includegraphics[scale=1]{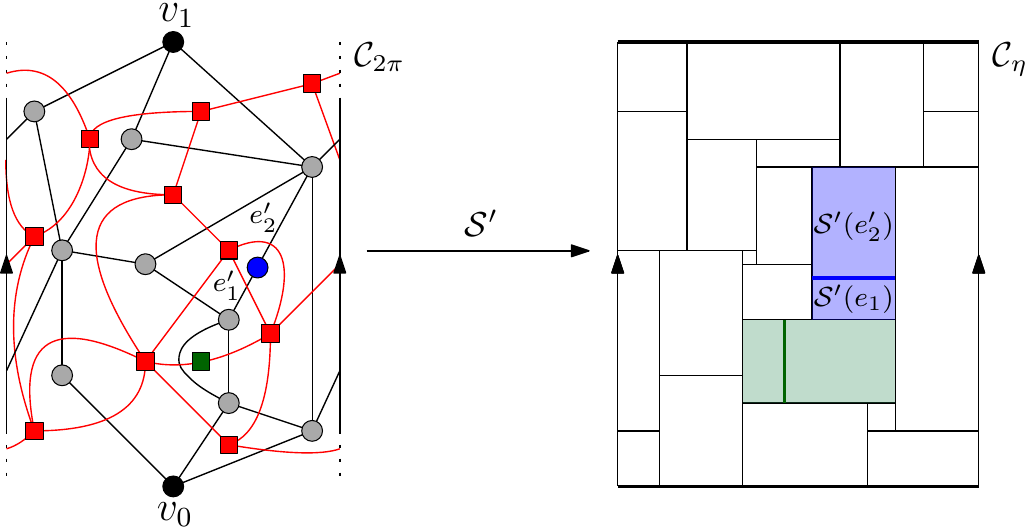}
	\caption[short form]{\small An edge $e$ is split into two edges, $\smash{e_1'}$ and $\smash{e_2'}$, with conductances $c_{e'_1}$ and $c_{e'_2}$, as specified in Definition~\ref{def_new_graph}, by adding a new blue vertex. On the right, the corresponding original rectangle $\smash{\SS(e)}$ is split into two rectangles $\smash{\SS'(e_1')}$ and $\smash{\SS'(e_2')}$ such that $\smash{\SS(e) = \SS'(e_1') \cup \SS'(e_2')}$. Similarly, a dual edge is split into two dual edges with suitable conductances by adding a new green dual vertex. On the right, the corresponding original rectangle is then split into the union of two rectangles.}
	\label{fig_smith_tiling_add_new}
\end{figure}

\noindent
Let $\frkh': \VV\GG' \to [0, 1]$ be the voltage function associated to the quadruple $(\GG', c', v_0, v_1)$. The following lemma relates the function $\frkh$ with $\frkh'$.

\begin{lemma}
\label{lm_harm_coincides}
For every $x \in \VV\GG$, it holds that $\frkh(x) = \frkh'(x)$. 
\end{lemma}
\begin{proof}
We denote by $K := \#\{\VV\GG' \setminus \VV\GG\}$ the number of new vertices added to the graph $\GG'$. If $K = 0$, then the result follows immediately. Let us now assume that $K = 1$ and suppose that the edge $e \in \EE\GG$ is split into two new edges $e_1'$ and $e_2'$ with conductances $\smash{c'_{e_1'}}$ and $\smash{c'_{e_2'}}$ as specified in Definition~\ref{def_new_graph}. Then, the desired result follows immediately from the series law of electrical network (cf.\ \cite[Section~2.3.I]{LP16}). The general case follows from a simple induction argument on $K$.
\end{proof}

\begin{lemma}
\label{lm_random_walk_assoc}
For $x \in \VV\GG$, let $X^{\prime, x}$ be the random walk on $(\GG', c')$ started from $x$. Let $\tau_0 := 0$ and, for every $k \in \N_0$, we define inductively $\tau_{k+1} := \inf\{j > \tau_{k} \, : \, X^{\prime, x}_j \in \VV\GG\}$. Then $\{X_{\tau_k}^{\prime, x}\}_{k \in \N}$ has the same distribution as the random walk on $(\GG, c)$ started from $x$.
\end{lemma}
\begin{proof}
As in the proof of Lemma~\ref{lm_harm_coincides}, let $K := \#\{\VV\GG' \setminus \VV\GG\}$. If $K = 0$, then the result is obvious. Let us now assume that $K = 1$ and let $y \in W$ be the vertex added to $\GG'$. We claim that the distribution of $X^{\prime, x}_{\tau_1}$ is equal to that of $X^{x}_{1}$. If all the edges in $\EE\GG(x)$ do not contain $y$, then the claim is obvious. Let us assume that there exists an edge in $\EE\GG(x)$ which contains $y$. Then, thanks to an easy computation, one can verify that
\begin{equation*}
\P_x(X'_{\tau_1} = v) = \frac{c_{xv}}{\pi(x)}, \quad \forall v \in \VV\GG(x).
\end{equation*}
Therefore, thanks to the strong Markov property of $X^{\prime, x}$, we get that for all $k \geq 2$, it holds that
\begin{equation*}
\P_x(X'_{\tau_k} = v \mid X'_{\tau_{k-1}} = w) = \begin{cases}
	c_{v w}/\pi(w), & v \in \VV\GG(w)  \\
	0 , & \text{otherwise}, 
\end{cases}\quad \forall k \geq 2,
\end{equation*}
which proves the result if $K=1$. The general case follows from a simple induction argument on $K$.
\end{proof}

\begin{remark}
\label{rem_width_coinc}
Given a finite set $\hat{W} \subset \hat{\G}$, following a similar procedure to the one described above, one can also construct the weighted dual graph $(\hat{\GG}', \hat{c}')$ associated to $(\hat{\GG}, c)$ and $\hat{W}$. In particular, results similar to the one stated above hold also for the respective dual counterparts. For example, if $\frkw': \VV\hat{\GG}' \to \R/\eta \Z$ is the width coordinate function associated to the new weighted graph, then $\frkw'$ restricted to $\VV\hat{\GG}$ coincides with the original width coordinate function. Moreover, adding new dual vertices to the interior of some dual edges corresponds to vertically dissecting the associated rectangles in the Smith diagram. We refer to Figure~\ref{fig_smith_tiling_add_new} for a diagrammatic representation of this procedure.
\end{remark}

%%%%%%%%%%%%%%%%%%%%%%%%%%%%%%%%%%%%%%%%%%
\subsection{Periodicity}
\label{sub_periodicity}
We collect here some properties of the lifted Smith diagram that are due to its periodicity. We recall that the map $\dotSS^{\dag, \rm{rand}}$ is defined in Definition~\ref{def_Smith_rand}.
\begin{lemma}
\label{lm_indep_position}
Let $\x_1$, $\x_2 \in \VV\GG^{\dag}$ be such that $\sigma_{2 \pi}(\x_1) = \sigma_{2 \pi}(\x_2)$. Then, it holds almost surely that
\begin{equation*}
\left|\frac{\Re(\dotSS^{\dag, \rm{rand}}(\x_2)) - \Re(\dotSS^{\dag, \rm{rand}}(\x_1))}{\eta} - \frac{\Re(\x_2)- \Re(\x_1)}{2 \pi}\right| \leq 1. 
\end{equation*}
\end{lemma}
\begin{proof}
We observe that the proof of this result is an easy consequence of Lemma~\ref{lm_harm_conj}. However, it does not follow immediately from Lemma~\ref{lm_harm_conj} since the points $\dotSS^{\dag, \rm{rand}}(\x_1)$ and $\dotSS^{\dag, \rm{rand}}(\x_2)$ are points chosen uniformly at random from the horizontal segments $\SS^{\dag}(\x_1)$ and $\SS^{\dag}(\x_2)$, respectively. We give here the details for completeness. 

\medskip
\noindent
Let $\EE\GG^{\dag, \uparrow}(\x_1) = [\e^1_{n}]$ be the set of lifted harmonically oriented edges with tails equals to $\x_1$ ordered in such a way that $\he^1_{1} \cdots \he^1_{n}$ forms a path in $\hat{\GG}^{\dag}$ oriented from left to right. Define $\EE\GG^{\dag, \uparrow}(\x_2) = [\e^2_{n}]$ in a similar way. Then, by construction of the map $\dotSS^{\dag, \rm{rand}}$, it holds that 
\begin{equation*}
\Re(\dotSS^{\dag, \rm{rand}}(\x_1)) \in \l[\frkw^{\dag}(\hat{\e}^{1, -}_{1}) , \frkw^{\dag}(\hat{\e}^{1, +}_{n})\r], \qquad \Re(\dotSS^{\dag, \rm{rand}}(\x_2)) \in \l[\frkw^{\dag}(\hat{\e}^{2, -}_{1}) , \frkw^{\dag}(\hat{\e}^{2, +}_{n})\r].
\end{equation*}
Moreover, by Lemma~\ref{lm_harm_conj}, we have that
\begin{equation*}
\frkw^{\dag}(\hat{\e}^{1, -}_{1}) - \frkw^{\dag}(\hat{\e}^{2, -}_{1})= \frac{\eta}{2\pi}\left(\Re(\x_1) - \Re(\x_2)\right), \qquad \frkw^{\dag}(\hat{\e}^{1, +}_{n}) - \frkw^{\dag}(\hat{\e}^{2, +}_{n}) = \frac{\eta}{2\pi}\left(\Re(\x_1) - \Re(\x_2)\right),
\end{equation*}
and also 
\begin{equation*}
\l|\frkw^{\dag}(\hat{\e}^{1, +}_{n}) - \frkw^{\dag}(\hat{\e}^{1, -}_{1})\r| = \l|\frkw^{\dag}(\hat{\e}^{2, +}_{n})- \frkw^{\dag}(\hat{\e}^{2, -}_{1})\r| \leq \eta, 
\end{equation*}
where the first equality is again due to Lemma~\ref{lm_harm_conj}, and the last inequality is due to the fact that, by construction of the Smith diagram, every horizontal line segment has width at most $\eta$.
Therefore, putting together all these facts, one can readily obtain the desired result.
\end{proof}

\medskip
\noindent
In order to state and prove the next lemma of this subsection, we need to introduce some notation. Let $a \in (0, 1)$, and assume that all points in the set $\frkh^{-1}(a)$ are vertices in $\VV\GG$. We define the set of harmonically oriented edges $\EE\GG_a$ and the corresponding set of dual edges $\EE\hat\GG_a$ as follows
\begin{equation*}
\EE\GG_a := \l\{e \in \EE\GG \, : \, \frkh(e^{-}) = a < \frkh(e^{+})\r\}, \qquad \EE\hat\GG_a := \l\{\hat e \in \EE\hat\GG \, : \, \hat{e} \text{ is the dual of some } e \in \EE\GG_a \r\},
\end{equation*}
and we define the following sets of vertices 
\begin{equation*}
\VV\GG_a := \l\{x \in \VV\GG \, : \, x = e^{-} \text{ for some } e \in \EE\GG_a\r\}, \qquad \VV\hat{\GG}_a := \l\{\hat{x} \in \VV\hat{\GG} \, : \, \hat{x} = \hat{e}^{-} \text{ for some } \hat{e} \in \EE\hat{\GG}_a\r\}, 
\end{equation*} 
i.e., $\VV\GG_a$ is nothing but $\frkh^{-1}(a)$.
Furthermore, we denote by $\smash{\EE\GG_a^{\dag}}$, $\smash{\EE\hat{\GG}^{\dag}_a}$, $\smash{\VV\GG_a^{\dag}}$, $\smash{\VV\hat{\GG}_a^{\dag}}$ the lifts to the universal cover of the corresponding objects.

\medskip
\noindent
Now, we let $\smash{\hx^a_0}$ be the vertex in $\smash{\VV\hat{\GG}^{\dag}_a}$ whose real coordinate is nearest to $0$. We note that the set $\smash{\EE\GG_a}$ is a cutset and that the associated set of oriented dual edges $\smash{\EE\hat{\GG}_a}$ admits an enumeration $\smash{[\hat{e}_n]}$ such that $\smash{\hat{e}_1 \cdots \hat{e}_n}$ forms a counter-clockwise oriented noncontractible simple closed loop in the dual graph $\smash{\hat{\GG}}$. In particular, every edge $\smash{\hat{e}_i}$ admits a lift $\smash{\he_i}$ such that $\smash{\he_1 \cdots \he_n}$ forms a simple path oriented from left to right joining $\smash{\hx^a_0}$ to the shifted vertex $\smash{\hx^a_0 + (0, 2\pi)}$ (see Lemma~\ref{lm_path_lifting}). We refer to Figure~\ref{fig_period} for a diagrammatic representation.

\begin{definition}
\label{def_sets_path}
Letting $[\he_n]$ be as specified above, we define the set of lifted dual vertices 
\begin{equation}
\label{eq_imp_dual}
\hat{\PPP}_a^\dagger := \l\{\hx \in \VV\hat{\GG}^{\dag}_a \, : \, \hx = \he_i^{-} \text{ for some } i \in [n] \r\},
\end{equation}
and we also define the set of lifted vertices 
\begin{equation}
\label{eq_imp_primal}
\PPP_a^\dagger := \l\{\x \in (\frkh^{\dag})^{-1}(a) \, : \, \x = \e_i^{-} \text{ for some } i \in [n] \r\} , 
\end{equation}
where $[\e_n]$ is the set of lifted edges associated to $[\he_n]$.
\end{definition}

\begin{remark}
\label{rem_large_exc}
Consider the infinite strip $\smash{\sfS^{a}_{2\pi} := \l[\Re(\hx_0^a), \Re(\hx_0^a) + 2\pi\r) \times \R}$. We remark that in general the set $\smash{\hat{\PPP}_a^\dagger}$ is not fully contained in $\smash{\sfS^{a}_{2\pi}}$.
We also observe that, thanks to the fact that $\EE\GG_a$ is a cutset, it holds that $\smash{\sigma_{2\pi}(\PPP_a^\dagger)  = \frkh^{-1}(a)}$.
\end{remark}

\noindent
We are now ready to state and prove the next lemma of this subsection. We refer to Figure~\ref{fig_period} for a diagram illustrating the various objects involved in the proof of Lemma~\ref{lm_bound_width_dot}.

\begin{figure}[h]
\centering
\includegraphics[scale=1]{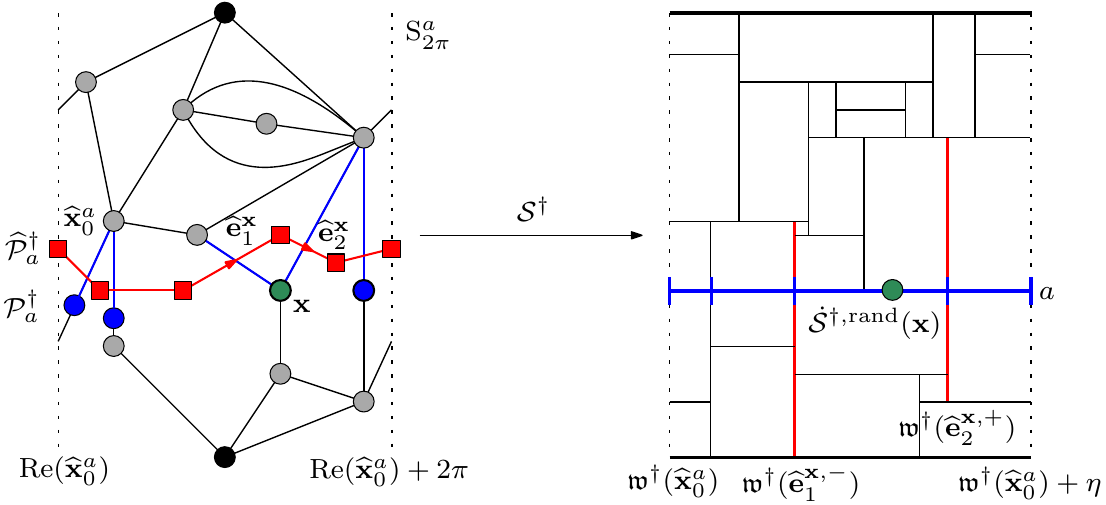}
\caption[short form]{\small A diagram illustrating the proof of Lemma~\ref{lm_bound_width_dot}. \textbf{Left:} A portion of the lifted graph $\smash{\GG^{\dag}}$ embedded in $\mathbb{R}^2$. The blue/green vertices are the vertices in the set $\smash{\PPP_a^\dagger}$. The green vertex $\x$ is a fixed vertex in $\smash{\PPP_a^\dagger}$. The red dual vertices are the vertices in $\smash{\hat{\PPP}_a^\dagger}$. The blue edges belong to the set $\smash{\EE\GG_a^{\dag}}$, and the red dual edges belong to the set $\EE\hat{\GG}^{\dag}_a$. The leftmost red dual vertex corresponds to the vertex $\hx^a_0$. The two red dual edges with an arrow are the edges $\smash{\he_1^{\x}}$ and $\smash{\he_2^{\x}}$. \textbf{Right:} A portion of the tiling of $\mathbb{R} \times [0,1]$ constructed via the lifted tiling map $\SS^{\dag}$. The blue horizontal line corresponds to $\smash{\cup_{\x \in \PPP_a^\dagger} \SS^{\dag}(\x)}$. The vertical red segments correspond to $\smash{\SS^{\dag}(\he_1^{\x, -})}$ and $\smash{\SS^{\dag}(\he_2^{\x, +})}$. The green point corresponds to a possible realization of $\dotSS^{\dag, \rm{rand}}(\x)$.}
\label{fig_period}
\end{figure}

\begin{lemma}
\label{lm_bound_width_dot}
Fix $a \in (0, 1)$ and assume that all the points in $\frkh^{-1}(a)$ are vertices in $\VV\GG$. 
Let $\hx^a_0 \in \VV\hat{\GG}_a^{\dag}$ be as specified above, then it holds that 
\begin{equation*}
0 \leq \frkw^{\dag}(\hx) - \frkw^{\dag}(\hx^a_0) \leq \eta, \quad \forall \hx \in \hat{\PPP}_a^\dagger,
\end{equation*}
and also, it holds almost surely that
\begin{equation*}
0 \leq  \Re(\dotSS^{\dag, \rm{rand}}(\x)) - \frkw^{\dag}(\hx^a_0) \leq \eta, \quad \forall \x \in \PPP_a^\dagger.
\end{equation*}
\end{lemma}
\begin{proof}
We start by proving the first part of the lemma. We consider the set of dual edges $[\he_n]$ such that $\he_1 \cdots \he_n$ forms a simple path oriented from left to right joining $\hx^a_0$ to the shifted vertex $\hx^a_0 + (0, 2\pi)$, as specified before the lemma statement. Thanks to Lemma~\ref{lm_harm_conj}, we know that
\begin{equation*}
\sum_{i = 1}^n \nabla \frkw^{\dag}(\he_i) = \eta,
\end{equation*}
and each summand in the sum is non-negative. Now, let $\hx \in \hat{\PPP}_a^\dagger$, then there exists a subpath $\he^{\hx}_1 \cdots \he^{\hx}_k$ of $\he_1 \cdots \he_n$ which connects $\hx^a_0$ to $\hx$. Therefore, we have that
\begin{equation*}
\frkw^{\dag}(\hx) - \frkw^{\dag}(\hx^a_0) = \sum_{i = 1}^k \nabla \frkw^{\dag}(\he^{\hx}_i),
\end{equation*} 
and the conclusion then follows thanks to the fact that
\begin{equation*}
0 \leq \sum_{i = 1}^k \nabla \frkw^{\dag}(\he^{\hx}_i) \leq \sum_{i = 1}^n \nabla \frkw^{\dag}(\he_i)  = \eta.
\end{equation*}

\medskip
\noindent
We now prove the second part of the lemma. To this end, we fix $\x \in \PPP_a^\dagger$, and we let $\smash{\EE\GG^{\dag, \uparrow}(\x) = [\e^\x_k]}$ be the set of lifted harmonically oriented edges with tails equal to $\x$ ordered in such a way that the corresponding lifted dual edges $\smash{\he^\x_{1} \cdots \he^\x_k}$ forms a path in $\smash{\hat{\GG}^{\dag}}$ oriented from left to right (see~Figure~\ref{fig_period}). By construction of the map $\smash{\dotSS^{\dag, \rm{rand}}}$, it holds that
\begin{equation*}
\Re(\dotSS^{\dag, \rm{rand}}(\x)) \in \l[\frkw^{\dag}(\he^{\x,-}_{1}), \frkw^{\dag}(\he^{\x, +}_{k})\r].
\end{equation*}
Finally, since both $\he^{\x,-}_{1}$ and $\he^{\x, +}_{k}$ belong to $\hat{\PPP}_a^\dagger$, the conclusion follows thanks to the first part of the lemma. 
\end{proof}

%%%%%%%%%%%%%%%%%%%%%%%%%%%%%%%%%%%%%%%%%%
\subsection{Hitting distribution of a horizontal line}
\label{sub_hitting}
Roughly speaking, the main goal of this subsection is to characterize the hitting distribution of a horizontal line on the Smith diagram for the Smith embedded random walk. Before precisely stating the result, we need to introduce some notation. 

\begin{definition}	
We define the map $\length: \VV\GG \to[0, \eta)$ as the function that assigns to each vertex $x \in \VV\GG$ the length of the horizontal segment $\SS(x)$ associated to $x$ by the Smith embedding, i.e., we let
\begin{equation*}
\length(x) := \sum_{\e\in \EE\GG^{\downarrow}(x)} \nabla \frkh(e), \quad \forall x \in \VV\GG, 
\end{equation*}
where $\EE\GG^{\downarrow}(x)$ denotes the set of harmonically oriented edges in $\EE\GG$ with heads equal to $x$.
\end{definition}

\noindent
We recall that the difference of the width coordinate function between the endpoints of a dual edge is equal to the gradient of the height coordinate of the corresponding primal edge. Moreover, thanks to the fact that the height coordinate function is harmonic, it is plain to see that 
\begin{equation*}
\length(x) = \sum_{\e\in \EE\GG^{\uparrow}(x)} \nabla \frkh(e), \quad \forall x \in \VV\GG, 
\end{equation*}
where $\EE\GG^{\uparrow}(x)$ denotes the set of harmonically oriented edges in $\EE\GG$ with tails equal to $x$. We can also naturally extend the definition of the length function to the metric graph $\G$. More precisely, given a point $x \in \G \setminus \VV\GG$, if $x$ lies in the interior of the edge $e \in \EE\GG$, then we set $\length(x):= \nabla\frkh(e)$. 
\begin{definition}
\label{def_measure_exit}
Given a value $a \in (0, 1)$, we define on the set $\frkh^{-1}(a) \subset \G$ the measure $\mu_a$ by letting
\begin{equation*}
\mu_a(x) := \frac{\length(x)}{\eta}, \quad  \forall x \in \frkh^{-1}(a).	
\end{equation*}
Since, by construction of the Smith diagram, it holds that $\smash{\sum_{x \in \frkh^{-1}(a)} \length(x) = \eta}$, the measure $\mu_a$ is a probability measure on the set $\smash{\frkh^{-1}(a)}$.
\end{definition}

\begin{remark}
\label{rem_starting_lifted_embedded}
We emphasize that, thanks to Remark~\ref{rem_large_exc}, one can also view $\mu_a$ as a probability measure on the set $\PPP_a^\dagger$, where we recall that $\PPP_a^\dagger$ is the set defined in \eqref{eq_imp_primal}.
\end{remark}

\noindent
From now on, we assume throughout the whole subsection that all the points in the set $\cup_{x \in \VV\GG}\frkh^{-1}(\frkh(x))$ are vertices in $\VV\GG$. At the level of the Smith diagram, this means that for all $x \in \VV\GG$, it holds that the set $\smash{\cup_{y \in \frkh^{-1}(\frkh(x))} \SS(y)}$ is equal to the noncontractible closed loop $\smash{\R/\eta\Z \times \{\frkh(x)\} \subset \CC_{\eta}}$. We observe that, for any given finite weighted planar graph, we can always canonically construct from it another finite weighted planar graph that satisfies this assumption. Indeed, suppose that there is a vertex $x \in \VV\GG$ such that $\smash{\frkh^{-1}(\frkh(x)) \not\subset \VV\GG}$, then, using the procedure described in Subsection~\ref{subsec_aux_height}, we can declare all the points in the finite set $\smash{\frkh^{-1}(\frkh(x)) \setminus \VV\GG}$ to be vertices of the graph.

\medskip
\noindent
We fix a value $a \in (0, 1)$ such that all the points in the set $\frkh^{-1}(a)$ are vertices in $\VV\GG$, and we let $X^{\mu_a}$ be the random walk on $(\GG, c)$ started from a point in $\frkh^{-1}(a)$ sampled according to the probability measure $\mu_a$.

\begin{definition}
\label{def_admissible_heights}
Let $X^{\mu_a}$ be as specified above. For $N \in \N$, we say that a finite sequence of height coordinates $[a_N]_0 \subset (0, 1)$ is \emph{admissible} for the random walk $X^{\mu_a}$ if $a_0 = a$ and 
\begin{equation*}
\P_{\mu_a}\l(\frkh(X_{n+1}) = a_{n+1} \mid \frkh(X_{n}) = a_{n}\r) > 0, \quad \forall n \in [N-1]_0.
\end{equation*}
\end{definition}

\noindent
We can now state the main result of this subsection.
\begin{lemma}
\label{lm_hitting_hor}
Let $X^{\mu_a}$ be as specified above. For $N \in \N$, let $[a_N]_0 \subset (0, 1)$ be an admissible sequence of height coordinates for the random walk $X^{\mu_a}$.  Then, for all $i \in [N]_0$, it holds that
\begin{equation*}
\P_{\mu_a}\l(X_{i} = x_i \mid \{\frkh(X_n) = a_n\}_{n=1}^{N}\r) = \mu_{a_i}(x_i), \quad \forall x_i \in \frkh^{-1}(a_i).
\end{equation*}
\end{lemma}

\noindent
Intuitively, the proof of the above lemma goes as follows. If $i = 1$ and we only had the conditioning on the event $\frkh(X_{1}) = a_{1}$, then the result would follow from a simple computation. However, since, in general, $i \neq 1$, and since we are also conditioning on the events $\{\frkh(X_n) = a_n\}_{n=1}^{i-1}$ and $\{\frkh(X_n) = a_n\}_{n=i+1}^{N}$, which represent the height coordinates of the random walk for past and future times respectively, the proof of the result is not immediate. However, the proof follows by a simple induction argument in which we show that we can ``forget'' about the conditioning on past and future times. Roughly speaking, the reason why we can forget about past times is due to the fact that the height component of the random walk is itself a Markov process, while the reason why we can forget about future times is due to the harmonicity of the height coordinate function. We now proceed with the proof of the lemma.

\begin{proof}[Proof of Lemma~\ref{lm_hitting_hor}] 
The proof involves three main steps. 

\medskip
\noindent\textbf{Step~1:} We start by proving that $\smash{\P_{\mu_a}(\frkh(X_1) = a_1 \mid X_0 = x_0) = \P_{\mu_a}(\frkh(X_1) = a_1)}$, for all $\smash{x_0 \in \frkh^{-1}(a)}$. Since we are assuming that $\smash{\VV\GG = \cup_{x \in \VV\GG}\frkh^{-1}(\frkh(x))}$, at its first step, the random walk can only reach two heights: height $a_{1}$ or height $\tilde{a}_{1}$ say. If $a_{1}> a$, since $\frkh$ is harmonic at $x_0$, it holds that
\begin{equation}
\label{eq_inter_hitting}
\frkh(x_0) = \frac{\sum_{e \in \EE\GG^{\downarrow}(x_0)} c_e \frkh(e^-) +\sum_{e \in \EE\GG^{\uparrow}(x_0)} c_e \frkh(e^+)}{\pi(x_0)} = \tilde{a}_{1} \frac{\sum_{e \in \EE\GG^{\downarrow}(x_0)} c_e} {\pi(x_0)} +  a_{1} \frac{\sum_{e \in \EE\GG^{\uparrow}(x_0)} c_e}{\pi(x_0)}, 
\end{equation}
where, as usual, $\EE\GG^{\downarrow}(x_0)$ (resp.\ $\EE\GG^{\uparrow}(x_0)$) denotes the set of harmonically oriented edges with heads (resp.\ tails) equal to $x_0$. 
Now, we observe that
\begin{equation*}
\P_{\mu_a}\l(\frkh(X_1) = a_{1} \mid X_0 = x_0 \r) = \frac{\sum_{e \in \EE\GG^{\uparrow}(x_0)} c_e}{\pi(x_0)}, \qquad \P_{\mu_a}\l(\frkh(X_1) = \tilde{a}_{1} \mid X_0 = x_0 \r) = \frac{\sum_{e \in \EE\GG^{\downarrow}(x_0)} c_e}{\pi(x_0)}.
\end{equation*}
In particular, plugging these expressions into \eqref{eq_inter_hitting} and rearranging, we obtain that
\begin{equation*}
\P\l(\frkh(X_1) = a_{1} \mid X_0 = x_0\r) = \frac{|\tilde{a}_{1} - a|}{|\tilde{a}_{1} - a_{1}|}, \quad \forall x_0 \in \frkh^{-1}(a),
\end{equation*}
from which the desired result follows since the right-hand side of the above expression does not depend on the particular choice of $\smash{x_0 \in \frkh^{-1}(a)}$. A similar conclusion also holds if $a_{1} < a$. Now, proceeding by induction, one can prove that for any $i \in[N]_0$ and for all $x_0 \in \frkh^{-1}(a), \ldots, x_i \in \frkh^{-1}(a_i)$, it holds that 
\begin{equation*}
\P_{\mu_a}\l(\{\frkh(X_{i+n}) = a_{i+n}\}_{n = 1}^{N-i} \mid \{X_j = x_j\}_{j = 0}^{i}\r) = \P_{\mu_a} (\{\frkh(X_{i+n}) = a_{i+n}\}_{n = 1}^{N-i} \mid \frkh(X_i) = a_i) .
\end{equation*} 

\medskip
\noindent\textbf{Step~2:} 
Thanks to the previous step and Bayes' rule, we have that $\smash{\P_{\mu_a}\l(X_0 = x_0 \mid \frkh(X_1) = a_1\r) = \mu_{a}(x_0)}$, for all $x_0 \in \frkh^{-1}(a)$. In particular, using this fact, we will now prove that 
\begin{equation*}
\P_{\mu_a}\l(X_{1} = x_{1} \mid \frkh(X_1) = a_1\r) = \mu_{a_1}(x_{1}), \quad \forall x_{1} \in \frkh^{-1}(a_1).
\end{equation*}
To this end, fix $x_{1} \in \frkh^{-1}(a_1)$ and suppose that $a_1 > a$. Then we can proceed as follows
\begin{align*}
& \P_{\mu_{a}}\l(X_1 = x_{1} \mid \frkh(X_1) = a_1\r) \\
& \qquad = \sum_{x_0 \in \frkh^{-1}(a) \cap \VV\GG(x_1)}  \P_{x_0}\l(X_1 = x_{1} \mid \frkh(X_1) = a_1\r) \P_{\mu_a}\l(X_0 = x_0 \mid \frkh(X_1) = a_1\r) \\
& \qquad = \sum_{x_0 \in \frkh^{-1}(a) \cap \VV\GG(x_{1})}\frac{c_{x_{0} x_{1}}}{\sum_{v \in \frkh^{-1}(a_1) \cap \VV\GG(x_0)} c_{x_0 v}} \mu_{a}(x_0) \\
& \qquad = \sum_{x_0 \in \frkh^{-1}(a) \cap \VV\GG(x_1)} \frac{c_{x_0 x_1}}{\sum_{v \in \frkh^{-1}(a_1) \cap \VV\GG(x_0)} c_{x_0 v}} \frac{|a_1 - a| \sum_{v \in \frkh^{-1}(a_1) \cap \VV\GG(x_0)} c_{x_0 v}}{\eta} \\
& \qquad = \frac{1}{\eta} \sum_{e \in \EE\GG^{\downarrow}(x)} \nabla \frkh(e) \\
& \qquad = \mu_{a_1}(x),
\end{align*}
where $\EE\GG^{\downarrow}(x)$ denotes the set of harmonically oriented edges with heads equal to $x$. In order to pass from the second line to the third line of the above display, we used the definition of the probability measure $\mu_{a}$, and the fact that, for all $e \in \EE\GG^{\downarrow}(x)$, it holds that $\frkh(e^{-}) = a$ and $\frkh(e^{+}) = a_1$. One can proceed similarly if $a_1 < a$. Now, proceeding again by induction, one can easily prove that, for all $i \in[N]$, it holds that
\begin{equation*}
\P_{\mu_a}\l(X_{i} = x_i \mid \{\frkh(X_n) = a_n\}_{n = 1}^{i}\r) = \mu_{a_i}(x_i), \quad \forall x_i \in \frkh^{-1}(a_i).
\end{equation*}

\medskip
\noindent\textbf{Step~3:} For $i \in [N]_0$, we observe that a consequence of Step~1 is that the sequence of random variables $\{X_j\}_{j = 0}^{i}$ is conditionally independent from $\{\frkh(X_{i+n})\}_{n = 1}^{N-i}$ given $\frkh(X_i) = a_i$. Therefore, the conditional law of $\{X_j\}_{j = 0}^{i}$ given $\{\frkh(X_{n}) = a_n\}_{n = 1}^{i}$ is the same as the conditional law given $\{\frkh(X_{n}) = a_n\}_{n = 1}^{N}$. Hence, this implies that 
\begin{equation*}
\P_{\mu_a}\l(X_{i} = x_i \mid \l\{\frkh(X_n) = a_n\r\}_{n=1}^{N}\r) = \P_{\mu_a}\l(X_{i} = x_i \mid \l\{\frkh(X_n) = a_n\r\}_{n=1}^{i}\r), \quad \forall x_i \in \frkh^{-1}(a_i), 
\end{equation*}
and so the conclusion follows from Step~2.
\end{proof}

%%%%%%%%%%%%%%%%%%%%%%%%%%%%%%%%%%%%%%%%%%
\subsection{Expected horizontal winding}
\label{subsec_winding}
Roughly speaking, the main goal of this subsection is to establish that the average winding that the Smith-embedded random walk does before hitting a given horizontal line on the Smith diagram is zero. Before precisely stating the result, we need to introduce some notation.

\begin{definition}
Given the random walk $\X^{\x}$ on the lifted weighted graph $(\GG^{\dag}, c^{\dag})$ started from $\x \in \VV\GG^{\dag}$, we let 
\begin{equation*}
	\N_0 \ni n  \mapsto \dotX^{\x}_n 
\end{equation*}
be the discrete time process taking values in $\R \times [0, 1]$ such that, for each $n \in \N_0$, the conditional law of $\dotX^{\x}_n$, given $\X^{\x}_n$, is equal to the law of $\dotSS^{\dag, \rm{rand}}(\X_n^{\x})$, where we recall that $\dotSS^{\dag, \rm{rand}}$ is defined in Definition~\ref{def_Smith_rand}. We call the process $\dotX^{\x}$ the \emph{lifted Smith-embedded random walk} associated to $\X^{\x}$.
\end{definition}

\noindent
It follows from the definition of $\dotX^{\x}$ that, at each time step $n \in \N_0$, the location of the point $\dotX^{\x}_n$ is sampled uniformly at random and independently of everything else from the horizontal line segment $\smash{\SS^{\dag}(\X^{\x}_n)}$.
With a slight abuse of notation, we also denote by $\smash{\dotX^{\x}}$ the continuous time version $\smash{\{\dotX^{\x}_t\}_{t \geq 0}}$, where the continuous path is generated by piecewise linear interpolation at constant speed. 

\medskip
\noindent
We assume also in this subsection that all the points in the set $\cup_{x \in \VV\GG}\frkh^{-1}(\frkh(x))$ are vertices in $\VV\GG$. Furthermore, we fix a value $a \in (0, 1)$ such that all the points in the set $\frkh^{-1}(a)$ are vertices in $\VV\GG$, and we let $X^{\mu_a}$ be the random walk on $(\GG, c)$ started from a point in $\frkh^{-1}(a)$ sampled according to the probability measure $\mu_a$ defined in Definition~\ref{def_measure_exit}. We also adopt the convention to denote by 
\begin{equation*}
	\N_0 \ni n  \mapsto \X^{\mu_a}_n
\end{equation*}
the lift of the random walk $X^{\mu_a}$ to the lifted weighted graph $(\GG^{\dag}, c^{\dag})$ started from a point in $ \PPP_a^\dagger$ sampled according to the probability measure $\mu_a$ (see Remark~\ref{rem_starting_lifted_embedded}). Moreover, we denote by $\dotX^{\mu_a}$ the lifted Smith-embedded random walk associated to $\X^{\mu_a}$ as specified above.

\medskip
\noindent
In complete analogy with the definition of winding of a path in the infinite cylinder $\CC_{2\pi}$, we have the following definition of winding on the cylinder $\CC_{\eta}$.
\begin{definition}
\label{def_winding_smith}
Let $0 \leq t_1 < t_2$, consider a path $P :[t_1, t_2] \to \CC_{\eta}$, and let $\PP:[t_1, t_2] \to \R \times [0, 1]$ be its lift to the covering space $\R \times [0, 1]$. Then, we define the winding of $P$ as follows
\begin{equation}
\label{eq_wind_eta}
\wind_\eta(P) =\frac{\Re(\PP(t_2)) - \Re(\PP(t_1))}{\eta}.
\end{equation}
\end{definition}
\noindent
In what follows, with a slight abuse of notation, if $\PP:[t_1, t_2] \to \R \times [0, 1]$ is a lift of a path $P :[t_1, t_2] \to \CC_{\eta}$, then we may write $\wind_{\eta}(\PP)$ to denote $\wind_{\eta}(P)$.

\medskip
\noindent
We are now ready to state the main result of this subsection.
\begin{lemma}
\label{lm_winding_embed_walk}
Let $X^{\mu_a}$ and $\dotX^{\mu_a}$ be as specified above. For $N \in \N$, let $[a_N]_0 \subset (0, 1)$ be an admissible sequence of height coordinates for the random walk $X^{\mu_a}$ as specified in Definition~\ref{def_admissible_heights}. Then it holds that
\begin{equation*}
\E_{\mu_a}\l[\wind_{\eta}\l(\dotX|_{[0, N]}\r) \mid \l\{\frkh(X_n) = a_n \r\}_{n = 1}^{N}\r] = 0.
\end{equation*}
\end{lemma}

\medskip
\noindent
Heuristically speaking, the proof of the above lemma goes as follows. First, we observe that, by definition of winding, it holds that 
\begin{equation*}
\wind_{\eta}(\dotX|_{[0, N]}) = \sum_{j = 0}^{N-1} \wind_{\eta}(\dotX|_{[j, j+1]}).
\end{equation*}
In particular, this implies that we can reduce the problem to studying the expected winding at each time step of the random walk. Suppose for a moment that $N=1$ in the lemma statement. Then we only need to prove that the random walk started uniformly at random from height $a$ and conditioned to hit height $a_1$ at its first time step has zero expected winding. The reason why this is true is due to the fact that the re-randomization procedure that takes place inside each small segment ensures that the conditional hitting distribution of the segment at height $a_1$ is also uniformly distributed. To extend the argument to the general case $N \in \N$, we just need to use Lemma~\ref{lm_hitting_hor}.  

\begin{proof}[Proof of Lemma~\ref{lm_winding_embed_walk}]
We start by finding an equivalent condition under which the relation stated in the lemma holds true.
To this end, we observe that, thanks to the Definition~\ref{def_winding_smith} of winding, it holds that
\begin{equation*}
\label{eq_decomp_winding}
\E_{\mu_a}\l[\wind_{\eta}\l(\dotX|_{[0, N]}\r) \mid \l\{\frkh(X_n) = a_n\r\}_{n = 1}^{N}\r] = \frac{1}{\eta} \sum_{j = 0}^{N-1} \E_{\mu_a}\l[\Re(\dotX_{j+1}) - \Re(\dotX_{j}) \mid \l\{\frkh(X_n) = a_n\r\}_{n = 1}^{N}\r],
\end{equation*}
and so the claim follows if we prove that 
\begin{equation*}
 \E_{\mu_a}\l[\Re(\dotX_{j+1})\mid \l\{\frkh(X_n) = a_n\r\}_{n = 1}^{N}\r]	=  \E_{\mu_a}\l[\Re(\dotX_{j}) \mid \l\{\frkh(X_n) = a_n\r\}_{n = 1}^{N}\r], \quad \forall j \in [N-1]_0.
\end{equation*} 
An immediate application of Lemma~\ref{lm_hitting_hor} shows that the previous equality holds true if and only if 
\begin{equation}
\label{eq_pre_aux_winding_11}
 \E_{\mu_a}\l[\Re(\dotX_{j+1})\mid \frkh(X_{j+1}) = a_{j+1}\r]	=  \E_{\mu_a}\l[\Re(\dotX_{j}) \mid \frkh(X_j) = a_j\r], \quad \forall j \in [N-1]_0.
\end{equation} 
Now, for every $k \in [N-1]_{0}$, we let $X^{k}$ be the random walk on the weighted graph $(\GG, c)$ started from a point in $\frkh^{-1}(a_k)$ sampled according to the probability measure $\mu_{a_k}$. A consequence of Lemma~\ref{lm_hitting_hor} is that the conditional law of $X^{\mu_a}$ given $\frkh(X_k^{\mu_a}) = a_k$ is equal to the law of $X^k$. In particular, this implies that \eqref{eq_pre_aux_winding_11} is equivalent to the fact that
\begin{equation}
\label{eq_aux_winding}
\E\l[\Re(\dotX^{j+1}_{0})\r] = \E\l[\Re(\dotX^j_{0})\r], \quad \forall j \in [N-1]_0,
\end{equation}
where, for $k \in [N-1]_{0}$, $\smash{\X^{k}}$ denotes the lift of $\smash{X^{k}}$ started from a point in $\smash{\PPP^{\dag}_{a_k}}$, and $\smash{\dotX^{k}}$ is the lifted Smith-embedded random walk associated to $\X^{k}$. Let us also emphasize that here we are relying on the fact that the definition of winding does not depend on the particular choice of the lift.
In order to show that \eqref{eq_aux_winding} holds true, recalling that $\sigma_{\eta}$ is defined in \eqref{eq_covering_eta} and denotes the covering map of the cylinder $\CC_{\eta}$, the equality \eqref{eq_aux_winding} can be rewritten as follows
\begin{equation}
\label{eq_aux_winding_unlifted}
\E\l[\Re(\sigma_{\eta}(\dotX^{j+1}_{0})) \r] = \E\l[\Re(\sigma_{\eta}(\dotX^j_{0}))\r], \quad \forall j \in [N-1]_0.
\end{equation}
By construction, for every $k \in [N-1]_{0}$, the random variable $\smash{\Re(\sigma_{\eta}(\dotX^k_{0}))}$ is uniformly distributed on the interval $[0, \eta)$, and so the result follows.
\end{proof}
%%%%%%%%%%%%%%%%%%%%%%%%%%%%%%%%%%%%%%%%%%

%%%%%%%%%%%%%%%%%%%%%%%%%%%%%%%%%%%%%%%%%
\section{Proof of the main result}
\label{sec_assumption_main}
The main goal of this section is to prove Theorem~\ref{th_main_1}. To this end, we fix a sequence $\{(\GG^n, c^n, v_0^n, v_1^n)\}_{n \in \N}$ of doubly marked finite weighted planar maps and we let $\{(\hat{\GG}^n, \hat{c}^n)\}_{n \in \N}$ be the sequence of associated weighted dual planar graphs. We assume throughout this section that such sequences satisfy hypotheses \ref{it_embedding}, \ref{it_invariance}, \ref{it_invariance_dual}.

\medskip
\noindent
Before moving to the proof of the main theorem, we prove a couple of simple results which are immediate consequences of our assumptions and which will be useful later on. In particular, the next lemma is basically an analogue of assumptions~\ref{it_invariance}~and~\ref{it_invariance_dual} in the setting of the universal cover. 
\begin{lemma}
For each $n \in \N$, view the embedded random walk on the lifted weighted graph $\smash{(\GG^{\dag, n}, c^{\dag, n})}$, stopped when it traverses an unbounded edge, as a continuous curve in $\R^2$ obtained by piecewise linear interpolation at constant speed. For each $R > 0$ and for any $z \in \R \times [-R, R]$, the law of the random walk on $\smash{(\GG^{\dag, n}, c^{\dag, n})}$ started from the vertex $\smash{\x_z^n \in \VV\GG^{\dag, n}}$ nearest to $z$ weakly converges as $n \to \infty$ to the law of the Brownian motion in $\R^2$ started from $z$ with respect to the local topology on curves viewed modulo time parameterization specified in Subsection \ref{subsub_metric_CMP}, uniformly over all $z \in \R \times [-R, R]$. Furthermore, the same result holds for the random walk on the sequence of lifted weighted dual graphs $\smash{\{(\hat{\GG}^{\dag, n}, \hat{c}^{\dag, n})\}_{n \in N}}$.
\end{lemma}
\begin{proof}
Obviously, we can just prove the result for the random walk on the sequence of primal graphs as the result for the random walk on the sequence of dual graphs can be proved in the same exact way.
Fix $R > 0$ and $z \in \R \times [-R, R]$, and let $\X^{n, z}$ be the continuous time random walk on $(\GG^{\dag, n}, c^{\dag, n})$ started from $\x_z^n \in \VV\GG^{\dag, n}$ as specified in the lemma statement. We let $\tau_0 := 0$ and for $k \in \N_0$ we define inductively
\begin{equation*}
\tau_{k+1} := \inf\l\{t \geq \tau_{k} \, : \, \Re(X_t) \not\in [\Re(X_{\tau_k}) - \pi, \Re(X_{\tau_k}) + \pi)\r\}.
\end{equation*}
For each $k \in \N_0$, we observe that the universal covering map $\sigma_{2\pi}: \R^2 \to \CC_{2\pi}$ is a biholomorphism when restricted to $[\Re(X_{\tau_k}) - \pi, \Re(X_{\tau_k}) + \pi)$. Moreover, by assumption~\ref{it_invariance}, we know that the law of $\sigma_{2\pi}(\X^{n, z})$ weakly converges as $n \to \infty$ to the law of the Brownian motion in $\CC_{2\pi}$ with respect to the metric $\dCMPloc$ specified in Subsection~\ref{subsub_metric_CMP}. Therefore, since Brownian motion is conformally invariant and putting together the previous facts, we obtain that the law of the random walk $\X^{n, z}$ weakly converges as $n \to \infty$ to the law of the Brownian motion in $\R^2$ with respect to the metric $\dCMP$ on the time interval $[\tau_{k}, \tau_{k+1}]$, for each $k \in \N_0$. Hence, the desired result follows.
\end{proof}

\noindent
For each $n \in \N$, we recall that $\FF\GG^n$ denotes the set of faces of the graph $\GG^n$. The next lemma states that, thanks to the invariance principle assumption on the sequence of weighted graphs, the maximal diameter of the faces in $\FF\GG^n$ which intersect a compact subset of $\CC_{2\pi}$ is of order $o_n(1)$. 

\begin{lemma}
\label{lm_no_macroscopic_edges}
Let $R > 0$, and, for any $n \in \N$, consider the set $\FF\GG^{n}(R)$ of faces in $\FF\GG^n$ that intersect $\R/2\pi\Z \times [-R, R]$. Then, for any $\eps > 0$ and for any $n>n(R, \eps)$ large enough, it holds that
\begin{equation*}
\diam(f) \leq \eps, \quad \forall f \in \FF\GG^n(R).
\end{equation*}
The same result holds also for the sequence of dual graphs, i.e., with $\FF\hat\GG^{n}$ in place of $\FF\GG^{n}$.
\end{lemma}
\begin{proof}
Obviously, we can just prove the result for the sequence of primal graphs as the result for the sequence of dual graphs can be proved in the same exact way.
We proceed by contradiction by assuming that there is $\eps > 0$ such that, for any $n \in \N$, there is a face $f_n \in \FF\GG^n(R)$ for which $\diam(f_n) > \eps$. We notice that each set $f_n \cap \R/2\pi\Z \times [-2R, 2R]$ is compact and contains a path $P_n$ with $\diam(P_n) > \eps$. By compactness of the Hausdorff distance $\dH$, we can assume, by possibly taking a subsequential limit, that $\lim_{n \to \infty} \dH(P_n, P) = 0$ for some compact and connected set $P \subset \R/2\pi\Z \times [-2R, 2R]$. 
Now, choose a rectangle $\sfQ$ such that $P$ disconnects the left and right sides of $\sfQ$, or the top and bottom sides of $\sfQ$. For any $n>n(R, \eps, \sfQ)$ large enough, also the path $P_n$ disconnects the rectangle $\sfQ$ in the same way as $P$. Therefore, for any $x \in \R/2\pi\Z \times [-R, R]$, the random walk $X^{n, x}$ on the weighted graph $(\GG^n, c^n)$ started from $x$ cannot cross the rectangle $\sfQ$ from left to right, or from top to bottom. However, the Brownian motion $B^x$ on $\CC_{2\pi}$ started from $x$ has positive probability to do so. This is in contradiction with assumption~\ref{it_invariance}, and so the desired result follows.  
\end{proof}

\noindent
We are now ready to start the proof of Theorem~\ref{th_main_1}. We observe that this result admits a natural version on the sequence of lifted weighted graphs. Indeed, in order to prove it, we will work in the setting of the universal cover. To be more precise, we will first study in Subsection~\ref{sub_height} the sequence of lifted height coordinate functions, and then, in Subsection~\ref{sub_width}, we will study the sequence of lifted width coordinate functions. Once this is done, we will put everything together and we will conclude the proof of Theorem~\ref{th_main_1} in Subsection~\ref{sub_proof_main}.

%%%%%%%%%%%%%%%%%%%%%%%%%%%%%%%%%%%%%%%%%%
\subsection{Height coordinate function}
\label{sub_height}
For each $n \in \N$, consider the lifted height coordinate function 
\begin{equation*}
\frkh^{\dag}_n:\VV\GG^{\dag, n} \to [0,1],	
\end{equation*}
as defined in \eqref{eq_harmo_lift}.
The main goal of this subsection is to show that there exists an affine transformation of the function $\frkh^{\dag}_n$ that is asymptotically close, uniformly on lifts of compact subsets of the infinite cylinder $\CC_{2\pi}$, to the height coordinate function $\x \mapsto \Im(\x)$ in the a priori lifted embedded graph $\GG^{\dag, n}$, as $n \to \infty$.
More precisely, we have the following result.

\begin{proposition}
\label{pr_main_height}
There exit two sequences $\{b^{\frkh}_n\}_{n \in \N}$,  $\{c^{\frkh}_n\}_{n \in \N} \subset \R$ such that for all $R > 1$, $\delta \in (0, 1)$, and for any $n > n(R, \delta)$ large enough, it holds that
\begin{equation*}
\l|c_n^{\frkh} \frkh^{\dag}_n(\x) + b_n^{\frkh} - \Im(\x)\r| \leq \delta, \quad \forall \x \in \VV\GG^{\dag, n}(\R \times [-R, R]).
\end{equation*}
\end{proposition}

\noindent
The proof of this proposition is given in Subsection~\ref{subsub_proof_main_height}. As we will see, the proof will follow easily thanks to the harmonicity of the height coordinate function together with Lemma~\ref{lm_main_height} below.

\subsubsection{Setup of the proof}
For each $n \in \N$, consider the metric graph $\G^n$ associated to $\GG^n$, and let $\frkh^n : \G^n \to [0, 1]$ be the extended height coordinate function as specified in Remark~\ref{rm_extened_harm}. Given a value $S \in \R$, we define the set 
\begin{equation*}
V^n_{S} := \l\{x \in \G^n \, : \,  \Im(x) = S\r\},
\end{equation*}
and we let
\begin{equation}
\label{eq_key_values_height}
\bar{a}_{S}^n :=\sup\left\{a \in (0, 1) \, : \, \frkh^{-1}_n(a) \cap V^n_S \neq \emptyset\right\}, \quad \underline{a}_{S}^n := \inf\left\{a \in (0, 1) \, : \, \frkh^{-1}_n(a) \cap V^n_S \neq \emptyset\right\}.
\end{equation}
We fix throughout $R > 1$ and $\delta \in (0, 1)$ as in the proposition statement, and we let 
\begin{equation*}
R' := \frac{R}{\delta}.
\end{equation*}
We consider the set 
\begin{equation*}
W_{R, \delta}^n := \left\{V^n_{R'} \cup V^n_{-R'}\right\} \cup \left\{\frkh^{-1}_n(\bar{a}_{R'}^n) \cup \frkh^{-1}_n(\underline{a}_{-R'}^n)\right\} \subset \G^n.
\end{equation*}
For each $n \in \N$, by possibly locally modifying the a priori embedding of the graph $\GG^n$ in the infinite cylinder $\CC_{2\pi}$, we can assume without loss of generality that each edge in $\EE\GG^n$ crosses the circles at height $R'$ and $-R'$ at most finitely many times. In particular, this implies that we can assume that the set $W_{R, \delta}^n$ contains at most finitely many points, and therefore, by Lemma~\ref{lm_harm_coincides}, we can assume, without any loss of generality, that $\VV\GG^n$ contains all the points in $W^n_{R, \delta}$. 

\medskip
\noindent
In what follows, in order to lighten up notation, we adopt the following notational convention 
\begin{equation*}
	\hat{V}_{S} \equiv \hat{V}^n_{S}, \qquad \bar{a} \equiv \bar{a}^n_{R'}, \quad \underline{a} \equiv \underline{a}^n_{-R'}.
\end{equation*}
Furthermore, we denote by $V^{\dag}_S$ the lift to the universal cover of $V_S$, and we write $\VV\GG^{\dag, n}(S)$ (resp.\ $\VV\GG^{n}(S)$) as a shorthand for $\VV\GG^{\dag, n}(\R \times [-S, S])$ (resp.\ $\VV\GG^{n}(\R/2\pi\Z \times [-S, S])$). We refer to Figure~\ref{fig_height} for an illustration of the sets involved in the proof of Proposition~\ref{pr_main_height}.

\begin{figure}[h]
\centering
\includegraphics[scale=1]{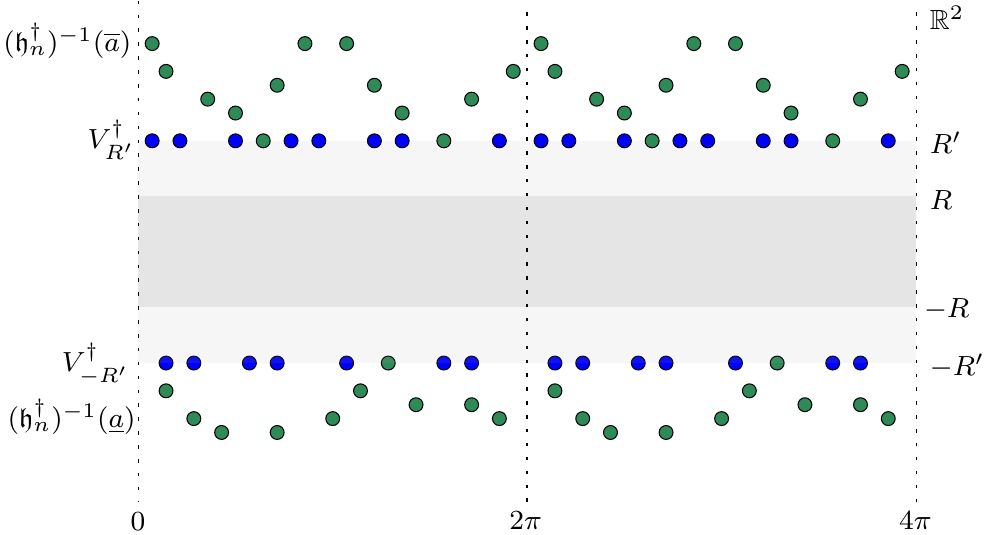}
\caption[short form]{\small A diagram illustrating the sets involved in the proof of Proposition~\ref{pr_main_height}. The shaded dark-gray strip between heights $R$ and $-R$ contains all the vertices in the set $\smash{\VV\GG^{\dag, n}(R)}$. The shaded light-gray strip between heights $R'$ and $-R'$ contains all the vertices in the set $\smash{\VV\GG^{\dag, n}(R')}$. The blue vertices at the top (resp.\ bottom) are the vertices in the set $\smash{V^{\dag}_{R'}}$ (resp.\ $\smash{V^{\dag}_{-R'}}$). The green vertices at the top (resp.\ bottom) are the vertices in the set $\smash{(\frkh_n^{\dag})^{-1}(\bar{a})}$ (resp.\ $(\frkh_n^{\dag})^{-1}(\underline{a})$). Note that, by definition of $\bar{a}$, the set $\smash{(\frkh_n^{\dag})^{-1}(\bar{a})}$ intersects $\smash{V^{\dag}_{R'}}$, and similarly, by definition of $\underline{a}$, the set $\smash{(\frkh_n^{\dag})^{-1}(\underline{a})}$ intersects $\smash{V^{\dag}_{-R'}}$.}  
\label{fig_height}
\end{figure}

\paragraph{Random walk notation.}
For $\x \in \VV\GG^{\dag, n}(R')$, we consider the continuous time random walk $\{\X^{n, \x}_t\}_{t \geq 0}$
on the lifted weighted graph $(\GG^{\dag, n}, c^{\dag, n})$ started from $\x \in \VV\GG^{\dag, n}(R')$. We recall that the continuous path of this random walk is generated by piecewise linear interpolation at constant speed. We consider the following stopping times
\begin{equation}
\label{eq_stopping_times}
\sigma_{\x} := \inf\l\{t \geq 0  \, : \, \X^{n, \x}_t \in V_{R'}^{\dag} \cup V_{-R'}^{\dag}\r\}, \quad \tau_{\x} := \inf \l\{t \geq 0  \, : \, \X^{n, \x}_t \in (\frkh^{\dag}_n)^{-1}(\bar{a}) \cup (\frkh^{\dag}_n)^{-1}(\underline{a})\r\},
\end{equation}
and we observe that, thanks to the definitions of $\bar{a}$ and $\underline{a}$, it holds that $\tau_{\x} \geq \sigma_{\x}$, for all $\x \in \VV\GG^{\dag, n}(R')$. Looking at Figure~\ref{fig_height}, the stopping time $\sigma_\x$ accounts for the first time at which the random walk hits one of the blue vertices, while the stopping time $\tau_\x$ accounts for the first time at which the random walk hits one of the green vertices.

\subsubsection{Proof of Proposition~\ref{pr_main_height}}
\label{subsub_proof_main_height}
We can now state a key lemma for the proof of Proposition~\ref{pr_main_height}. The proof of the below result is given in Subsection~\ref{sub_aux_res_height}. 

\begin{lemma}
\label{lm_main_height}
For any $n > n(R, \delta)$ large enough, there exists a real number $b'_n = b_n'(R, \delta)$ such that
\begin{equation*}
\l|2 R' \P_{\x}\l(\X^n_{\tau_\x} \in (\frkh_n^{\dag})^{-1}(\bar{a})\r) - b'_n  -  \Im(\x)\r| \leq \delta, \quad \forall \x \in  \VV\GG^{\dag, n}(R).
\end{equation*}	
Similarly, for any $n > n(R, \delta)$ large enough, there exists a real number $b''_n = b_n''(R, \delta)$ such that
\begin{equation*}
\l|2 R' \P_\x\l(\X^n_{\tau_\x} \in \sigma_{2\pi}^{-1}(\frkh_n^{-1}(\underline{a}))\r)+ b''_n  + \Im(\x)\r| \leq \delta, \quad \forall \x \in  \VV\GG^{\dag, n}(R).
\end{equation*}	
\end{lemma}

\noindent
Given Lemma~\ref{lm_main_height}, the proof of Proposition~\ref{pr_main_height} follows by using the harmonicity of the height coordinate function. 

\begin{proof}[Proof of Proposition~\ref{pr_main_height}]
We divide the proof in two steps.

\medskip
\noindent\textbf{Step 1.} In this first step we show that, for fixed $R > 1$ and $\delta \in (0, 1)$, for any $n > n(R, \delta)$ large enough, we can find real numbers $b^{R, \delta}_n$ and $c^{R, \delta}_n$ such that the conclusion of the proposition holds. To this end, let $R' := R/\delta$, fix $\x \in \VV\GG^{\dag, n}(R)$ and let $\X^{n,\x}$ be the random walk on $(\GG^{\dag, n}, c^{\dag, n})$ started from $\x$. Thanks to the harmonicity of the height coordinate function $\frkh^{\dag}_n$ and to the optional stopping theorem, we have that
\begin{equation*}
\frkh^{\dag}_n(\x)  = \bar{a} \P_\x\l(\X^n_{\tau_\x} \in (\frkh_n^{\dag})^{-1}(\bar{a})\r)  + \underline{a} \P_\x\l(\X^n_{\tau_\x} \in (\frkh_n^{\dag})^{-1}(\underline{a})\r), \quad \forall \x \in \VV\GG^{\dag, n}(R).
\end{equation*}
Therefore, the problem has been reduced to proving that the probabilities appearing in the previous expressions are approximately affine transformations of the height coordinate function in the a priori embedding. By Lemma~\ref{lm_main_height}, for all $n > n(R, \delta)$ large enough, there exist real numbers $b_n' = b_n'(R, \delta)$ and $b_n'' = b_n''(R, \delta)$ for which, for all $\x \in \VV\GG^{\dag, n}(R)$, it holds that
\begin{align*}
& \l|2 R' \frkh^{\dag}_n(\x) - \bar{a} (b_n' + \Im(\x))  + \underline{a} (b_n'' + \Im(\x)\r|  \\
& \qquad \leq \bar{a} \l|2 R' \P_\x\l(\X^n_{\tau_\x} \in (\frkh_n^{\dag})^{-1}(\bar{a})\r)  - b_n'  - \Im(\x)\r| + \underline{a} \l| 2 R' \P_\x\l(\X^n_{\tau_\x} \in (\frkh_n^{\dag})^{-1}(\underline{a})\r) + b_n'' + \Im(\x)\r| \\
& \qquad \leq \l|2 R' \P_\x\l(\X^n_{\tau_\x} \in (\frkh_n^{\dag})^{-1}(\bar{a})\r)  - b_n'  - \Im(\x)\r| + \l| 2 R' \P_\x\l(\X^n_{\tau_\x} \in (\frkh_n^{\dag})^{-1}(\underline{a})\r) + b_n'' + \Im(\x)\r| \\
& \qquad \leq \delta.
\end{align*}
Therefore, rearranging the terms in the above expression and letting
\begin{equation*}
c^{R, \delta}_n :=\frac{2R'}{|\bar{a} -\underline{a}|},  \qquad b^{R, \delta}_n := \frac{\underline{a}b_n''-\bar{a}b_n'}{|\bar{a} -\underline{a}|},
\end{equation*}
we find that, for any $n>n(R, \delta)$ large enough, it holds that
\begin{equation*}
\l|c^{R, \delta}_n \frkh^{\dag}_n(\x) + b^{R, \delta}_n  - \Im(\x)\r| \leq \delta, \quad \forall \x\in \VV\GG^{\dag, n}(R).
\end{equation*}

\medskip
\noindent\textbf{Step 2.} In this second step, we show how we can define real sequences $\{b_n\}_{n \in \N}$ and $\{c_n\}_{n \in \N}$ independent of $R$, $\delta$ such that the conclusion of the proposition holds. To this end, consider an increasing sequence $\{R_k\}_{k \in \N} \subset [1, \infty)$ and a decreasing sequence $\{\delta_k\}_{k \in \N} \subset (0, 1)$ such that $R_k \to \infty$ and $\delta_k \to 0$, as $k \to \infty$. Then, thanks to the previous step, we know that, for all $k \in \N$, there is $n_k = n_k(R_k, \delta_k) \in \N$ such that, for all $n > n_k$, there exist real numbers $c_n^{R_k, \delta_k}$, $b_n^{R_k, \delta_k} \in \R$ such that
\begin{equation*}
\l|c^{R_k, \delta_k}_n \frkh^{\dag}_n(\x) + b^{R_k, \delta_k}_n - \Im(\x)\r| \leq \delta_k, \quad \forall \x \in \VV\GG^{\dag, n}(R_k).
\end{equation*}
Without any loss of generality, we can assume that the sequence $\{n_k\}_{k \in \N}$ is increasing. Then, for all $n \in [0, n_1) \cap \N$,  we let $c^{\frkh}_n := 1$, $b^{\frkh}_n := 1$, and for all $k \in \N$ and $n \in [n_k, n_{k+1}) \cap \N$, we set $c^{\frkh}_n := c^{R_k, \delta_k}_n$, $b^{\frkh}_n := b^{R_k, \delta_k}_n$. Therefore, if we fix $R > 1$ and $\delta \in (0, 1)$, for all $n > n(R, \delta)$ large enough, it holds that 
\begin{equation*}
\l|c^{\frkh}_n \frkh^{\dag}_n(\x) + b^{\frkh}_n - \Im(\x)\r| \leq \delta, \quad \forall \x \in \VV\GG^{\dag, n}(R), 
\end{equation*}
which concludes the proof.
\end{proof}

\subsubsection{Proof of Lemma~\ref{lm_main_height}}
\label{sub_aux_res_height}
The main goal of this subsection is prove the key Lemma~\ref{lm_main_height}. Roughly speaking, the first estimate in Lemma~\ref{lm_main_height} states that, for $n > n(R, \delta)$ large enough, the probability that the lifted random walk started inside $\smash{\VV\GG^{\dag, n}(R)}$ hits the set $\smash{(\frkh_n^{\dag})^{-1}(\bar{a})}$ before hitting $\smash{(\frkh_n^{\dag})^{-1}(\underline{a})}$ depends linearly on the height coordinate of the starting point of the walk on the a priori embedding. In order to prove this result, we need to rule out the possibility that the preimage of a horizontal line on the Smith embedding has large vertical fluctuations (see Figure~\ref{fig_height}). To do so, we use the invariance principle assumption on the sequence of primal graphs, and more precisely we will follow the following steps.
\begin{enumerate}[(a)]
	\item We start by proving that the probability that the lifted random walk started inside $\smash{\VV\GG^{\dag, n}(R)}$ hits the set $\smash{V^{\dag}_{R'}}$ before hitting $\smash{V^{\dag}_{-R'}}$ depends linearly on the height coordinate of the starting point of the walk on the a priori embedding. This follows easily thanks to the invariance principle assumption and it is the content of Lemma~\ref{lm_height_tech1} below.
	\item We then prove that the probability that the random walk started from $\smash{V^{\dag}_{-R'}}$ has probability of order $1/R'$ to hit $\smash{(\frkh_n^{\dag})^{-1}(\bar{a})}$ before hitting $\smash{(\frkh_n^{\dag})^{-1}(\underline{a})}$. Once again, this is an easy consequence of the invariance principle assumption, and it is the content of Lemma~\ref{lm_height_tech2} below. 
	\item Finally, roughly speaking, in order to conclude, we need to improve the bound on the probability appearing in the previous step from order $1/R'$ to order $o_n(1/R')$. This is done by using Lemma~\ref{lm_tot_variation} together with the invariance principle assumption. Indeed, as it will be more clear in the proof of Lemma~\ref{lm_main_height}, it is sufficient to estimate the probability that a random walk started inside $\VV\GG^n(R)$ does not disconnect $V_R \cup V_{-R}$ from $V_{R'} \cup V_{-R'}$ before hitting the latter set. In Lemma~\ref{lm_height_tech3} below, we will see that this probability is of order $R'/R$, and this will be enough to conclude.
\end{enumerate}

\noindent
Before proceeding, we observe that, for all $\x \in  \VV\GG^{\dag, n}(R)$, it holds that
\begin{equation*}
\P_\x\l(\X^n_{\tau_\x} \in \sigma_{2\pi}^{-1}(\frkh_n^{-1}(\underline{a}))\r) = 1 - \P_{\x}\l(\X^n_{\tau_\x} \in (\frkh_n^{\dag})^{-1}(\bar{a})\r),
\end{equation*}
hence, from now on, we can just focus on the first estimate in the statement of Lemma~\ref{lm_main_height}.

\medskip
\noindent
We can now state and prove the technical lemmas mentioned above. We start with the following lemma where we study the probability that the lifted random walk started inside $\smash{\VV\GG^{\dag, n}(R)}$ hits the set $\smash{V^{\dag}_{R'}}$ before hitting $\smash{V^{\dag}_{-R'}}$.
\begin{lemma}
\label{lm_height_tech1}
For any $n > n(R, \delta)$ large enough, it holds that 
\begin{equation*}
\l| 2 R' \P_\x\l(\X^{n}_{\sigma_\x} \in V_{R'}^{\dag}\r) - R' - \Im(\x)\r| \leq \delta, \quad \forall \x \in \VV\GG^{\dag, n}(R).
\end{equation*}	
\end{lemma}
\begin{proof}
For $n \in \N$, fix a vertex $\x \in \VV\GG^{\dag, n}(R)$, consider a planar Brownian motion $B^\x$ started from $\x$, and define the stopping time 
\begin{equation*}
\sigma_{B, \x} := \inf\l\{t \geq 0 \, : \, \l|\Im(B^\x_t)\r| = R'\r\}.
\end{equation*}
Then, by assumption~\ref{it_invariance}, for any $n > n(R, \delta)$ large enough, we have that
\begin{equation*}
\l|\P_\x(\X^{n}_{\sigma_\x} \in \sigma_{2\pi}^{-1}(V_{R'}))  - \P_\x\l(\Im(B_{\sigma_{B, \x}}) = R'\r) \r| \leq \frac{\delta} {2 R'}, \quad\forall \x \in \VV\GG^{\dag, n}(R).
\end{equation*}
Since $\Im(B^\x)$ is just a linear Brownian motion started from $\Im(\x)$, thanks to the gambler's ruin formula, it holds that
\begin{equation*}
\P_\x\l(\Im(B_{\sigma_{B, \x}}) = R'\r) = \frac{R' + \Im(\x)}{2 R'} ,
\end{equation*}
from which the conclusion follows.
\end{proof}

\noindent
We can now move to the second lemma which gives an estimate for the probability that the random walk started from $\smash{V^{\dag}_{-R'}}$ hits $\smash{(\frkh_n^{\dag})^{-1}(\bar{a})}$ before hitting $\smash{(\frkh_n^{\dag})^{-1}(\underline{a})}$.
\begin{lemma}
\label{lm_height_tech2}
For any $n > n(R, \delta)$ large enough, it holds that
\begin{equation*}
\P_\x\l(\X^n_{\tau_\x} \in (\frkh_n^\dag)^{-1}(\bar{a})\r) \lesssim \frac{1}{R'}, \quad \forall \x \in V_{-R'}^{\dag},
\end{equation*}
where the implicit constant is independent of everything else.
\end{lemma}
\begin{proof}
We start by noticing that, for all $\x \in V_{-R'}^{\dag}$, it holds that
\begin{equation*}
\P_\x\l(\X^n_{\tau_\x} \in (\frkh_n^\dag)^{-1}(\bar{a})\r) \leq \P_\x\l(\sigma_{2\pi}(\X^n|_{[0, \tau_\x]})  \text{ does not wind around the cylinder below height } -R'\r), 
\end{equation*}
where we recall that $\sigma_{2\pi}$ is defined in \eqref{eq_covering_map} and denotes the covering map of the infinite cylinder $\CC_{2\pi}$. The above inequality is due to the fact that, if $\sigma_{2\pi}(\X^n|_{[0, \tau_\x]})$ winds around the cylinder below height $-R'$, then, by definition of $\underline{a}$, it has to hit the set $\frkh^{-1}_n(\underline{a})$. We can now exploit assumption~\ref{it_invariance} and find the corresponding upper bound for the Brownian motion. More precisely, let $B^{\x}$ be a planar Brownian motion started from $\x \in \sigma_{2\pi}^{-1}(V_{-R'})$ and define the stopping time 
\begin{equation*}
\tau_{B, \x} := \inf\l\{t \geq 0 \, : \, \Im(B^\x_t) = -2R' \text{ or } \Im(B^\x_t) = R'\r\}.
\end{equation*}
Then, for any $n > n(R, \delta)$ large enough, we have that
\begin{align*}
& \P_\x\l(\sigma_{2\pi}(\X^n|_{[0, \tau_\x]})  \text{ does not wind around the cylinder below height } -R'\r) \\
& \qquad \leq 2 \P_{\x}\l(\sigma_{2\pi}(B|_{[0, \tau_{B, \x}]}) \text{ does not wind around the cylinder below height } -R'\r).
\end{align*} 
Therefore, to conclude, it is sufficient to find a uniform upper bound for the quantity on the right-hand side of the above expression. This is done in Lemma~\ref{lm_std_BM_1} from which the conclusion follows.
\end{proof}

\noindent
In order to prove Lemma~\ref{lm_main_height}, we also need the following lemma which provides an estimate for the probability that a random walk started inside $\VV\GG^n(R)$ disconnects $V_R \cup V_{-R}$ from $V_{R'} \cup V_{-R'}$ before hitting the latter set.
\begin{lemma}
\label{lm_height_tech3}
For any $n > n(R, \delta)$ large enough, it holds that
\begin{equation*}
\P_\x\l(\sigma_{2\pi}(\X^n|_{[0, \sigma_\x]})  \text{ does not disconnect } V_{R} \cup V_{-R} \text{ from } V_{R'} \cup V_{-R'}\r) \lesssim \frac{R}{R'}, \quad \forall \x \in \VV\GG^{\dag, n}(R),
\end{equation*}
where the implicit constant is independent of everything else.
\end{lemma}
\begin{proof}
For $\x \in \VV\GG^{\dag, n}(R)$, let $B^\x$ be a planar Brownian motion started from $\x$, and define the stopping time 
\begin{equation*}
\sigma_{B, \x} : = \inf \l\{t \geq 0 \, : \, |\Im(B_t^\x)| = R'\r\}.
\end{equation*}
By assumption~\ref{it_invariance}, we know that for any $n > n(R, \delta)$ large enough, it holds that
\begin{align*}
& \P_\x\l(\sigma_{2\pi}(\X^n|_{[0, \sigma_\x]})  \text{ does not disconnect } V_{R} \cup V_{-R} \text{ from } V_{R'} \cup V_{-R'}\r) \\
& \qquad \leq 2 \P_{\x}\l(\sigma_{2\pi}(B|_{[0, \sigma_{B, \x}]}) \text{ does not disconnect } \R/2\pi\Z \times \{-R, R\} \text{ from } \R/\Z \times \{-R', R'\}\r).
\end{align*}
Therefore, it is sufficient to find a uniform upper bound for the quantity on the right-hand side of the above expression. This is the content of Lemma~\ref{lm_std_BM_2} from which the desired result follows.
\end{proof}

\noindent
We are now ready to give a proof of Lemma~\ref{lm_main_height}. As we have already remarked, this will be a consequence of the previous three lemmas and of Lemma~\ref{lm_tot_variation}.
\begin{proof}[Proof of Lemma~\ref{lm_main_height}]
We start by defining the function $\frkf^{\dag}_n : \VV\GG^{\dag, n}(R') \to \R$ as follows
\begin{equation*}
\frkf^{\dag}_n(\x) : = \P_\x\l(\X^n_{\tau_\x} \in  (\frkh_n^{\dag})^{-1}(\bar{a})\r) - \P_\x\l(\X^n_{\sigma_\x} \in V_{R'}^{\dag}\r), \qquad  \forall \x \in \VV\GG^{\dag, n}(R'),
\end{equation*}
so that, we can write
\begin{equation}
\label{eq_deco_height_coordinate}
\P_\x\l(\X^n_{\tau_\x} \in  (\frkh_n^{\dag})^{-1}(\bar{a})\r) = \P_\x\l(\X^n_{\sigma_\x} \in V_{R'}^{\dag}\r) + \frkf^{\dag}_n(\x), \qquad  \forall \x \in \VV\GG^{\dag, n}(R').
\end{equation}
We now observe that, thanks to Lemma~\ref{lm_height_tech1}, for any $n > n(R, \delta)$ large enough, it holds that
\begin{equation}
\label{eq_bound_height_coordinate_1}
\l|2 R' \P_\x\l(\X^n_{\sigma_\x} \in  V_{R'}^{\dag}\r) - R' - \Im(\x)\r| \leq \delta,  \quad \forall \x \in \VV\GG^{\dag, n}(R).
\end{equation} 
Therefore, it is sufficient to study the function $\smash{\frkf^{\dag}_n}$ appearing in \eqref{eq_deco_height_coordinate}. To this end, we consider the functions $\smash{\bar{\frkf}^{\dag}_n : \VV\GG^{\dag, n}(R') \to [0,1]}$ and $\smash{\underline{\frkf}^{\dag}_n : \VV\GG^{\dag, n}(R') \to [0,1]}$ defined as follows
\begin{equation*}
\bar{\frkf}^{\dag}_n(\x) := \P_\x\l(\X^n_{\sigma_\x} \in V_{-R'}^{\dag}, \, \X^n_{\tau_\x} \in  (\frkh_n^{\dag})^{-1}(\bar{a})\r), \qquad \underline{\frkf}^{\dag}_n(\x) := \P_\x\l(\X^n_{\sigma_\x} \in V_{R'}^{\dag}, \, \X^n_{\tau_\x} \in (\frkh_n^{\dag})^{-1}(\underline{a})\r).
\end{equation*}
In particular, as one can easily check, it holds that 
\begin{equation*}
 \l|\frkf_n^{\dag}(\x)\r| \leq \bar{\frkf}_n^{\dag}(\x) + \underline{\frkf}_n^{\dag}(\x), \quad \forall \x \in \VV\GG^{\dag, n}(R'),
\end{equation*}
and so, we can reduce the problem to the study of the functions $\smash{\bar{\frkf}^{\dag}_n}$ and $\smash{\underline{\frkf}^{\dag}_n}$. We will only study the function $\smash{\bar{\frkf}^{\dag}_n}$ as the function $\smash{\underline{\frkf}^{\dag}_n}$ can be treated similarly. Thanks to the strong Markov property of the random walk $\X^{n, \x}$, we have that 
\begin{equation*}
\bar{\frkf}^{\dag}_n(\x) = \E_\x\l[\bar{\frkf}^{\dag}_n(\X^n_{\sigma_\x})\r], \quad \forall \x \in \VV\GG^{\dag, n}(R').
\end{equation*}
 Therefore, for $\x$, $\y \in \VV\GG^{\dag, n}(R)$, it holds that 
\begin{equation}
\label{eq_bound_frkf1}
\l|\bar{\frkf}^{\dag}_n(\x) - \bar{\frkf}^{\dag}_n(\y)\r| \leq \sup \l\{\l|\bar{\frkf}^{\dag}_n(\v)\r| \, : \, \v \in V_{-R'}^{\dag} \r\} \dTV\l(\X^{n, \x}_{\sigma_\x}, \X^{n, \y}_{\sigma_\y}\r),
\end{equation}
where $\dTV$ denotes the total variation distance.
Hence, it is sufficient to find an upper bound for the two terms on the right-hand side of \eqref{eq_bound_frkf1}. We treat the two factors separately. 
\begin{itemize}
\item In order to bound the first factor, we just need to bound uniformly on $\smash{\v \in V_{-R'}^{\dag}}$ the probability that a random walk on $\smash{(\GG^{\dag, n}, c^{\dag, n})}$ started from $\v$ hits $\smash{(\frkh_n^{\dag})^{-1}(\bar{a})}$ before hitting $\smash{(\frkh_n^{\dag})^{-1}(\underline{a})}$. This is exactly the content of Lemma~\ref{lm_height_tech2} from which we can deduce that, for all $n > n(R, \delta)$ large enough, it holds that 
\begin{equation*}	
\sup \l\{\l|\bar{\frkf}^{\dag}_n(\v)\r| \, : \, \v \in V_{-R'}^{\dag} \r\} \lesssim \frac{1}{R'},
\end{equation*}
where the implicit constant is universal.
\item In order to bound the second factor, we can use Lemma~\ref{lm_tot_variation}. More precisely, it is sufficient to estimate the probability that $\sigma_{2\pi}(\X^{n,\x}|_{[0, \sigma_\x]})$ disconnects $V_{R} \cup V_{-R}$ from $V_{R'} \cup V_{-R'}$. This is exactly the content of Lemma~\ref{lm_height_tech3} which guarantees that, for all $n > n(R, \delta)$ large enough and for all $\x \in \VV\GG^{\dag, n}(R)$, it holds that
\begin{equation*}
\P_\x\l(\sigma_{2\pi}(\X^{n}|_{[0, \sigma_\x]})  \text{ does not disconnect } V_{R} \cup V_{-R} \text{ from } V_{R'} \cup V_{-R'}\r) \lesssim \frac{R}{R'},
\end{equation*}
where the implicit constant is independent of everything else. Therefore, this fact together with Lemma~\ref{lm_tot_variation} imply that 
\begin{equation*}
\dTV\l(\X^{n, \x}_{\sigma_\x}, \X^{n, \y}_{\sigma_\y}\r) \lesssim \frac{R}{R'}, \quad \forall \x, \y \in \VV\GG^{\dag, n}(R). 
\end{equation*}
\end{itemize}
Therefore, putting together the previous two bullet points and going back to \eqref{eq_bound_frkf1}, we get that for every $n > n(R, \delta)$ large enough, it holds that
\begin{equation*}
\l|\bar{\frkf}^{\dag}_n(\x) - \bar{\frkf}^{\dag}_n(\y)\r| \lesssim \frac{R}{{R'}^2}, \quad \forall \x, \y \in \VV\GG^{\dag, n}(R),
\end{equation*}
Furthermore, the same uniform bound can be also obtained for the function $\underline{\frkf}^{\dag}_n$, but we omit the details since the argument is similar. Summing up, we obtained that, for all $n  > n(R, \delta)$ large enough, it holds that
\begin{equation*}
\l|\frkf^{\dag}_n(\x) - \frkf^{\dag}_n(\y)\r| \lesssim \frac{R}{{R'}^2}, \quad \forall \x, \,\y \in \VV\GG^{\dag, n}(R).	
\end{equation*}
Hence, to conclude, we can simply proceed as follows. For every $n > n(R, \delta)$ large enough, fix an arbitrary vertex $\y \in \VV\GG^{\dag, n}(R)$. Then, recalling that by definition $R' = R/\delta$, it holds that 
\begin{equation*}
2 R' |\frkf^{\dag}_n(\x) -\frkf^{\dag}_n(\y)| \lesssim \delta, \quad  \forall \x \in \VV\GG^{\dag, n}(R).
\end{equation*}
Therefore, thanks to \eqref{eq_deco_height_coordinate} and estimate \eqref{eq_bound_height_coordinate_1}, we find that, for any $n > n(R, \delta)$ large enough, it holds that
\begin{equation*}
\l|2 R' \P_\x\l(\X^n_{\tau_\x} \in (\frkh_n^{\dag})^{-1}(\bar{a})\r) - b'_n - \Im(\x)\r| \leq \delta, \quad \forall \x \in \VV\GG^{\dag, n}(R),
\end{equation*}
where $ b'_n := 2 R' \frkf^{\dag}_n(\y) + R'	$.
\end{proof}

%%%%%%%%%%%%%%%%%%%%%%%%%%%%%%%%%%%%%%%%%%
\subsection{Width coordinate function}
\label{sub_width}
For each $n \in \N$, consider the lifted width coordinate function 
\begin{equation*}
\frkw^{\dag}_n:\VV\hat{\GG}^{\dag, n} \to \R	,
\end{equation*}
as defined in \eqref{eq_conj_harmo_lift}.
The main goal of this subsection is to show that there exists an affine transformation of the function $\frkw^{\dag}_n$ that is asymptotically close, uniformly on lifts of compact subsets of the infinite cylinder $\CC_{2\pi}$, to the width coordinate function $\hx \mapsto \Re(\hx)$ in the a priori lifted embedded graph $\GG^{\dag, n}$, as $n \to \infty$. More precisely, we have the following result.

\begin{proposition}
\label{pr_main_witdh}
There exits a sequence $\{b^{\frkw}_n\}_{n \in \N} \subset \R$ such that for all $R > 1$, $\delta \in (0, 1)$, and for any $n > n(R, \delta)$ large enough, it holds that
\begin{equation}
\label{eq_width_prop}
\left|\frac{2\pi}{\eta_n} \frkw^{\dag}_n(\hx) + b^{\frkw}_n - \Re(\hx)\right| \leq \delta, \quad \forall \hx \in\VV\hat{\GG}^{\dag, n}(\R \times [-R, R]),
\end{equation}
where we recall that $\eta_n$ denotes the strength of the flow induced by $\frkh_n$ as defined in \eqref{eq_flow_strength}.
\end{proposition}

\noindent
The proof of this proposition is given in Subsection~\ref{subsub_proof_main_width}. As we will see, the proof is based on Lemma~\ref{lm_order_one_width} below and the following two facts: (a) the harmonicity of the lifted width coordinate function $\frkw^{\dag}_n$; (b) the invariance principle assumption \ref{it_invariance_dual} on the sequence of dual maps.

\subsubsection{Setup of the proof}
For each $n \in \N$, consider the metric graphs $\G^n$ and $\hat{\G}^n$ associated to $\GG^n$ and $\hat{\GG}^n$, respectively. Let $\frkh_n : \G^n \to [0, 1]$ be the extended height coordinate function as specified in Remark~\ref{rm_extened_harm}, and $\frkw_n : \hat{\G}^n \to \R$ be the extended width coordinate function as specified in Remark~\ref{rm_extened_discrte_harm}. Given a value $S \in \R$, we define the set  
\begin{equation*}
\hat{V}^n_S := \l\{\hat{x} \in \hat{\G}^n \, : \, \Im(\hat{x}) = S\r\}.
\end{equation*}
Given a value $a \in (0, 1)$, we recall that the sets $\smash{\EE\GG_a^{\dag}}$, $\smash{\EE\hat{\GG}^{\dag}_a}$, $\smash{\VV\GG_a^{\dag}}$, $\smash{\VV\hat{\GG}_a^{\dag}}$ are all defined in Subsection~\ref{sub_periodicity}. We also recall that $\smash{\hx^{n,a}_0}$ denotes the vertex in $\smash{\VV\hat{\GG}^{\dag, n}}$ whose real coordinate is nearest to $0$, and we refer to \eqref{eq_imp_dual} and \eqref{eq_imp_primal} for the definitions of the sets $\smash{\hat{\PPP}^{\dag, n}_{a}}$ and $\smash{\PPP^{\dag, n}_{a}}$, respectively.

\medskip
\noindent
We fix throughout $R > 1$ and $\delta \in (0, 1)$ as in the proposition statement, and we let 
\begin{equation*}
R' := \frac{R}{\delta}.
\end{equation*}
For each $n \in \N$, we consider $\bar{a}_{R'}^n$ and $\underline{a}_{-R'}^n$ as defined in \eqref{eq_key_values_height}. Moreover, thanks to Proposition~\ref{pr_main_height}, for any $n > n(R, \delta)$ large enough, we can choose a value $a^n_{R'}$ such that the set $\frkh_n^{-1}(a^n_{R'})$ is fully contained in the infinite strip $\R \times [-1,1]$.

\medskip
\noindent
In what follows, in order to lighten up notation, we adopt the following notational convention 
\begin{equation*}
	\hat{V}_{S} \equiv \hat{V}^n_{S}, \qquad \bar{a} \equiv \bar{a}^n_{R'}, \quad \underline{a} \equiv \underline{a}^n_{-R'}, \quad a \equiv a^n_{R'} .
\end{equation*}
Furthermore, we denote by $\smash{\hat{V}^{\dag}_S}$ the lift to the universal cover of $\smash{\hat{V}_S}$, and we write $\smash{\VV\hat{\GG}^{\dag, n}(S)}$ (resp.\ $\smash{\VV\hat{\GG}^{n}(S)}$) as a shorthand for $\smash{\VV\hat{\GG}^{\dag, n}(\R \times [-S, S])}$ (resp.\ $\smash{\VV\hat{\GG}^{n}(\R/2\pi\Z \times [-S, S])}$). 

\medskip
\noindent
We also adopt the convention to fix the additive constant of the lifted width coordinate function $\frkw^{\dag}_n$ by setting it to zero on the point $\hx^{n, a}_0$. We observe that changing the base vertex of the lifted width coordinate function has only the effect of changing the value of $b_n^{\frkw}$ appearing in the statement of Proposition~\ref{pr_main_witdh}.

\medskip
\noindent
Similarly to what we discussed in the preceding subsection,  by Lemma~\ref{lm_harm_coincides} and Remark~\ref{rem_width_coinc}, we can assume, without loss of generality, that $\VV\GG^n$ contains all the points in the set
\begin{equation*}
\frkh_n^{\dag}(\bar{a}) \cup \frkh_n^{-1}(\underline{a}) \cup \frkh_n^{-1}(a) \subset \G^n,
\end{equation*}
and that $\VV\hat{\GG}^n$ contains all the points in the set
\begin{equation*}
\hat{V}_{R'-1} \cup \hat{V}_{-R'+ 1} \subset \hat{\G}^n.
\end{equation*}
This is allowed since, as in the case of the height coordinate function, by possibly locally modifying the a priori embedding of the dual graph $\hat{\GG}^n$ in $\CC_{2\pi}$, we can assume that each edge in $\EE\hat{\GG}^n$ crosses the circles at height $R'-1$ and $-R'+1$ at most finitely many times. We refer to Figure~\ref{fig_width} for an illustration of the sets involved in the proof of Proposition~\ref{pr_main_witdh}.

\begin{figure}[h]
\centering
\includegraphics[scale=1]{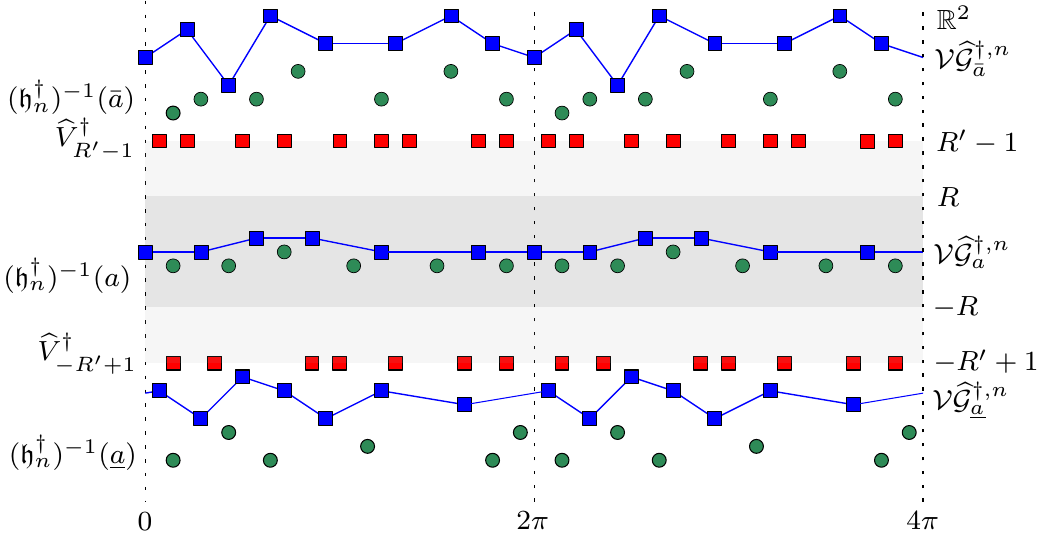}
\caption[short form]{\small A diagram illustrating the sets involved in the proof of Proposition~\ref{pr_main_witdh}.   
The shaded dark-gray strip between heights $R$ and $-R$ contains all the vertices in the set $\VV\hat{\GG}^{\dag, n}(R)$. The shaded light-grey strip between heights $R'-1$ and $-R'+1$ contains all the vertices in the set $\smash{\VV\hat{\GG}^{\dag, n}(R')}$. The red dual vertices at the top (resp.\ bottom) are the vertices in the set $\smash{\hat{V}^{\dag}_{R'-1}}$ (resp.\ $\smash{\hat{V}^{\dag}_{-R'+1}}$). Moving from top to bottom, the blue vertices are the dual vertices in the sets $\smash{\VV\hat{\GG}^{\dag, n}_{\bar{a}}}$, $\smash{\VV\hat{\GG}^{\dag, n}_{a}}$, $\smash{\VV\hat{\GG}^{\dag, n}_{\underline{a}}}$; the blue dual edges are the edges in the sets $\smash{\EE\hat{\GG}^{\dag, n}_{\bar{a}}}$, $\smash{\EE\hat{\GG}^{\dag, n}_{a}}$,$\smash{\EE\hat{\GG}^{\dag, n}_{\underline{a}}}$; and finally the green vertices are the vertices in the sets $\smash{(\frkh_n^{\dag})^{-1}(\bar{a})}$, $\smash{(\frkh_n^{\dag})^{-1}(a)}$, $\smash{(\frkh_n^{\dag})^{-1}(\underline{a})}$.}
\label{fig_width}
\end{figure}

\begin{remark}
\label{rem_stays_narrow_band}
We observe that, by Proposition~\ref{pr_main_height}, for any $\eps \in (0, 1)$ and for any $n > n(R, \delta, \eps)$ large enough, it holds that
\begin{equation*}
(\frkh_n^{\dag})^{-1}(\bar{a}) \subset \R \times [R' - \eps, R' + \eps], \qquad (\frkh_n^{\dag})^{-1}(\underline{a}) \subset \R \times [-R'-\eps, -R'+\eps].
\end{equation*}
Furthermore, by definitions of $\EE\hat{\GG}^{\dag, n}_{\bar{a}}$, $\EE\hat{\GG}^{\dag, n}_{\underline{a}}$ and Lemma~\ref{lm_no_macroscopic_edges}, we have that, for any $\eps \in (0, 1)$ and for any $n > n(R, \delta, \eps)$ large enough, it holds that
\begin{equation*}
\EE\hat{\GG}^{\dag, n}_{\bar{a}} \subset \R \times [R' - \eps, R' + \eps], \qquad \EE\hat{\GG}^{\dag, n}_{\underline{a}} \subset \R \times [-R'-\eps, -R'+\eps].
\end{equation*}
These facts will be of key importance in the remaining part of this subsection and they will be used several times.
\end{remark}

\paragraph{Random walk notation.}
For $\smash{\hx \in \VV\hat{\GG}^{\dag, n}(R'-1)}$, we consider the continuous time  random walk $\smash{\{\hat{\X}_t^{n, \hx}\}_{t \geq 0}}$ on the lifted weighted dual graph $\smash{(\hat{\GG}^{\dag, n}, \hat{c}^{\dag, n})}$ started from $\hx$. We recall that the continuous path of this random walk is simply generated by piecewise linear interpolation at constant speed. We consider the following stopping times 
\begin{equation}
\label{eq_stopping_times_dual}
\sigma_{\hx} := \inf\l\{t \geq 0  \, : \, \hat{\X}^{n, \hx}_t \in \hat{V}^{\dag}_{R'-1} \cup \hat{V}^{\dag}_{-R'+1}\r\}, \quad \tau_{\hx} := \inf\l\{t \geq 0  \, : \, \hat{\X}^{n, \hx}_t \in \VV\hat{\GG}^{\dag, n}_{\bar{a}} \cup \VV\hat{\GG}^{\dag, n}_{\underline{a}}\r\}.
\end{equation}
As we will observe more precisely below, for $n > n(R, \delta)$ large enough, thanks to Proposition~\ref{pr_main_height}, we have that
\begin{equation*}
\VV\hat{\GG}^{\dag, n}_{\bar{a}} \subset \R \times [R'-1, R'+1], \qquad \VV\hat{\GG}^{\dag, n}_{\underline{a}} \subset \R \times [-R'-1, -R'+1],
\end{equation*}
and so, it holds that $\smash{\tau_{\hx} \geq \sigma_{\hx}}$, for all $\smash{\hx \in \VV\hat{\GG}^{\dag, n}(R'-1)}$. Looking at Figure~\ref{fig_width}, the stopping time $\smash{\sigma_{\hx}}$ accounts for the first time at which the random walk hits one of the red dual vertices, while the stopping time $\smash{\tau_{\hx}}$ accounts for the first time at which the random walk hits one of the blue dual vertices at the top or bottom of the figure.

\medskip
\noindent
In what follows, we also need to consider random walks on the sequence of primal lifted weighted graphs. To be precise, let $a$ be as specified in the introduction of this section (we recall that here, $a$ is a shorthand for $a_{R'}^n$) and consider the probability measure $\mu_a$ on the set $\smash{\frkh^{-1}(a)}$ as specified in Definition~\ref{def_measure_exit}. We let $\smash{X^{n, \mu_a}}$ be the random walk on the weighted graph $(\GG^n, c^n)$ started from a point in $\smash{\frkh^{-1}_n(a)}$ sampled according to $\mu_{a}$. We also consider the associated continuous time lifted random walk $\smash{\{\X^{n, \mu_a}_t\}_{t \geq 0}}$ on $\smash{(\GG^{\dag, n}, c^{\dag, n})}$ started from a point in $\smash{\PPP^{\dag, n}_{a}}$ sampled according to $\mu_{a}$. We define the following stopping times
\begin{equation}
\label{eq_stopping_times_dual_primal}
\theta_+ := \inf \l\{t \geq 0  \, : \, \X^{n, \mu_a}_t \in (\frkh_n^{\dag})^{-1}(\bar{a})\r\} , \quad \theta_- := \inf \l\{t \geq 0  \, : \, \X^{n,  \mu_a}_t \in (\frkh_n^{\dag})^{-1}(\underline{a})\r\}.
\end{equation}
Furthermore, we will also need to consider the lifted Smith-embedded random walk $\smash{\{\dotX_t^{n ,\mu_a}\}_{t \geq 0}}$ associated to $\smash{\X^{n ,\mu_a}}$ defined in Subsection~\ref{subsec_winding}.

\subsubsection{Proof of Proposition~\ref{pr_main_witdh}}
\label{subsub_proof_main_width}
We can now state a key lemma for the proof of Proposition~\ref{pr_main_witdh}. The proof of the below result is given in Subsection~\ref{sub_aux_res_width}.
\begin{lemma}
\label{lm_order_one_width}
For any $n > n(R, \delta)$ large enough there is a finite constant $b'_n \in \R$ such that
\begin{equation*}
\left|\E_{\hx}\left[\frac{\Re(\hat{\X}^{n}_{\tau_{\hat{x}}})}{2 \pi} - \frac{\frkw_n^{\dag}(\hat{\X}^{n}_{\tau_{\hat{x}}}) + b'_n}{\eta_n} \right] \right| \lesssim 1, \quad  \forall \hx \in \VV\hat{\GG}^{\dag, n}(R'-1),
\end{equation*}
where the implicit constant is universal, and $\eta_n$ denotes the strength of the flow induced by $\frkh_n$ as defined in \eqref{eq_flow_strength}.
\end{lemma}

\noindent
Given Lemma~\ref{lm_order_one_width}, the proof of Proposition~\ref{pr_main_witdh} follows by using the harmonicity of the width coordinate function and the invariance principle assumption \ref{it_invariance_dual}. 

\begin{proof}[Proof of Proposition~\ref{pr_main_witdh}]
For $\smash{\hx \in \VV\hat{\GG}^{\dag, n}(R)}$, let $\smash{\hat{\X}^{n, \hx}}$ be the random walk on the lifted dual weighted graph $\smash{(\hat{\GG}^{\dag, n}, \hat{c}^{\dag, n})}$ started from $\hx$. Moreover, let $B^{\hx}$ be a planar Brownian motion started from $\hx$, and define the stopping time
\begin{equation}
\label{eq_dual_BM_stopping}
\tau_{B, \hx} := \inf\l\{t \geq 0 \, : \, |\Im(B_t^{\hx})| = R'\r\}.
\end{equation}
We divide the proof into several steps.

\medskip
\noindent\textbf{Step~1.} 
In this first step, we show that, for any $n > n(R, \delta)$ large enough, it holds that
\begin{equation}
\label{eq_width_step1_main}
\l|\E_{\hat{\x}}\l[\Re(\hat{\X}^{n}_{\tau_{\hx}})\r] - \Re(\hat{\x})\r| \leq \delta,
\end{equation}
where we recall that the stopping time $\tau_{\hx}$ is defined in \eqref{eq_stopping_times_dual}. As we will see below, this is an easy consequence of assumption~\ref{it_invariance_dual}.
We start by observing that, thanks to well-known properties of Brownian motion, it holds that
\begin{equation*}
\l|\E_{\hx}[\Re(B_{\tau_{B, \hx}})] - \Re(\hx)\r| = 0, \quad \forall \hx \in \VV\GG^{\dag, n}(R).
\end{equation*}
Indeed, this follows from the fact that $\smash{|\Re(B_{\tau_{B, \hx}})] - \Re(\hx)|}$ has exponentially decaying tails and from the optional stopping theorem. As we observed in Remark~\ref{rem_stays_narrow_band}, we have that, for any $\eps \in (0, 1)$ and for any $n > n(R, \delta, \eps)$ large enough, it holds that
\begin{equation*}
\EE\hat{\GG}^{\dag, n}_{\bar{a}} \subset \R \times [R' - \eps, R' + \eps], \qquad \EE\hat{\GG}^{\dag, n}_{\underline{a}} \subset \R \times [-R'-\eps, -R'+\eps].
\end{equation*}
Therefore, from this fact and assumption~\ref{it_invariance_dual}, we can deduce that, for any $n > n(R, \delta)$ large enough, it holds that
\begin{equation*}
\l|\E_{\hx}\l[\Re(B_{\tau_{B, \hx}})\r]- \E_{\hat{\x}}\l[\Re(\hat{\X}^{n}_{\tau_{\hat{x}}})\r]\r| \leq \delta.
\end{equation*}
More precisely, this fact can be obtained from assumption~\ref{it_invariance_dual} by arguing in the same exact way as in the proof of Lemma~\ref{lm_wind_apriori_walk} below. 

\medskip
\noindent\textbf{Step~2.} The main goal of this step is to prove that, for any $n > n(R, \delta)$ large enough, it holds that
\begin{equation}
\label{eq_width_step2_main}
\E_{\hat{\x}}\l[\frkw_n^{\dag}\l(\hat{\X}^{n}_{\hat{\tau}_{\hx}}\r)\r] = \frkw_n^{\dag}(\hat{\x}), \quad \forall \hx \in \VV\GG^{\dag, n}(R).
\end{equation}
We start by recalling that the function $\smash{\frkw_n^{\dag} : \VV\hat{\GG}^{\dag, n} \to \R}$ is harmonic, and so the process $\smash{\frkw_n^{\dag}(\hat{\X}^{n, \hx})}$ is a discrete martingale with respect to the filtration generated by $\smash{\hat{\X}^{n, \hx}}$. Therefore, if we prove that this martingale is uniformly integrable, then the claim follows from the optional stopping theorem. To this end, it is sufficient to prove that, for $M \in \N$, the probability of the event $\smash{|\frkw_n^{\dag}(\hat{\X}^{n, \hx}_{\hat{\tau}_{\hx}}) - \frkw_n^{\dag}(\hx)| > M R'}$ decays exponentially fast in $M$, uniformly in $\smash{\hx \in \VV\hat{\GG}^{\dag, n}(R)}$ and for all $n > n(R, \delta)$ large enough. This fact can be obtained from assumption~\ref{it_invariance_dual} by arguing in the same exact way as in the proof of Lemma~\ref{lm_unif_integr} below. Hence, the desired result follows.

\medskip
\noindent\textbf{Step~3.} Consider the function $\frkf^{\dag}_n : \VV\hat{\GG}^{\dag, n}(R'-1) \to \R$ defined as follows
\begin{equation*}
\frkf^{\dag}_n(\hat{\x}) : = \E_{\hat{\x}} \left[\frac{\Re(\hat{\X}^{n}_{\tau_{\hx}})}{2 \pi} - \frac{\frkw_n^{\dag}(\hat{\X}^{n}_{\tau_{\hx}}) + b'_n}{\eta_n} \right] , \quad \forall \hat{\x} \in \VV\hat{\GG}^{\dag, n}(R'-1),
\end{equation*}
where $b'_n$ is the same constant appearing in the statement of Lemma~\ref{lm_order_one_width}.
Now, recalling the definition \eqref{eq_stopping_times_dual} of the stopping time $\sigma_{\hx}$ and that $\sigma_{\hat{\x}}\leq \tau_{\hat{\x}}$, thanks to the strong Markov property of the random walk, for all $\hx$, $\hat{\y} \in \VV\hat{\GG}^{\dag, n}(R)$, it holds that
\begin{equation}
\label{eq_tot_var_dist_dual}
\l|\frkf^{\dag}_n(\hat{\x}) - \frkf^{\dag}_n(\hat{\y})\r| \leq \sup \l\{\l|\frkf^{\dag}_n(\hat{\v})\r|\, : \, {\hat{\v} \in \VV\hat{\GG}^{\dag, n}(R'- 1)}\r\} \dTV\l(\hat{\X}^{n, \hx}_{\sigma_{\hx}}, \hat{\X}^{n, \hy}_{\sigma_{\hy}}\r),
\end{equation}
where $\dTV$ denotes the total variation distance.
Hence, it is sufficent to find an upper bound for the two terms on the right-hand side of \eqref{eq_tot_var_dist_dual}. We treat the two factors separately.
\begin{itemize}
\item In order to bound the first factor, we just need to bound uniformly on $\smash{\VV\hat{\GG}^{\dag, n}(R'-1)}$ the quantity $\smash{|\frkf^{\dag}_n(\v)|}$. This is exactly the content of Lemma~\ref{lm_order_one_width} from which we can deduce that, for all $n > n(R, \delta)$ large enough, it holds that
\begin{equation*}
\sup \l\{\l|\frkf^{\dag}_n(\hat{\v})\r|\, : \, {\hat{\v} \in \VV\hat{\GG}^{\dag, n}(R'- 1)}\r\} \lesssim 1,
\end{equation*}
where the implicit constant is independent of everything else.
\item In order to bound the second factor, we can use Lemma~\ref{lm_tot_variation}. Indeed, as we have already remarked, thanks to Proposition~\ref{pr_main_height} and to Lemma~\ref{lm_no_macroscopic_edges}, for any $n \in \N$ large enough, it holds that
\begin{equation*}
\VV\hat{\GG}^{\dag, n}_{\bar{a}} \subset \R \times \l[R'-1, R'+1\r], \qquad \VV\hat{\GG}^{\dag, n}_{\underline{a}} \subset \R \times \l[-R'-1, -R'+1\r].
\end{equation*}
Therefore, it is sufficient to estimate the probability that $\smash{\sigma_{2\pi}(\hat{\X}^{n, \hat{\x}}|_{[0, \sigma_{\hx}]})}$ disconnects $\smash{\hat{V}_{R} \cup \hat{V}_{-R}}$ from $\smash{\hat{V}_{R'-1} \cup \hat{V}_{-R'+1}}$. To this end, one can argue in the same exact way as in Lemma~\ref{lm_height_tech3} in order to prove that, for any $n > n(R, \delta)$ large enough and for all $\smash{\hx \in \VV\hat{\GG}^{\dag, n}(R)}$, it holds that
\begin{equation*}
\P_{\hx}(\sigma_{2\pi}(\hat{\X}^{n}|_{[0, \hat{\sigma}_{\hx}]}) \text{ does not disconnect } \hat{V}_{R} \cup \hat{V}_{-R} \text{ from } \hat{V}_{R' - 1} \cup \hat{V}_{-R' + 1})  \lesssim \frac{R}{R'},
\end{equation*}
where the implicit constant is independent of everything else. Therefore, this fact together with Lemma~\ref{lm_tot_variation} imply that 
\begin{equation*}
\dTV\l(\hat{\X}^{n, \hat{x}}_{\hat{\sigma}_{\hat{\x}}}, \hat{\X}^{n, \hat{\y}}_{\hat{\sigma}_{\hat{y}}}\r) \lesssim \frac{R}{R'}, \quad \forall \hx, \hy \in \VV\hat{\GG}^{\dag, n}(R).
\end{equation*}
\end{itemize}
Therefore, putting together the two bullet points above, recalling that $R'=R/\delta$, and going back to \eqref{eq_tot_var_dist_dual}, we find that, for any $n>n(R, \delta)$ large enough, it holds that
\begin{equation}
\label{eq_dtv_main_width}
\l|\frkf^{\dag}_n(\hat{\x}) - \frkf^{\dag}_n(\hat{\y})\r| \lesssim \delta, \quad \forall \hat{\x}, \, \hat{\y} \in \VV\hat{\GG}^{\dag, n}(R),
\end{equation}
where the implicit constant is independent of everything else.

\medskip
\noindent\textbf{Step~4.} 
To conclude, for every $n > n(R, \delta)$ large enough, fix an arbitrary vertex $\hy \in \VV\hat{\GG}^{\dag, n}(R)$. Then, thanks to \eqref{eq_width_step1_main}, \eqref{eq_width_step2_main}, and \eqref{eq_dtv_main_width}, we have that for any $n > n(R, \delta)$ large enough, it holds that
\begin{equation*}
\left|\frac{2\pi}{\eta_n}\frkw_n^{\dag}(\hat{\x}) + b^{R, \delta}_n  - \Re(\hat{\x})\right| \leq \delta, \quad \forall \hx \in \VV\hat{\GG}^{\dag, n}(R),  
\end{equation*}
where $b^{R, \delta}_n := 2\pi \frkf^{\dag}_n(\hat{\y})$. Finally, in order to conclude, we need to remove the dependence of $b^{R, \delta}_n$ from $R$ and $\delta$. This can be easily done by arguing in the same exact way as in the second step of the proof of Proposition~\ref{pr_main_height}. Therefore, the proof is concluded. 
\end{proof}

\subsubsection{Proof of Lemma~\ref{lm_order_one_width}}
\label{sub_aux_res_width}
The proof of Lemma~\ref{lm_order_one_width} is relatively long and it is based on several intermediate technical lemmas.

\medskip
\noindent
We recall that, for $n \in \N$, the map $\dotSS^{\dag, \rm{rand}}_n$, defined in Definition~\ref{def_Smith_rand},  assigns to each vertex $ \x \in \VV\GG^{\dag, n}$ the random variable $\dotSS_n^{\dag, \rm{rand}}(\x)$, which is uniformly distributed on the horizontal segment $\SS_n^{\dag}(\x)$. The proof of Lemma~\ref{lm_order_one_width} can be basically divided into four main steps.
\begin{enumerate}[(a)]
	\item We start by considering the continuous time lifted random walk $\X^{n ,\mu_a}$. In Lemma~\ref{lm_unif_integr}, we prove that, for $M \in \N$, the conditional probability, given $\smash{\theta_{+} < \theta_{-}}$, of the event $\smash{|\Re(\X^{n, \mu_a}_{\theta_{+}}) - \Re(\X^{n, \mu_a}_{0})| > M R'}$ decays exponentially in $M$. Using this result, we then prove in Lemma~\ref{lm_wind_apriori_walk} that the conditional expectation of $\smash{\Re(\X^{n, \mu_a}_{\theta_{+}})- \Re(\X^{n, \mu_a}_{0})}$, given $\smash{\theta_{+} < \theta_{-}}$, is of order one. Both the proofs of Lemmas~\ref{lm_unif_integr}~and~\ref{lm_wind_apriori_walk} are based on the invariance principle assumption on the sequence of primal maps.
	\item We then consider the Smith-embedded random walk $\smash{\dotX^{n ,\mu_a}}$ associated with $\smash{\X^{n ,\mu_a}}$. In Lemma~\ref{lm_wind_embedded_walk}, we prove that the conditional expectation of $\smash{\Re(\dotX^{n, \mu_a}_{\theta_{+}})}$, given $\smash{\theta_{+} < \theta_{-}}$, is of order $\eta_n$. This result is basically a consequence of Lemma~\ref{lm_winding_embed_walk}, which guarantees that the expected horizontal winding of the Smith-embedded random walk is equal to zero.
	\item In Lemma~\ref{lm_order_one_width_1}, using the relation between the maps $\dotSS^{\dag, \rm{rand}}_n$ and $\frkw^{\dag}_n$, together with the results proved in the previous steps, we prove that there is a finite constant $b_n' \in \R$ such that the values of the width coordinate function in $\smash{\PPP^{\dag, n}_{\bar a}}$ plus $\smash{b'_n}$ is of order $\eta_n$. The key input to prove this result comes from Lemma~\ref{lm_tube_primal} in which we prove that the probability that the random walk $\X^{n, \mu_a}$ travels a large horizontal distance in a narrow horizontal tube decays exponentially fast. The proof of this fact follows from the invariance principle assumption on the sequence of primal maps.
	\item Finally, in Lemma~\ref{lm_tube_dual}, we state the analogous dual result of Lemma~\ref{lm_tube_primal}. This fact and the periodicity of the Smith diagram will allow us to deduce Lemma~\ref{lm_order_one_width}.
\end{enumerate}

\noindent
Let us emphasize that all the results explained above hold also with the role of $\theta_{+}$ and $\theta_{-}$ interchanged. In particular, the same result stated in point (c) holds with the set $\smash{\PPP^{\dag, n}_{\bar a}}$ replaced by $\smash{\PPP^{\dag, n}_{\underline a}}$.

\medskip
\noindent
We can now proceed to state and prove the technical lemmas mentioned above. We start with the following lemma which states that, for $M \in \N$, the conditional probability of the event $\smash{|\Re(\X^{n, \mu_a}_{\theta_{+}}) - \Re(\X^{n, \mu_a}_{0})| > M R'}$, given $\smash{\theta_{+} < \theta_{-}}$, decays exponentially in $M$. Heuristically speaking, this is due to the fact that, after each time that the random walk $\X^{n, \mu_a}$ travels horizontal distance $R'$, there is a positive chance of hitting the set $\smash{(\frkh_n^{\dag})^{-1}(\bar{a})}$.

\begin{lemma}
\label{lm_unif_integr}
There exists a universal constant $C > 0$ such that, for any $n > n(R, \delta)$ large enough, it holds that
\begin{equation*}
\P_{\mu_a}\l(|\Re(\X^n_{\theta_{+}}) - \Re(\X^n_{0})|> M R' \mid \theta_{+} < \theta_{-}\r) \lesssim \exp(-C M),
\quad \forall M \in \N.
\end{equation*}
where the implicit constant is independent of everything else. The same conclusion holds with the role of $\theta_{+}$ and $\theta_{-}$ interchanged.
\end{lemma}
\begin{proof}
As observed in Remark~\ref{rem_stays_narrow_band}, for any $n > n(R, \delta)$ large enough, we can assume that 
\begin{equation}
\label{eq_prelim_bands}
(\frkh_n^{\dag})^{-1}(\bar{a}) \subset \R \times [R', R' + 1], \quad  (\frkh_n^{\dag})^{-1}(\underline{a}) \subset \R \times [- R'-1, - R'], \quad (\frkh_n^{\dag})^{-1}(a) \subset \R \times [-1, 1].
\end{equation}
Now, letting $\theta := \theta_{+} \wedge \theta_{-}$ and $M \in \N$, we can write
\begin{equation*}
\P_{\mu_a}\l(|\Re(\X^n_{\theta_{+}}) - \Re(\X^n_{0})|> M R' \mid \theta_{+} < \theta_{-}\r) \leq \frac{\P_{\mu_a}\l(|\Re(\X^n_{\theta}) - \Re(\X^n_{0})|> M R'\r) }{\P_{\mu_a}\l(\theta_{+} < \theta_{-}\r)}.
\end{equation*}
We note that, thanks to assumption~\ref{it_invariance} and \eqref{eq_prelim_bands}, for any $\smash{n >n(R, \delta)}$ large enough, the probability on the denominator can be lower bounded by a constant independent of everything else. Therefore, we can just focus on the probability appearing on the numerator. To this end, let $\rho_0 : = 0$, and for $k \in \N_0$ define inductively 
\begin{equation*}
\rho_{k+1} := \inf\l\{t \geq \rho_{k} \, : \, \l|\Re(\X^{n , \mu_a}_t) - \Re(\X^{n , \mu_a}_{\rho_{k}})\r| \geq R'\r\}.
\end{equation*}
Moreover, for $k \in \N_0$, consider the event
\begin{equation*}
A^n_k := \l\{\X^{n, \mu_a}|_{[\rho_{k}, \rho_{k+1}]} \not\subset \R\times\l[\Im(\X^{n, \mu_a}_{\rho_{k}}) - 3R', \Im(\X^{n, \mu_a}_{\rho_{k}}) + 3R'\r]\r\}.
\end{equation*}
We observe that, thanks to the strong Markov property of the random walk, the events $\{A^n_k\}_{k \in \N_0}$ are independent and identically distributed. Moreover, thanks to assumption~\ref{it_invariance}, to estimates \eqref{eq_prelim_bands}, and to well-known properties of Brownian motion, for any $n \in \N$ large enough, the event $A^n_0$ happens with uniformly positive probability $p$ independent of everything else.
Therefore, we have that
\begin{equation*}
\P_{\mu_a}\l(|\Re(\X^n_{\theta_{+}}) - \Re(\X^n_{0})|> M R'\r)\leq \P_{\mu_a}\left(\bigcap_{i = 0}^M \bar{A^n_i}\right) = (1-p)^{M},\end{equation*} 
from which the desired result follows. Finally, the same argument also applies with the role of $\theta_{+}$ and $\theta_{-}$ interchanged.
\end{proof}

\noindent
In the next lemma, we use Lemma~\ref{lm_unif_integr} to prove that the conditional expectation of $\smash{\Re(\X^{n, \mu_a}_{\theta_{+}})- \Re(\X^{n, \mu_a}_{0})}$, given $\smash{\theta_{+} < \theta_{-}}$, is of order one.
\begin{lemma}
\label{lm_wind_apriori_walk}
For any $n > n(R, \delta)$ large enough, it holds that
\begin{equation*}
\l|\E_{\mu_{a}}\l[\Re(\X^n_{\theta_{+}}) - \Re(\X^n_{0}) \mid \theta_{+} < \theta_{-}\r]\r| \lesssim 1,
\end{equation*}
where the implicit constant is universal. The same result holds with the role of $\theta_{+}$ and $\theta_{-}$ interchanged.
\end{lemma}
\begin{proof}
The proof of this result is based on assumption~\ref{it_invariance}. More precisely, let $B^{\mu_a}$ be a planar Brownian motion started from a point in $\PPP^{\dag, n}_{a}$ sampled according to the probability measure $\mu_a$, and consider the following stopping times 
\begin{equation*}
\theta_{B,+} := \inf\l\{t \geq 0 \, : \, \Im(B_t^{\mu_a}) = R'\r\}, \qquad \theta_{B,-} := \inf\l\{t \geq 0 \, : \, \Im(B_t^{\mu_a}) = -R'\r\}.	
\end{equation*}
Since $\Re(B^{\mu_a})$ and $\Im(B^{\mu_a})$ are independent, and since the stopping times $\theta_{B, +}$ and $\theta_{B,-}$ only depend on $\Im(B^{\mu_a})$, thanks to well-known properties of Brownian motion, we have that
\begin{equation}
\label{eq_BM_cond_optional}
\E_{\mu_a}\l[\Re(B_{\theta_{B, +}}) - \Re(B_0) \mid \theta_{B, +} < \theta_{B, -} \r] = 0.
\end{equation}
Furthermore, as observed in Remark~\ref{rem_stays_narrow_band}, for any $\eps \in (0, 1)$ and for any $n > n(R, \delta, \eps)$ large enough, it holds that 
\begin{equation*}
(\frkh_n^{\dag})^{-1}(\bar{a}) \subset \R \times [R', R' + \eps], \qquad  (\frkh_n^{\dag})^{-1}(\underline{a}) \subset \R \times [- R'-\eps, - R']. 
\end{equation*}
In particular, this fact and Lemma~\ref{lm_no_macroscopic_edges} imply that the set $(\frkh_n^{\dag})^{-1}(\bar{a})$ (resp.\ $(\frkh_n^{\dag})^{-1}(\underline{a})$) converges in the Hausdorff metric to the horizontal line $\R \times \{R'\}$ (resp.\ $\R \times \{-R'\}$). Therefore, thanks to assumption~\ref{it_invariance}, the following weak convergence of laws holds
\begin{equation*}
\lim_{n \to \infty} \mathrm{Law}\l(\Re(\X^{n, \mu_a}_{\theta_{+}}) - \Re(X^{n, \mu_a}_0) \mid \theta_{+} < \theta_{-}\r) = \mathrm{Law}\l(\Re(B^{\mu_a}_{\theta_{B, +}}) - \Re(B^{\mu_a}_0) \mid \theta_{B, +} < \theta_{B, -}\r).
\end{equation*}
Hence, thanks to Lemma~\ref{lm_unif_integr} and Vitali's convergence theorem, for any $n > n(R, \delta)$ large enough, it holds that
\begin{equation}
\label{eq_Vitali}
\l|\E_{\mu_{a}}\l[\Re(\X^n_{\theta_{+}}) - \Re(\X^n_0) \mid \theta_{+} < \theta_{-}\r] - \E_{\mu_a}\l[\Re(B_{\theta_{B, +}}) - \Re(B_{0}) \mid \theta_{B, +} < \theta_{B, -} \r] \r| \leq 1.
\end{equation}
Hence, putting together \eqref{eq_BM_cond_optional} and \eqref{eq_Vitali}, we obtain the desired result. Finally, the same argument also applies with the role of $\theta_{+}$ and $\theta_{-}$ interchanged.
\end{proof}

\noindent
In the next lemma, we see how we can use Lemma~\ref{lm_winding_embed_walk} to prove that that the conditional expectation of $\smash{\Re(\dotX^{n, \mu_a}_{\theta_{+}})}$, given $\smash{\theta_{+} < \theta_{-}}$, is of order $\eta_n$. 
\begin{lemma}
\label{lm_wind_embedded_walk}
For any $n \in \N$, it holds that
\begin{equation*}
\l|\E_{\mu_{a}}\l[\Re(\dotX^n_{\theta_{+}}) \mid \theta_{+} < \theta_{-}\r]\r| \leq \eta_n,
\end{equation*}
and the same with the role of $\theta_{+}$ and $\theta_{-}$ interchanged.
\end{lemma}
\begin{proof}
Since we are assuming that the base vertex $\hx^{n, a}_0$ of the lifted width coordinate function $\frkw^{\dag}_n$ belongs to the set $\hat{\PPP}^{\dag, n}_{a}$, then, thanks to Lemma~\ref{lm_bound_width_dot}, it holds almost surely that
\begin{equation*}
\l| \Re(\dotSS^{\dag, \rm{rand}}_n(\x)) \r| \leq \eta_n, \quad \forall \x \in \PPP^{\dag, n}_{a}.
\end{equation*}
In particular, since the embedded random walk $\X^{n, \mu_a}$ is started from a point in the set $\PPP^{\dag, n}_{a}$, then it holds almost surely that 
\begin{equation}
\label{eq_bound_start_embedded}
\l| \Re(\dotX^{n, \mu_a}_0) \r| \leq \eta_n.
\end{equation}

\medskip
\noindent
Now, we would like to apply Lemma~\ref{lm_winding_embed_walk} in order to conclude. However, we note that we cannot directly apply this result since, a priori, it does not hold that $\cup_{x \in \VV\GG^n}\frkh_n^{-1}(\frkh_n(x)) \subseteq \VV\GG^n$. In order to overcome this issue, we could consider the weighted graph associated to $(\GG^n, c^n)$ and $\cup_{x \in \VV\GG^n}\frkh_n^{-1}(\frkh_n(x))$ as specified in Definition~\ref{def_new_graph}. We could then apply Lemma~\ref{lm_winding_embed_walk} to the random walk on this new weighted graph and then transfer the result to the original weighted graph by means of Lemma~\ref{lm_random_walk_assoc}. In order to lighten up the proof, we will assume directly that $\cup_{x \in \VV\GG^n}\frkh_n^{-1}(\frkh_n(x)) = \VV\GG^n$. Since the stopping time $\theta_{+}$ is almost surely finite, we can proceed as follows
\begin{equation}
\label{eq_wind_inter_1}
\begin{alignedat}{1}
& \E_{\mu_{a}}\l[\wind_{\eta_n}(\dotX^n|_{[0, \theta_{+}]}) \mid \theta_{+} < \theta_{-}\r] \\
& \qquad = \frac{1}{\P_{\mu_{a}}(\theta_{+} < \theta_{-})} \E_{\mu_{a}}\left[\sum_{N \in \N}\wind_{\eta_n}(\dotX^n|_{[0, \theta_{+}]}) \mathbbm{1}_{\{\theta_{+} < \theta_{-}, \theta_{+} = N\}}\right] \\
& \qquad = \frac{1}{\P_{\mu_{a}}(\theta_{+} < \theta_{-})}\sum_{N \in \N}\E_{\mu_{a}}\l[\wind_{\eta_n}(\dotX^n|_{[0, \theta_{+}]}) \mid  \theta_{-} > N, \theta_{+} = N\r] \P_{\mu_a}(\theta_{-} > N, \theta_{+} = N).
\end{alignedat}
\end{equation}
In order to pass from the first line to the second line of the above expression, we used the fact that 
\begin{equation}
\label{eq_finite_wind}
	\E\l[\l|\wind_{\eta_n}(\dotX^{n, \mu_a}|_{[0, \theta_{+}]})\r|\r] < \infty,
\end{equation}
and Fubini's theorem. The reason why \eqref{eq_finite_wind} holds is an immediate consequence of Lemmas~\ref{lm_indep_position}~and~\ref{lm_wind_apriori_walk}. In order to conclude, it is sufficient to prove that the expectation in each summand of the sum appearing in \eqref{eq_wind_inter_1} is equal to zero. To this end, fix $N \in \N$ and consider a sequence of admissible height coordinates $[a_N]_0 \subset (0, 1)$ for the random walk $X^{\mu_a}$, as specified in Definition~\ref{def_admissible_heights}, such that $a_m > \underline{a}$ for all $m \in [N]_0$ and $a_N = \bar{a}$. Thanks, to Lemma~\ref{lm_winding_embed_walk}, we have that 
\begin{equation*}
	\E_{\mu_{a}}\l[\wind_{\eta_n}(\dotX^n|_{[0, N]}) \mid \{\frkh_n(X_m) = a_m\}_{m = 1}^N\r] = 0,
\end{equation*}
uniformly over all such sequences of height coordinates. Hence, this is sufficient to conclude that 
\begin{equation}
\label{eq_winding_final_proof}
\E_{\mu_{a}}\l[\wind_{\eta_n}(\dotX^n|_{[0, \theta_{+}]}) \mid  \theta_{-} > N, \theta_{+} = N\r] = 0.
\end{equation}
Therefore, the conclusion follows thanks to \eqref{eq_bound_start_embedded}, and to the fact that \eqref{eq_wind_inter_1} and \eqref{eq_winding_final_proof} imply that $\smash{\E_{\mu_{a}}[\wind_{\eta_n}(\dotX^n|_{[0, \theta_{+}]}) \mid \theta_{+} < \theta_{-}] = 0}$. Finally, we observe that the same argument also applies with the role of $\theta_{+}$ and $\theta_{-}$ interchanged.
\end{proof}

\noindent
In order to prove Lemma~\ref{lm_order_one_width}, we basically need to prove that the difference between the value of the width coordinate in the set $\smash{\hat{\PPP}^{\dag, n}_{\bar{a}}}$ and in the set $\smash{\hat{\PPP}^{\dag, n}_{\underline{a}}}$ is of order $\eta_n$. This fact is the content of Lemma~\ref{lm_order_one_width_1} below. However, in order to prove this fact, we first need to prove that it is extremely unlikely for the random walk $\smash{\X^{n, \mu_a}}$ to travel a large horizontal distance in a ``narrow horizontal tube''. We refer to Figure~\ref{fig_tube_primal} for a diagram illustrating the various objects involved in the proof of Lemma~\ref{lm_tube_primal}.

\begin{figure}[h]
\centering
\includegraphics[scale=0.9]{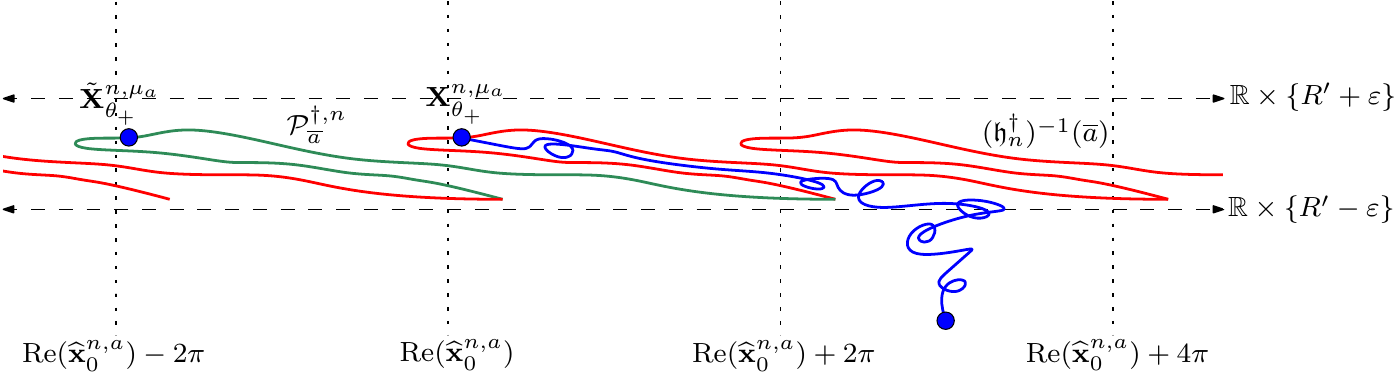}
\caption[short form]{\small The vertices $(\frkh^{\dag}_n)^{-1}(\bar{a})$ are drawn for simplicity as a continuous red/green path contained in the infinite strip $\R \times [R'-\eps, R'+\eps]$. In particular, the green part of this path contains the vertices in the set $\smash{\PPP^{\dag, n}_{\bar{a}}}$. The blue vertex at the bottom of the figure is a vertex sampled from the set $\smash{\PPP^{\dag, n}_{a}}$ according to the probability measure $\mu_a$. The blue path represents a trajectory of the random walk $\X^{n, \mu_a}$, and the blue point at the top-right coincides with the hitting point $\X^{n, \mu_a}_{\theta_{+}}$. The blue point $\tilde \X^{n, \mu_a}_{\theta_{+}}$ at the top-left is the vertex in $\smash{\PPP^{\dag, n}_{\bar{a}}}$ such that $\smash{\sigma_{2\pi}(\tilde \X^{n, \mu_a}_{\theta_{+}}) = \sigma_{2\pi}(\X^{n, \mu_a}_{\theta_{+}})}$. Roughly speaking, Lemma~\ref{lm_tube_primal} says that the event depicted in the figure is extremely unlikely to happen. A consequence of this fact is that the conditional expectation of $\smash{\Re(\tilde \X^{n, \mu_a}_{\theta_{+}})}$, given $\theta_{+} < \theta_{-}$, is of order one.}
\label{fig_tube_primal}
\end{figure}

\begin{lemma}
\label{lm_tube_primal}
Fix $\eps \in (0, 1)$, $M \in \N$, and define the following event
\begin{equation*}
\mathrm{A}_{M, \eps}^{n, +} := \l\{ \exists \, s, t \in [0, \theta_{+}] \, : \, \l|\Re(\X^{n, \mu_a}_t) - \Re(\X^{n, \mu_a}_s)\r| > M ; \; \l|\Im(\X^{n, \mu_a}_u)\r| \in [R' - \eps, R' + \eps], \; \forall u \in [s, t]\r\}.
\end{equation*} 
There exists a universal constant $C > 0$ such that, for any $n > n(R, \delta, \eps)$ large enough, it holds that
\begin{equation*}
\P_{\mu_a}\l(\mathrm{A}_{M, \eps}^{n, +} \mid \theta_{+} < \theta_{-}\r) \lesssim \exp(-C M/\eps),
\quad \forall M \in \N,
\end{equation*}
where the implicit constant is independent of everything else. Furthermore, the same conclusion holds with the role of $\theta_{+}$ and $\theta_{-}$ interchanged.
\end{lemma}
\begin{proof}
Fix $\eps \in (0, 1)$. We start by recalling that, as observed in Remark~\ref{rem_stays_narrow_band}, for any $n > n(R, \delta, \eps)$ large enough, it holds that
\begin{equation}
\label{eq_eghat_narrow}
(\frkh_n^{\dag})^{-1}(\bar{a}) \subset \R \times [R' - \eps, R' + \eps].
\end{equation}
It is easy to see that, without any loss of generality, we can assume that $\VV\GG^{n}$ contains all the points in the set
\begin{equation*}
	\l\{x \in \G^n \, : \, \Im(x) = R' - \eps\r\}. 
\end{equation*}
Therefore, we can define the stopping time $\rho_0 : = \inf\l\{t \geq 0 \, : \, \Im(\X^{n, \mu_a}_t) = R'- \eps\}$, and for $k \in \N_0$ we define inductively 
\begin{equation*}
\tilde \rho_{k+1} := \inf\l\{t \geq \rho_{k} \, : \, \l|\Re(\X^{n, \mu_a}_t) - \Re(\X^{n, \mu_a}_{\rho_{k}})\r| > M \}, \quad \rho_{k+1} := \inf\l\{t \geq \tilde \rho_{k+1} \, : \, \Im(\X^{n, \mu_a}_t) = R' - \eps \r\}.
\end{equation*}
Moreover, for $k \in \N$ we consider the events
\begin{equation*}
\mathrm{A}_{M, \eps}^{n, +}(\rho_{k-1}, \tilde \rho_{k}) := \l\{\Im(\X^{n, \mu_a}|_{[\rho_{k-1}, \tilde \rho_{k}]}) \subset [R'-\eps, R'+\eps]\r\} , \quad F_{k} := \l\{\rho_{k} > \hat{\tau}_{\hx}\r\}.
\end{equation*}
Let $K$ be the smallest $k \in \N$ such that $F_k$ occurs. Then, we have that
\begin{equation}
\label{eq_inter_step_narrow}
\begin{alignedat}{1}
\P_{\mu_a}\l(\mathrm{A}_{M, \eps}^{n, +} \mid \theta_{+} < \theta_{-}\r) & \leq \P_{\mu_a}\left(\bigcup_{i = 1}^{K} \mathrm{A}_{M, \eps}^{n, +}(\rho_{i-1}, \tilde \rho_{i}) \mid \theta_{+} < \theta_{-} \right) \\
&  \leq \sum_{k \in \N} \sum_{i = 1}^{k}\P_{\mu_a}\left(\mathrm{A}_{M, \eps}^{n, +}(\rho_{i-1}, \tilde \rho_{i})\mid \theta_{+} < \theta_{-} \right) \P_{\mu_a}\l(K = k \mid \theta_{+} < \theta_{-}\r).
\end{alignedat}
\end{equation}
Thanks to the strong Markov property of the random walk, the events $\{\mathrm{A}_{M, \eps}^{n, +}(\rho_{k-1}, \tilde \rho_{k})\}_{k \in \N}$ are conditionally independent and identically distributed given $\theta_{+} < \theta_{-}$. 
Now, thanks to assumption~\ref{it_invariance} and well-known properties of Brownian motion, it is possible to prove that there is a universal constant $C > 0$ such that for any $n \in \N$ large enough it holds that
\begin{equation}
\label{eq_bound_exp_k}
	\P_{\mu_a}\l(\mathrm{A}_{M, \eps}^{n, +}(\rho_{0}, \tilde \rho_{1}) \mid \theta_{+} < \theta_{-}\r) \lesssim \exp\l(-C M / \eps\r), 
\end{equation}
where the implicit constant is independent of everything else. More precisely, in order to obtain the above upper bound, it is sufficient to study the probability that the random walk travels horizontal distance $M$ before exiting an infinite horizontal band of height of order $\eps$. This can be done by proceeding similarly to the proof of Lemma~\ref{lm_unif_integr}, and so we do not detail the argument here. Now, for $k \in \N$, we consider the event
\begin{equation*}
\mathrm{B}_{M}^{n, +}(\rho_{k-1}, \tilde \rho_{k}) := \l\{\Im(\X^{n, \mu_a}|_{[\rho_{k-1}, \tilde \rho_{k}]}) \subset [-R'-1, R'+1]\r\}.
\end{equation*}
For the same reason explained above, we have that the events $\{\mathrm{B}_{M}^{n, +}(\rho_{k-1}, \tilde \rho_{k})\}_{k \in \N}$ are conditionally independent and identically distributed given $\theta_{+} < \theta_{-}$. Also, using again assumption~\ref{it_invariance} and a standard gambler's ruin estimate for Brownian motion, one can prove that  
\begin{equation}
\label{eq_bound_exp_k_1}
	\P_{\mu_a}\l(\mathrm{B}_{M}^{n, +}(\rho_{0}, \tilde \rho_{1}) \mid \theta_{+} < \theta_{-} \r) \lesssim (M+1)^{-1}, 
\end{equation}
where the implicit constant is independent of everything else. In particular, thanks to \eqref{eq_eghat_narrow}, for any $k \in \N$, it holds that 
\begin{align*}
	\P_{\mu_a}\l(K = k \mid \theta_{+} < \theta_{-} \r) & \leq \P_{\mu_a}\left(\bigcap_{i = 1}^{k-1} \mathrm{B}_{M}^{n, +}(\rho_{i-1}, \tilde \rho_{i}) \mid \theta_{+} < \theta_{-}\right) \\
	& = \prod_{i=1}^{k-1} \P_{\mu_a}\l(\mathrm{B}_{M}^{n, +}(\rho_{0}, \tilde \rho_{1}) \mid \theta_{+} < \theta_{-} \r) \\
	& \lesssim (M+1)^{-(k-1)}.
\end{align*}
Therefore, putting together \eqref{eq_inter_step_narrow}, \eqref{eq_bound_exp_k} and \eqref{eq_bound_exp_k_1}, we obtain that there is a universal constant $C > 0$ such that
\begin{equation*}
\P_{\mu_a}\l(\mathrm{A}_{M, \eps}^{n, +} \mid \theta_{+} < \theta_{-}\r)  \lesssim \exp \l(-C M / \eps\r) \sum_{k \in \N} k (M+1)^{-(k-1)} \lesssim \exp \l(-C M / \eps\r).
\end{equation*}
Finally, we observe that the result with the role of $\theta^{+}$ and $\theta^{-}$ interchanged can be obtained in the same way. Therefore, this concludes the proof.
\end{proof}

\noindent
We are now ready to prove that the difference between the value of the width coordinate in the set $\smash{\hat{\PPP}^{\dag, n}_{\bar{a}}}$ and in the set $\smash{\hat{\PPP}^{\dag, n}_{\underline{a}}}$ is of order $\eta_n$. We refer to Figure~\ref{fig_period_width} for a diagram illustrating the various objects involved in the proof of Lemma~\ref{lm_order_one_width_1}.

\begin{figure}[h]
\centering
\includegraphics[scale=1]{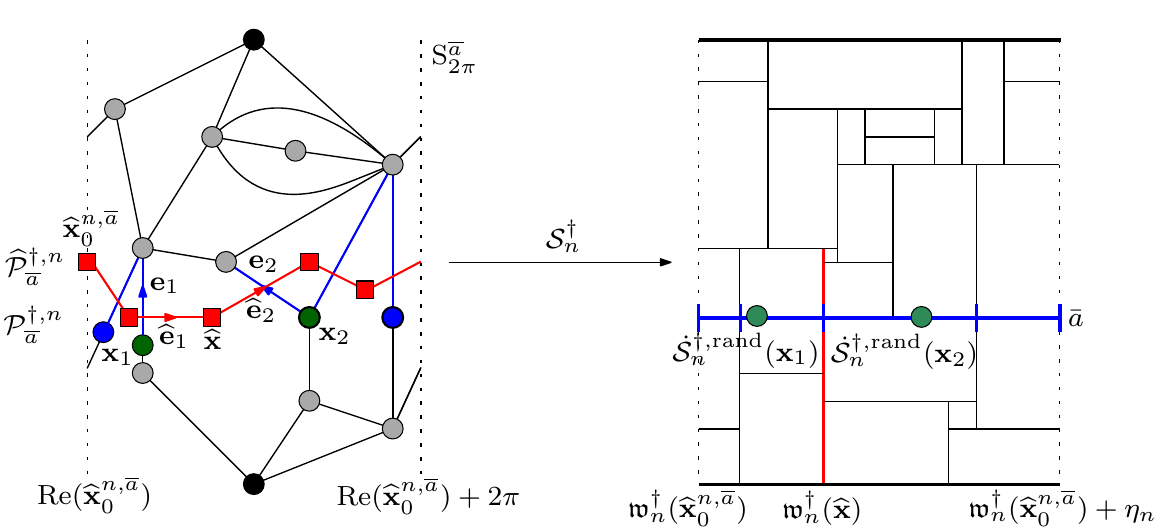}
\caption[short form]{\small A diagram illustrating Step~2 of the proof of Lemma~\ref{lm_order_one_width_1}. \textbf{Left:} A portion of the lifted graph $\smash{\GG^{\dag, n}}$ embedded in $\R^2$. The blue/green vertices are the vertices in $\smash{\PPP_{\bar a}^{\dagger, n}}$. The blue edges are the edges in the set $\smash{\EE\GG_{\bar{a}}^{\dag, n}}$. The red dual edges are the edges in the set $\smash{\EE\hat{\GG}^{\dag, n}_{\bar{a}}}$. The red dual vertices are the vertices in the set $\smash{\hat{\PPP}^{\dag, n}_{\bar{a}}}$. Consider the dual vertex $\hx$ in $\smash{\hat{\PPP}^{\dag, n}_{\bar{a}}}$, then the two red dual edges with an arrow are the edges $\he_1$, $\he_2$ associated to $\hx$. The blue edges with an arrow are the corresponding primal edges $\e_1$, $\e_2$. The two green points are the endpoints of $\e_1$, $\e_2$ in $\smash{(\frkh_n^{\dag})^{-1}(\bar{a})}$. \textbf{Right:} A portion of the tiling of $\mathbb{R} \times [0,1]$ constructed via the lifted tiling map $\smash{\SS_n^{\dag}}$. The blue horizontal line corresponds to $\smash{\cup_{\x \in \PPP^{\dag, n}_{\bar{a}}} \SS^{\dag}_n(\x)}$. The two green points are possible realizations of $\smash{\dotSS_n^{\dag,\rm{rand}}(\x_1)}$ and $\smash{\dotSS_n^{\dag,\rm{rand}}(\x_2)}$. The vertical red segment with horizontal coordinate $\smash{\frkw_n^{\dag}(\hx)}$ corresponds to $\smash{\SS_n^{\dag}(\hx)}$.}
\label{fig_period_width}
\end{figure}
 
\begin{lemma}
\label{lm_order_one_width_1}
For any $n > n(R, \delta)$ large enough, there is a finite constant $b'_n \in \R$ such that 
\begin{equation*}
\l|\frkw^{\dag}_n(\hat{\x}) + b_n'\r| \lesssim \eta_n, \quad \forall \hx \in \hat{\PPP}^{\dag, n}_{\bar{a}} \cup \hat{\PPP}^{\dag, n}_{\underline{a}},
\end{equation*}
where the implicit constant is independent of everything else.
\end{lemma}
\begin{proof}
We start by letting $\theta := \theta_{+} \wedge \theta_{-}$ and defining 
\begin{equation*}
	b_n' := \E_{\mu_{a}}\l[\Re(\X^n_{0}) - \Re(\tilde \X^n_{\theta})\r].
\end{equation*}
We observe that $b_n'$ is not of constant order in general. Indeed, this is due to the fact that $\smash{\Re(\X^{n, \mu_a}_{0})}$ can be far from being of order one since the starting point of the random walk $\X^{n, \mu_a}$ is sampled from the set $\smash{\PPP^{\dag,n}_{a}}$ over which we do not have any a priori control. We will only prove the result in the case $\smash{\hx \in \hat{\PPP}^{\dag, n}_{\bar{a}}}$ since the result for $\hx \in \smash{\PPP^{\dag,n}_{\underline{a}}}$ can be obtained similarly. In particular, we split the proof in two steps. 

\medskip
\noindent\textbf{Step~1:} In this first step, we claim that for any $n > n(R, \delta)$ large enough, it holds almost surely that 
\begin{equation*}
\l|\Re(\dotSS^{\dag, \rm{rand}}_n(\x)) + b_n'\r| \lesssim \eta_n , \quad \forall \x \in \PPP^{\dag, n}_{\bar{a}}, 
\end{equation*}
where the implicit constant does not depend on anything else. To this end, we consider the vertex $\smash{\tilde{\X}^{n, \mu_a}_{\theta_{+}} \in \PPP^{\dag, n}_{\bar{a}}}$ such that $\smash{\sigma_{2\pi}(\tilde{\X}^{n, \mu_a}_{\theta_{+}}) = \sigma_{2\pi}(\X^{n, \mu_a}_{\theta_{+}})}$. Now, applying Lemma~\ref{lm_indep_position}, we see that
\begin{equation}
\label{eq_period_both_sides}
\left|\frac{\Re(\dotSS^{\dag, \rm{rand}}_n(\tilde{\X}^{n, \mu_a}_{\theta_{+}})) - \Re(\dotX^{n, \mu_a}_{\theta_{+}})}{\eta_n} - \frac{\Re(\tilde{\X}^{n, \mu_a}_{\theta_{+}}) - \Re({\X}^{n, \mu_a}_{\theta_{+}})}{2\pi}\right| \leq 1.
\end{equation}
Therefore, rearranging the various terms in the previous inequality and then taking the conditional expectation given $\theta_{+} < \theta_{-}$, we get that
\begin{equation}
\label{eq_triang_bunch}
\begin{alignedat}{1} 
\l| \E_{\mu_{a}}\l[\Re(\dotSS^{\dag, \rm{rand}}_n(\tilde{\X}^{n}_{\theta_{+}})) \mid \theta_{+} < \theta_{-} \r] + b_n'\r| & \lesssim\l|\E_{\mu_a}\l[\Re(\dotX^{n}_{\theta_{+}}) \mid \theta_{+} < \theta_{-}\r]\r| \\
& + \eta_n \l|\E_{\mu_{a}}\l[\Re(\X^n_{\theta_{+}}) - \Re(\X^n_{0})  \mid \theta_{+} < \theta_{-}\r]\r| \\
& + \eta_n \l|b_n' + \E_{\mu_{a}}\l[\Re(\tilde \X^n_{\theta_{+}}) - \Re(\X_0^n) \mid \theta_{+} < \theta_{-}\r]\r|.
\end{alignedat}
\end{equation}
Thanks to Lemmas~\ref{lm_wind_embedded_walk}~and~\ref{lm_wind_apriori_walk}, we have that the following inequalities hold for any $n > n(R, \delta)$ large enough
\begin{equation}
\label{eq_ineqs_wind}
\l|\E_{\mu_a}\l[\Re(\dotX^{n}_{\theta_{+}}) \mid \theta_{+} < \theta_{-}\r]\r| \leq \eta_n, \qquad \l|\E_{\mu_{a}}\l[\Re(\X^n_{\theta_{+}}) - \Re(\X^n_{0}) \mid \theta_{+} < \theta_{-}\r]\r| \lesssim 1.
\end{equation}
Therefore, the claim follows if we prove that the term in the third line of \eqref{eq_triang_bunch} is of order one. More precisely, thanks to the definition of $b_n'$, it is sufficient to prove that, for any $n>n(R, \delta)$ large enough, it holds that
\begin{equation*}
	\l|\E_{\mu_{a}}\l[\Re(\tilde \X^n_{\theta})\r] - \E_{\mu_{a}}\l[\Re(\tilde \X^n_{\theta_{+}}) \mid \theta_{+} < \theta_{-}\r]\r| \lesssim 1.
\end{equation*}
Since, for any $n>n(R, \delta)$ large enough, $\smash{\P_{\mu_a}(\theta_{+} < \theta_{-})}$ is of constant order, this is equivalent to prove that
\begin{equation*}
\l|\E_{\mu_{a}}\l[\Re(\tilde \X^n_{\theta_{+}}) \mid \theta_{+} < \theta_{-}\r]\r| + \l|\E_{\mu_{a}}\l[\Re(\tilde \X^n_{\theta_{-}}) \mid \theta_{-} < \theta_{+}\r]\r| \lesssim 1,
\end{equation*}
where the implicit constant must be independent of everything else. We note that this inequality follows easily from Lemma~\ref{lm_tube_primal} (see Figure~\ref{fig_tube_primal}). Hence, putting everything together and going back to \eqref{eq_triang_bunch}, we obtain that 
\begin{align*}
\l| \E_{\mu_{a}}\l[\Re(\dotSS^{\dag, \rm{rand}}_n(\tilde{\X}^{n}_{\theta_{+}})) \mid \theta_{+} < \theta_{-} \r] + b_n'\r| \lesssim \eta_n ,
\end{align*}
where the implicit constant is independent of everything else. Therefore, the desired claim follows from the previous inequality and from the fact that, thanks to Lemma~\ref{lm_bound_width_dot}, it holds almost surely that 
\begin{equation*}
\l|\Re(\dotSS^{\dag, \rm{rand}}_n(\x_1)) - \Re(\dotSS^{\dag, \rm{rand}}_n(\x_2))\r| \leq 2 \eta_n, \quad  \forall \x_1, \smash{\x_2 \in \PPP^{\dag, n}_{\bar{a}}}
\end{equation*}

\medskip
\noindent\textbf{Step~2:}
In this step, we will actually prove the result in the lemma statement. To this end, we fix $\smash{\hx \in \hat{\PPP}_{\bar{a}}^{\dagger, n}}$ and we consider the dual edges $\smash{\he_1}$, $\smash{\he_2 \in \EE\hat{\GG}^{\dag, n}_{{\bar{a}}}}$ such that $\smash{\he_1^{+} = \hx = \he_2^{-}}$. Furthermore, we let $\smash{\e_1}$, $\smash{\e_2 \in \EE\GG^{\dag, n}_{{\bar{a}}}}$ be the corresponding primal edges, and we let $\smash{\x_1}$, $\smash{\x_2 \in \VV\GG^{\dag, n}}$ be the endpoints of $\e_1$ and $\e_2$ in the set $\smash{(\frkh_n^{\dag})^{-1}({\bar{a}})}$ (see Figure~\ref{fig_period_width}). At this point, we need to split the proof in two different cases:
\begin{itemize}
	\item If $\x_1 \neq \x_2$, then, by construction of the Smith diagram, it holds almost surely that 
\begin{equation*}
\frkw_n^{\dag}(\hat{\x}) \in \l[\Re(\dotSS_n^{\dag, \rm{rand}}(\x_1)), \Re(\dotSS_n^{\dag, \rm{rand}}(\x_2))\r].
\end{equation*}
\item If $\x_1 = \x_2$, then it holds that
\begin{equation*}
\frkw_n^{\dag}(\hat{\x}) \in \l[\min\l\{\Re(\v) \, : \v \in \SS_n^{\dag}(\x_1)\}, \max\{\Re(\v) \, : \v \in \SS_n^{\dag}(\x_1)\r\}\r],
\end{equation*}
where we recall that $\SS_n^{\dag}(\x_1)$ denotes the horizontal segment associated to $\x_1$ by the lifted tiling map. 
\end{itemize}
 In both cases, if $\smash{\x_1}$, $\smash{\x_2 \in \PPP^{\dag, n}_{\bar{a}}}$, then the conclusion follows from the previous step. However, in general, it could be that $\smash{\x_1 \in \PPP^{\dag, n}_{\bar{a}} - (2\pi, 0)}$ or $\smash{\x_2 \in \PPP^{\dag, n}_{\bar{a}} + (2\pi, 0)}$. In both these cases, we cannot directly appeal to the previous step to conclude. Nevertheless, a simple application of Lemma~\ref{lm_indep_position} implies that the same result of the first step holds also for the vertices in $\smash{\PPP^{\dag, n}_{\bar{a}} - (2\pi, 0)}$ and in $\smash{\PPP^{\dag, n}_{\bar{a}} + (2\pi, 0)}$. Therefore, this concludes the proof. 
\end{proof}

\noindent
Before proceeding with the proof of Lemma~\ref{lm_order_one_width}, we need to state a lemma which is the dual counterpart of Lemma~\ref{lm_tube_primal}.
\begin{lemma}
\label{lm_tube_dual}
Fix $\eps \in (0, 1)$, $\hx \in \VV\hat{\GG}^{\dag, n}(R'-1)$, $M \in \N$, and define the following event
\begin{equation*}
\hat{\mathrm{A}}_{M, \eps}^{n, \hx} := \l\{ \exists \, s, t \in [0, \hat{\tau}_{\hx}] \, : \, \l|\Re(\hat{\X}^{n, \hx}_t) - \Re(\hat{\X}^{n, \hx}_s)\r| > M ; \; \l|\Im(\hat{\X}^{n, \hx}_u)\r| \in [R' - \eps, R' + \eps], \; \forall u \in [s, t]\r\}.
\end{equation*} 
There exists a universal constant $C > 0$ such that, for any $n > n(R, \delta, \eps)$ large enough, it holds that
\begin{equation*}
\P_{\hx}\l(\hat{\mathrm{A}}_{M, \eps}^{n, \hx} \r) \lesssim \exp(-C M/\eps),
\quad \forall M \in \N,
\end{equation*}
where the implicit constant is independent of everything else.
\end{lemma}
\begin{proof}
	The proof of this lemma can be done by employing a similar argument to that used in Lemma~\ref{lm_tube_primal}.
\end{proof}

\noindent
We are now ready to prove Lemma~\ref{lm_order_one_width}, which is now an immediate consequence of the results proved above.
\begin{proof}[Proof of Lemma~\ref{lm_order_one_width}]
The proof basically consists of putting together some of the previous results. More precisely, fix $\smash{\hx \in \VV\hat{\GG}^{\dag, n}(R'-1)}$ and let $\smash{\tilde{\X}^{n, \hx}_{\hat{\tau}_{\hat{x}}} \in \hat{\PPP}_{\bar{a}}^{\dagger, n}\cup \hat{\PPP}_{\underline{a}}^{\dagger, n}}$ be such that $\smash{\sigma_{2\pi}(\tilde{\X}^{n, \hx}_{\hat{\tau}_{\hat{x}}}) = \sigma_{2\pi}(\hat{\X}^{n, \hx}_{\hat{\tau}_{\hat{x}}})}$. Then, by Lemma~\ref{lm_harm_conj}, we have that
\begin{equation*}
\E_{\hx}\left[\frac{\Re(\hat{\X}^{n}_{\hat{\tau}_{\hat{x}}})}{2 \pi} - \frac{\frkw_n^{\dag}(\hat{\X}^{n}_{\hat{\tau}_{\hat{x}}}) + b'_n}{\eta_n} \right] = \E_{\hx}\left[\frac{\Re(\tilde{\X}^{n}_{\hat{\tau}_{\hat{x}}})}{2 \pi} - \frac{\frkw_n^{\dag}(\tilde{\X}^{n}_{\hat{\tau}_{\hat{x}}}) + b'_n}{\eta_n} \right].
\end{equation*}
Therefore, the result follows thanks to Lemma~\ref{lm_order_one_width_1} if we show that also $\smash{\l|\E_{\hx}[\Re(\tilde{\X}^{n}_{\hat{\tau}_{\hat{x}}})]\r| \lesssim 1}$. We note that this fact is easily implied by Lemma~\ref{lm_tube_dual}, and so the proof is completed.
\end{proof}

%%%%%%%%%%%%%%%%%%%%%%%%%%%%%%%%%%%%%%%%%%
\subsection{Proof of Theorem~\ref{th_main_1}}
\label{sub_proof_main}
We are now ready to give a proof of the main theorem of this article. As we have already remarked, the proof of this theorem is a consequence of Propositions~\ref{pr_main_height}~and~\ref{pr_main_witdh}. 

\begin{proof}[Proof of Theorem~\ref{th_main_1}]
Fix $R > 1$, $\delta \in (0, 1)$, and consider a point $\x \in \VV\GG^{\dag, n}(R)$. We divide the proof into three main steps. 

\medskip
\noindent\textbf{Step~1.} 
By definition of the Smith embedding $\dotSS^{\dag}_n$, we have that $\Im(\dotSS^{\dag}_n(\x)) = \frkh_n^{\dag}(\x)$. Hence, from Proposition~\ref{pr_main_height}, we know that there exist two real sequences $\{b^{\frkh}_n\}_{n \in \N}$ and $\{c^{\frkh}_n\}_{n \in \N}$, independent of $R$ and $\delta$, such that, for $n> n(R, \delta)$ large enough, it holds that
\begin{equation*}
\l|c^{\frkh}_n \Im(\dotSS^{\dag}_n(\x)) + b^{\frkh}_n -\Im(\x)\r| \leq \frac{\delta}{\sqrt{2}}, \quad \forall \x \in \VV\GG^{\dag, n}(\R \times [-R, R]).
\end{equation*}

\medskip
\noindent\textbf{Step~2.} 
Let $\EE\GG^{\dag, n, \downarrow}(\x) = [\e_{k}]$ be the set of harmonically oriented edges in $\EE\GG^{\dag, n}$ with heads equal to $\x$ ordered in such a way that $\he_{1} \cdots \he_{k}$ forms a counter-clockwise oriented path in the lifted dual graph $\hat{\GG}^{\dag, n}$. Then, by construction of the Smith embedding $\dotSS^{\dag}_n$, we have that
\begin{equation}
\label{eq_point_inside_aux}
\Re(\dotSS^{\dag}_n(\x)) \in \l[\frkw^{\dag}(\hat{\e}^{-}_{1}) , \frkw^{\dag}(\hat{\e}^{+}_{k})\r].
\end{equation}
Therefore, letting $\{b^{\frkw}_n\}_{n \in \N}$ be the sequence in the statement of Proposition~\ref{pr_main_witdh}, we have that
\begin{align*}
& \left|\frac{2\pi}{\eta_n} \Re(\dotSS^{\dag}_n(\x)) + b^{\frkw}_n -\Re(\x)\right| \\
& \qquad \leq \left|\frac{2\pi}{\eta_n} \frkw_n(\he_1^{-}) + b^{\frkw}_n - \Re(\he_1^{-})\right| + \l|\Re(\he^{-}_1)  -\Re(\x)\r| + \frac{2\pi}{\eta_n}\l|\Re(\dotSS_n^{\dag}(\x)) - \frkw^{\dag}_n(\he_1)\r|.
\end{align*}
The first term on the right-hand side of the above expression is bounded by $\delta/(5\sqrt{2})$ thanks to Proposition~\ref{pr_main_witdh}. The second term is also bounded by $\delta/(5\sqrt{2})$ since Lemma~\ref{lm_no_macroscopic_edges} rules out the existence of macroscopic faces. Concerning the third term, recalling \eqref{eq_point_inside_aux}, we have that
\begin{align*}
& \frac{2\pi}{\eta_n} \l|\Re(\dotSS_n^{\dag}(\x)) - \frkw^{\dag}_n(\he_1^{-})\r| \\
& \qquad \leq \left|\frac{2\pi}{\eta_n} \frkw^{\dag}_n(\he_1^{-}) + b^{\frkw}_n -\Re(\he_1^{-})\right| + \left|\frac{2\pi}{\eta_n}\frkw^{\dag}_n(\he_k^{+}) + b^{\frkw}_n -\Re(\he_k^{+})\right| + \l|\Re(\he_k^{+}) - \Re(\he_1^{-})\r|.
\end{align*}
The first and second term on the right-hand side of the above expression are bounded by $\delta/(5\sqrt{2})$, thanks to Proposition~\ref{pr_main_witdh}. The third term is also bounded by $\delta/(5\sqrt{2})$, for $n> n(R, \delta)$ thanks, once again, to Lemma~\ref{lm_no_macroscopic_edges}. We remark that all the previous bounds obviously hold only for $n>n(R, \delta)$ large enough. Therefore, putting it all together, we obtain that, for all $n> n(R, \delta)$ large enough, it holds that
\begin{equation*}
\left|\frac{2\pi}{\eta_n}\Re(\dotSS^{\dag}_n(\x)) + b^{\frkw}_n -\Re(\x)\right| \leq \frac{\delta}{\sqrt{2}}, \quad \forall \x \in \VV\GG^{\dag, n}(\R \times [-R, R]). 
\end{equation*}

\medskip
\noindent\textbf{Step~3.} 
For each $n \in \N$, we define the affine transformation $T^{\dag}_n : \R \times [0, 1] \to \R^2$ by letting 
\begin{equation*}
\Re(T^{\dag}_n \x) := \frac{2\pi}{\eta_n} \Re(\x) + b_n^{\frkw} \quad \text{ and } \quad \Im(T^{\dag}_n \x) := c_n^{\frkh} \Im(\x) + b^{\frkh}_n, \qquad \forall \x \in \R \times [0, 1],
\end{equation*}
Therefore, the previous two steps yield that, for any $n > n(R, \delta)$ large enough, it holds that
\begin{equation*}
\d_{\R^2}\l(T^{\dag}_n \dotSS^{\dag}_n(\x), \x\r) \leq \delta, \quad \forall \x \in \VV\GG^{\dag, n}(\R \times [-R, R]),
\end{equation*}
where $\d_{\R^2}$ denotes the Euclidean distance in $\R^2$. This is obviously equivalent to the desired result, and so the proof is completed.
\end{proof}
%%%%%%%%%%%%%%%%%%%%%%%%%%%%%%%%%%%%%%%%%%

%%%%%%%%%%%%%%%%%%%%%%%%%%%%%%%%%%%%%%%%%%
\section{Application to mated-CRT maps}
\label{sec_mated}
The main goal of this section is to prove Theorem~\ref{th_main_2}. Roughly speaking, the plan is as follows. We will first introduce an a priori embedding of mated-CRT maps which is ``close'' to LQG. We then prove that this a priori embedding satisfies the assumptions of Theorem~\ref{th_main_1}. Finally, we show how this allows to conclude.

%%%%%%%%%%%%%%%%%%%%%%%%%%%%%%%%%%%%%%%%%%
\subsection{SLE/LQG description of mated-CRT maps}
\label{sub_SLE_LQG}
We now discuss an equivalent description of mated-CRT maps in terms of SLE/LQG, which comes from the results of \cite{DMS21}. These results imply that mated-CRT maps can be realized as cell configurations constructed from space-filling $\SLE_{\kappa}$ curves parameterized by quantum mass with respect to a certain independent LQG surface. We will not need many properties of the SLE/LQG objects involved, so we will not give detailed definitions, but we will give precise references instead.

\paragraph{Liouville quantum gravity surfaces.}
For $\gamma \in (0, 2)$ and $D \subseteq \C$, a doubly marked $\gamma$-LQG surface is an
equivalence class of quadruples $(D, h,z_1, z_2)$ where $h$ is a random generalized function
on $D$ (which we will always take to be an instance of some variant of the Gaussian free field),
and $z_1$, $z_2 \in D$. Two such quadruples $\smash{(D, h, z_1, z_2)}$ and $\smash{(\tilde{D}, \tilde{h}, \tilde{z_1}, \tilde{z_2})}$ are declared to be equivalent if there is a conformal map $f : \tilde D \to D$ such that
\begin{equation}
\label{eq_equiv_LQG}
	\tilde{h} = h \circ f + Q \log |f'| \quad \text{ and } \quad f(\tilde{z}_1) = z_1, \; f(\tilde{z}_2) = z_2, \quad \text{ where } \quad Q = \frac{2}{\gamma} + \frac{\gamma}{2}.
\end{equation}
For $\gamma \in (0, 2)$, it is well-known that one can construct a random measure, called the $\gamma$-LQG area measure, which is formally given by $\mu_h := e^{\gamma h}\d^2z$, where $d^2z$ denotes the Lebesgue measure on $D$. Since $h$ is a random generalized function, this definition does not make rigorous sense and one should proceed using a standard regularization and limiting procedure \cite{DS11}. The $\gamma$-LQG area measure satisfies a certain change of coordinates formula. More precisely, given two equivalent doubly marked $\gamma$-LQG surface $(D, h, z_1, z_2)$ and $(\tilde{D}, \tilde{h}, \tilde{z_1}, \tilde{z_2})$, then it holds almost surely that $\mu_h(f(A)) = \mu_{\tilde{h}}(A)$ for all Borel sets $A \subseteq \tilde{D}$, where $f : \tilde D \to D$ is a conformal map such that \eqref{eq_equiv_LQG} holds.

\medskip
\noindent
In this article, we are interested in two different kind of doubly marked $\gamma$-LQG surfaces.
\begin{itemize}
 \item The \emph{doubly marked quantum sphere} $(\C, h, 0, \infty)$, where $h$ is a variant of the Gaussian free field precisely defined in \cite[Definition~4.21]{DMS21}. For $\gamma \in (0, 2)$, it is well-known that one can associate with the random generalized function $h$ a random measure $\mu_h$ on $\C$, the $\gamma$-LQG measure, with $\mu_h(\C) < \infty$ (again, we will not need the precise definition here). Typically, one considers a unit-area quantum sphere, which means that we fix $\mu_h(\C) = 1$. 
 \item The \emph{$0$-quantum cone} $(\C, h^{\rm{c}}, 0, \infty)$, where $h^{\rm{c}}$ is a variant of the Gaussian free field precisely defined in \cite[Definition 4.10]{DMS21}. Also in this case, for $\gamma \in (0, 2)$, we can associate to $h^{\rm{c}}$ the $\gamma$-LQG measure $\mu_{h^{\rm c}}$ which has infinite total mass, but it is locally finite.
\end{itemize}

\paragraph{Schramm--Loewner evolution.}
We do not need to precisely define SLE$_{\kappa}$, but rather it is sufficient to know that whole-plane space-filling SLE$_\kappa$, for $\kappa > 4$, is a random space-filling curve $\theta$ which travels from $\infty$ to $\infty$ in $\C$. It is a variant of SLE$_{\kappa}$ \cite{Sh00} which was introduced in \cite{IM_four, DMS21}. Space-filling SLE$_{\kappa}$ for $\kappa \geq 8$ is a two-sided version of ordinary SLE$_{\kappa}$ (which is already space-filling), whereas space-filling SLE$_{\kappa}$ for $\kappa \in (4, 8)$ can be obtained from ordinary SLE$_{\kappa}$ by iteratively filling in the ``bubbles'' which the path disconnects from $\infty$.

\paragraph{Construction of the a priori embedding.}
An important feature of the mated-CRT map is that it comes with an a priori embedding into $\C$ described by SLE-decorated LQG. To explain this embedding, consider a doubly marked quantum sphere $(\C, h, 0, \infty)$ and, for $\gamma \in (0,2)$, consider the associated $\gamma$-LQG measure $\mu_h$.  Sample a space-filling SLE$_{\kappa}$ curve $\theta$ with $\kappa = 16/\gamma^2$, independently from the random generalized function $h$, and reparametrize $\theta$ so that
\begin{equation*}
\theta(0) = 0 \quad \text{ and } \quad \mu_h\l(\theta([a, b]\r) = b-a, \quad \forall a, b \in \R \text{ with } a < b.
\end{equation*}
For $\gamma \in (0, 2)$ and $n \in \N$, we define the $n$-\emph{structure graph} $\GG^n$ associated with the pair $(h, \theta)$ as follows. The vertex set of $\GG^n$ is given by 
\begin{equation*}
\VV\GG^n := \frac{1}{n} \Z \cap (0, 1].	
\end{equation*}
Two distinct vertices $x_1$, $x_2 \in \VV\GG^n$ are connected by one edge (resp.\ two edges) if and only if the intersection of the corresponding cells $\theta([x_1 - 1/n, x_1])$ and $\theta([x_2 - 1/n, x_2])$ has one connected component which is not a singleton (resp.\ two connected components which are not singletons). We refer to Figure~\ref{fig_cells} for a diagrammatic construction of the SLE/LQG embedding of the mated-CRT map.

\begin{figure}[H]
\centering
\includegraphics[scale=0.54]{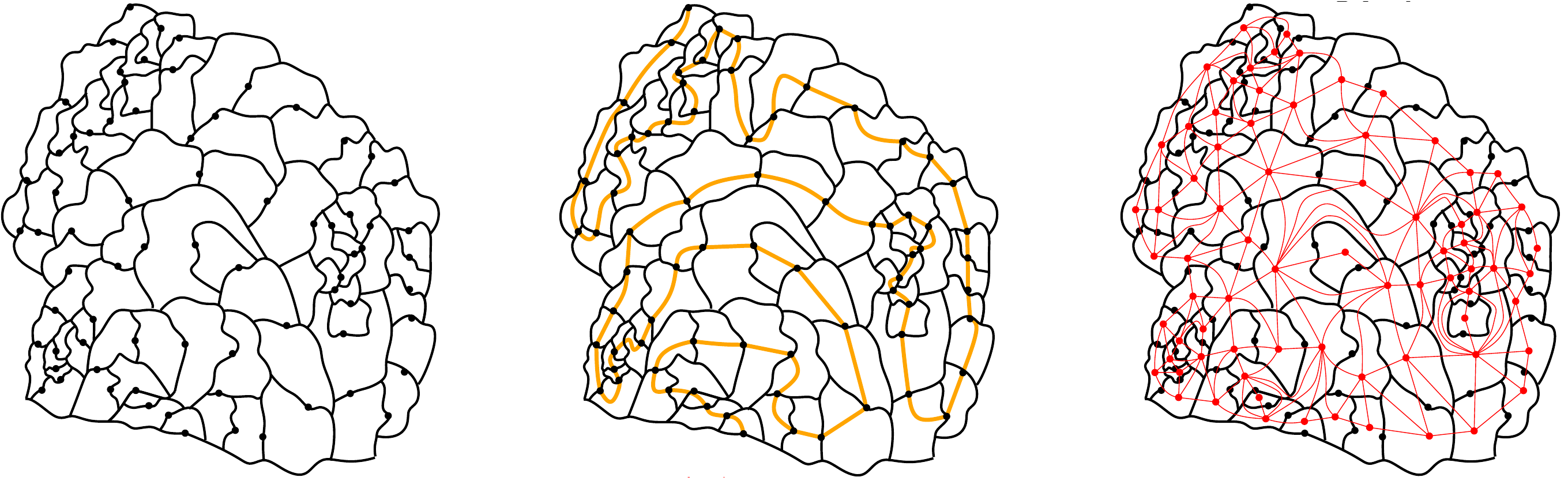}
\caption[short form]{\small \textbf{Left:} A space-filling SLE$_{\kappa}$ curve $\theta$, for $\kappa \geq 8$, divided into cells $\theta([x - 1/n, x])$ for a collection of $x \in \VV\GG^n$. Each cell has $\mu_h$-measure equal to $1/n$. \textbf{Middle:} The same as in the left figure but with a orange path showing the order in which cells are traversed by $\theta$. \textbf{Right:} In each cell we drawn a red point corresponding to a vertex in $\VV\GG^n$. Two vertices are connected by a red edge if the corresponding cells intersect. This illustrates how the SLE/LQG embedding of the $n$-mated CRT map with the sphere topology is built. A similar figure has appeared in \cite{GMS_Harmonic}.}
\label{fig_cells}
\end{figure}

\noindent
When $\kappa \geq 8$, the intersection of cells $\theta([x_1 - 1/n, x_1]) \cap \theta([x_2 - 1/n, x_2])$ is always either empty or the union of one or two non-singleton connected components. However, when $\kappa \in (4,8)$ it is possible that $\theta([x_1 - 1/n, x_1]) \cap \theta([x_2 - 1/n, x_2])$ is a totally disconnected Cantor-like set, in which case $x_1$ and $x_2$ are not joint by an edge in $\GG^n$ (see Figure~\ref{fig_cells_irr}).

\begin{figure}[h]
\centering
\includegraphics[scale=0.8]{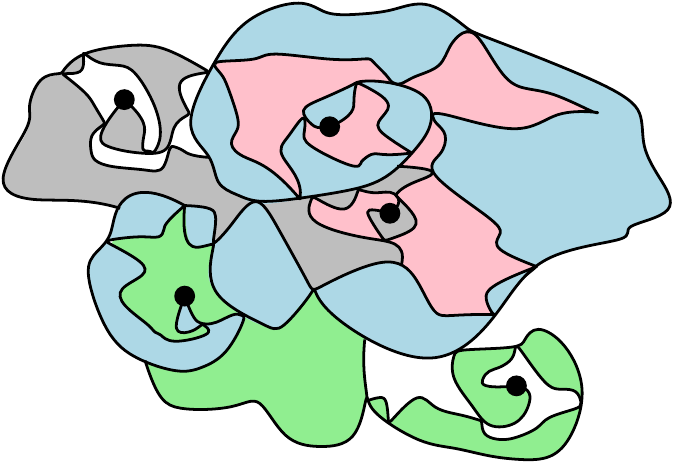}
\caption[short form]{\small A typical space-filling SLE$_{\kappa}$ curve $\theta$, for $\kappa \in (4, 8)$, divided into cells. The picture is slightly misleading since the set of ``pinch point'' where the left and right boundaries of each cell meet is actually uncountable. The black dots indicate the points where $\theta$ starts and finishes filling each cell. Note that the gray and green cells intersect at several points, but do not share a connected boundary arc so are not considered to be adjacent. A similar figure has appeared in \cite{GMS_Tutte}.}
\label{fig_cells_irr}
\end{figure}

\medskip
\noindent
The following result explains the connection between the pair $(h, \theta)$ and the mated-CRT map $\GG^n$, and it is a consequence of \cite[Theorems~1.9~and~8.18]{DMS21}.

\begin{proposition}
The family of structure graphs $\{\GG^n\}_{n \in \N}$ agrees in law with the family of $n$-mated-CRT maps with the sphere topology defined in Subsection~\ref{sec_mated-CRT}.
\end{proposition}

\noindent
The previous proposition gives us an a priori embedding of the mated-CRT map in $\C$ by sending each $x \in \VV\GG^n$ to the point $\theta(x) \in \C$. Furthermore, the graph $\GG^n$ comes naturally with two marked vertices. Indeed, we can let $v^n_0 \in \VV\GG^n$ (resp.\ $v^n_1 \in \VV\GG^n$) to be the vertex corresponding to the cell containing $0$ (resp.\ $\infty$). We also emphasize that here all the edges in $\EE\GG^n$ are assumed to have unit conductance. 

\paragraph{Construction of the a priori embedding of the dual graph.}
The a priori SLE/LQG embedding of $\smash{\GG^n}$ also induces an a priori embedding into $\C$ of the associated planar dual graph $\smash{\hat{\GG}^n}$. Indeed, each vertex of $\smash{\hat{\GG}^n}$ is naturally identified with the set of points in $\C$ where three of the cells $\theta([x - 1/n, x])$ meet\footnote{Note that there cannot be more than three cells meeting at a single point or we would have a face of degree greater than three.}, for $x \in \VV\GG^n$. The set of edges of $\smash{\hat{\GG}^n}$ can be identified with the boundary segments of the cells which connect these vertices. To be precise, this is not an embedding when $\kappa\in (4,8)$ since the edges can intersect (but not cross) each other (see Figure~\ref{fig_cells_irr}). To deal with this case, we can very slightly perturb the edges so that they do not intersect except at their endpoints. We refer to the proof of Proposition~\ref{pr_invariance_dual_mated} for more details on how one can handle this situation.

%%%%%%%%%%%%%%%%%%%%%%%%%%%%%%%%%%%%%%%%%%
\subsection{Mated-CRT maps satisfy the assumptions}
\label{sub_mated_assumptions}
In this subsection, we show that the a priori embedding of mated-CRT maps with the sphere topology satisfies the assumptions of Theorem~\ref{th_main_1}. In particular, we will show here that assumptions~\ref{it_invariance}~and~\ref{it_invariance_dual} are satisfied in this specific case. We recall that, for each $n \in \N$ and $\gamma \in (0, 2)$, the mated-CRT map $\GG^n$ comes with two marked vertices $v_0^n$ and $v_1^n$ which correspond to the cell containing $0$ and $\infty$, respectively

\begin{proposition}[Invariance principle on the mated-CRT map, {\cite[Theorem~3.4]{GMS_Tutte}}]
\label{pr_invariance_mated}
For each $n \in \N$ and $\gamma \in (0, 2)$, let $\GG^n$ be the $n$-mated CRT map with the sphere topology embedded in $\C$ under the a priori SLE/LQG embedding specified in Subsection~\ref{sub_SLE_LQG}. View the embedded random walk on $\GG^n$, stopped when it hits $v_1^n$, as a continuous curve in $\C$ obtained by piecewise linear interpolation at constant speed. For each compact subset $K \subset \C$ and for any $z \in K$, the conditional law given the pair $(h, \theta)$ of the random walk on $\GG^n$ started from the vertex $x_z^n \in \VV\GG^n$ nearest to $z$ weakly converges in probability as $n \to \infty$ to the law of the Brownian motion on $\C$ started from $z$ with respect to the local topology on curves viewed modulo time parameterization specified in Subsection \ref{subsub_metric_CMP}, uniformly over all $z \in K$.
\end{proposition}

\begin{remark}
To be precise, in \cite[Theorem~3.4]{GMS_Tutte}, the above theorem is stated for a mated-CRT map built using a quantum sphere with only one marked point. However, the quantum sphere with two marked points can be obtained from the quantum sphere with one marked point $(\C, h, \infty)$ by first sampling a point $z \in \C$ uniformly from the $\gamma$-LQG area measure $\mu_h$, then applying a conformal map which sends $z$ to $0$ (see \cite[Proposition~A.13]{DMS21}). Therefore, since Brownian motion is conformally invariant, the statement in \cite[Theorem~3.4]{GMS_Tutte} immediately implies Proposition~\ref{pr_invariance_mated} for the quantum sphere with two marked points. 
\end{remark}

\noindent
The main purpose of this subsection is to prove that an analogous result holds also on the sequence of a priori SLE/LQG embedding of the dual $n$-mated-CRT map with the sphere topology. More precisely, we want to prove the following proposition.

\begin{proposition}[Invariance principle on the dual mated-CRT map]
\label{pr_invariance_dual_mated}
For each $n \in \N$ and $\gamma \in (0, 2)$, let $\hat{\GG}^n$ be the dual planar graph associated to the $n$-mated-CRT map $\GG^n$ with the sphere topology embedded in $\C$ under the a priori embedding specified in Subsection~\ref{sub_SLE_LQG}. View the embedded random walk on $\hat{\GG}^n$ as a continuous curve in $\C$ obtained by piecewise linear interpolation at constant speed. For each compact subset $K \subset \C$ and for any $z \in K$, the conditional law given the pair $(h, \theta)$ of the random walk on $\hat{\GG}^n$ started from the vertex $\hat{x}_z^n \in \VV\hat{\GG}^n$ nearest to $z$ weakly converges as $n \to \infty$ to the law of the Brownian motion on $\C$ started from $z$ with respect to the local topology on curves viewed modulo time parameterization specified in Subsection \ref{subsub_metric_CMP}, uniformly over all $z \in K$.
\end{proposition}

\subsubsection{Invariance principle on embedded lattices}
In order to prove Proposition~\ref{pr_invariance_dual_mated}, we need to take a step back and state a general theorem from \cite{GMS22} which guarantees that the random walk on certain embedded lattices has Brownian motion as a scaling limit. For the reader's convenience, we recall here the main definitions needed and in what follows we adopt notation similar to that of \cite{GMS22}. 

\begin{definition}
\label{def_embedded_lattice}
An embedded lattice $\MM$ is a graph embedded in $\C$ in such a way that each edge is a simple curve with zero Lebesgue measure, the edges intersect only at their endpoints, and each compact subset of $\C$ intersects at most finitely many edges of $\MM$. As usual, we write $\VV\MM$ for the set of vertices of $\MM$, $\EE\MM$ for the set of edges of $\MM$, and $\FF\MM$ for the set of faces of $\MM$.
\end{definition}

\begin{definition}
For an embedded lattice $\MM$ and $x \in \VV\MM$, we define the \emph{outradius} of $x$ by setting
\begin{equation*}
\Out(x) := \diam\left(\bigcup_{\{H \in \FF\MM \, : \, x \in \partial H\}} H\right),
\end{equation*}
i.e., the diameter of the union of the faces with $x$ on their boundaries. Here $\partial H$ denotes the boundary of the face $H \in \FF\MM$.
\end{definition}

\noindent
For $C > 0$ and $z \in \C$, we write $C(\MM - z)$ for the embedded lattice obtained by first translating everything by the amount $-z$ and then by scaling everything by the factor $C$. We are interested in random embedded lattices that satisfy the following assumptions.
\begin{enumerate}[start=1,label={(I\arabic*})]
	\item \label{a_1}(\textbf{Translation invariance modulo scaling}) There is a (possibly random and $\MM$-dependent) increasing sequence of open sets $U_j \subset \C$, each of which is either a square or a disk, whose union is all of $\C$ such that the following is true. Conditional on $\MM$ and $U_j$, let $z_j$ for $j \in \N$ be sampled uniformly from the Lebesgue measure on $U_j$. Then the shifted lattice $\MM - z_j$ converge in law to $\MM$ modulo scaling as $j \to \infty$, i.e., there are random numbers $C_j > 0$ (possibly depending on $\MM$ and $z_j$) such that $C_j(\MM - z_j) \to  \MM$ in law with respect to the metric specified in \cite[Equation~(1.6)]{GMS22}.\footnote{Several equivalent formulations of this condition are given in \cite[Definition 1.2]{GMS22}.}
	\item \label{a_2} (\textbf{Ergodicity modulo scaling}) Every real-valued measurable function $F = F(\MM)$ which is invariant under translation and scaling, i.e., $F(C(\MM - z)) = F(\MM)$ for each $z \in \C$ and $C > 0$, is almost surely equal to a deterministic constant.
	\item \label{a_3} (\textbf{Finite expectation}) Let $H_0 \in \FF\MM$ be the face of $\MM$ containing 0, then 
\begin{equation*}
\E\left[\sum_{x \in \VV\MM \cap \partial H_0}\frac{\Out(x)^2 \deg(x)}{\Area(H_0)}\right] < \infty,
\end{equation*}
where, in this case, $\deg(x)$ denotes the number of edges incident to $x$.
\end{enumerate}

\begin{proposition}[{\cite[Theorem~1.11]{GMS22}}]
\label{pr_invariance_general}
Let $\MM$ be an embedded lattice satisfying assumptions~\ref{a_1}, \ref{a_2}, \ref{a_3}. For $\eps > 0$, view the embedded random walk on $\eps\MM$ as a continuous curve in $\C$ obtained by piecewise linear interpolation at constant speed. For each compact subset $K \subset \C$ and for any $z \in K$, the conditional law given $\MM$ of the random walk on $\eps\MM$ started from the vertex $\hat{x}_z^{\eps} \in \VV(\eps\MM)$ nearest to $z$ weakly converges in probability as $\eps \to 0$ to the law of a Brownian motion on $\C$, with some deterministic, non-degenerate covariance matrix $\Sigma$, started from $z$ with respect to the local topology on curves viewed modulo time parameterization specified in Subsection \ref{subsub_metric_CMP}, uniformly over all $z \in K$.
\end{proposition}

\noindent
To be precise, the above theorem follows from the proof \cite[Theorem~1.11]{GMS22} (which gives the same statement but without the uniform rate of convergence on compact subsets) by using \cite[Theorem~3.10]{GMS22} in place of \cite[Theorem~1.16]{GMS22}.

%%%%%%%%%%%%%%%%%%%%%%%%%%%%%%%%%%%%%%%%%%
\subsubsection{Proof of Proposition~\ref{pr_invariance_dual_mated}}
The purpose of this subsection is to transfer the general result of Proposition~\ref{pr_invariance_general} to the particular setting of Proposition~\ref{pr_invariance_dual_mated}. In order to do this, we need to proceed in two steps. First, we verify that the hypothesis of Proposition~\ref{pr_invariance_general} are satisfied for a sequence of embedded lattices built through the $0$-quantum cone. Then, we transfer the result to the sequence of $n$-mated CRT maps with the sphere topology by means of an absolute continuity argument.

\medskip
\noindent
Let $(\C, h^{\rm{c}}, 0, \infty)$ be a 0-quantum cone as specified in Subsection~\ref{sub_SLE_LQG}. We now define a graph associated with the 0-quantum cone in a way which is exactly analogous to the SLE/LQG description of mated-CRT maps with the sphere topology described in Section \ref{sub_SLE_LQG}.
 For $\gamma \in (0, 2)$, let $\theta$ be a whole-plane space-filling SLE$_\kappa$, with $\kappa = 16/\gamma^2$, sampled independently from $h^{\rm{c}}$ and then parameterized by the $\gamma$-LQG measure $\mu_{h^{\rm{c}}}$ in such a way that $\theta(0) = 0$. Let $\xi$ be sampled uniformly from the unit interval $[0, 1]$, independently from everything else, and let
\begin{equation*}
	\HH := \l\{\theta\l([x - 1, x]\r) \, : \, x \in \Z + \xi\r\}.
\end{equation*}
The reason for considering times in $\Z +\xi$ instead just in $\Z$ is to avoid making the point $0 = \theta(0)$ special. This is needed in order to check the translation invariance modulo scaling hypothesis \ref{a_1}. 
We view $\HH$ as a planar map whose vertex set is $\HH$ itself. Two vertices $H$, $H' \in \HH$ are joined by one edge (resp.\ two edges) if and only if $H$ and $H'$ intersect along one (resp.\ two) non-trivial connected  boundary arc. (resp.\ arcs). For $H$, $H' \in \HH$, with $H \neq H'$, we write $H \sim H'$ if they are joined by at least one edge. Given a set $A \subset \C$, we write 
\begin{equation*}
\HH(A):= \l\{H \in \HH \, : \, H \cap A \neq \emptyset\r\}
\end{equation*}
and moreover, for $H \in \HH$, we write $\deg(H)$ for the number of cells $H' \in \HH$ (counted with edge multiplicity) such that $H \sim H'$.

\begin{proof}[Proof of Proposition~\ref{pr_invariance_dual_mated}]
For $\gamma \in (0, 2)$, define the cell configuration $\HH$ associated with a space-filling SLE$_{\kappa}$ on a $0$-quantum cone as specified above. Let $\MM$ be the embedded defined as follows.
\begin{itemize}
	\item The vertex set of $\MM$ consists of points $x \in \C$ such that there are three cells in $\HH$ that meet at $x$ and with the property that the boundary of each of such cells has a connected subset that touches $x$;
	\item The edge set of $\MM$ consists of the boundary segments of the cells which connect the vertices.
\end{itemize}
In other words, $\MM$ is nothing but the planar dual of $\HH$. Moreover, by construction, $\MM$ is an embedded lattice in the sense of Definition~\ref{def_embedded_lattice}, except that embedded edges are allowed to intersect but not cross, in the case when $\gamma \in (\sqrt{2}, 2)$. To deal with this case, we can consider a different embedding of $\MM$ in which all the vertices occupy the same position, but the edges are slightly perturbed so that they do not intersect except at their endpoints. More precisely, for each edge with touching points in its interior, we can slightly perturb it in such a way that the modification only depends on the position of the edge itself and on the positions of the adjacent faces. In particular, this perturbation can be carried out in a translation and scaling invariant way. Therefore, since this procedure only depends on the local configuration of the lattice, if the starting cell configuration is translation invariant modulo scaling, then also the perturbed lattice is translation invariant modulo scaling.

\medskip
\noindent
We will now check that $\MM$ satisfies assumptions~\ref{a_1},~\ref{a_2},~and~\ref{a_3}. The translation invariance modulo scaling assumption~\ref{a_1} and the ergodicity modulo scaling assumption~\ref{a_2} follows from the corresponding properties for the associated cell configuration $\HH$ as checked in \cite[Proposition 3.1]{GMS_Tutte}. Therefore, we can just focus on proving the finite expectation assumption~\ref{a_3}. To this end, recalling that the cell $H_0$ is the face of $\MM$ containing $0$, we proceed as follows
\begin{equation}
\label{eq_first_0_cone}
\begin{alignedat}{1}
\sum_{x \in \VV\MM \cap \partial H_0} \frac{\Out(x)^2 \deg(x)}{\Area(H_0)} & \leq 12 \sum_{x \in \VV\MM \cap \partial H_0} \sum_{\{H \in \HH \, : \, x \in \partial H\}} \frac{\diam(H)^2}{\Area(H_0)} \\
& \leq 48 \sum_{\{H \in \HH \, : \,H \sim H_0\}} \frac{\diam(H)^2}{\Area(H_0)} + 12 \frac{\diam(H_0)^2}{\Area(H_0)}\deg(H_0).
\end{alignedat}
\end{equation}
In the first line of~\eqref{eq_first_0_cone}, we used the fact that each vertex of $\MM$ has degree at most $3$ and the inequality $(a+b+c)^2 \leq 4(a^2 +b^2 +c^2)$. In the second line, we use that each cell $H \in \HH$ with $H \sim H_0$ intersects $H_0$ along at most two disjoint connected boundary arcs (one on its left boundary and one on its right boundary), so there are at most $4$ vertices of $\MM$ in $H \cap H_0$. Now, we notice that the second quantity in the last line of \eqref{eq_first_0_cone} has finite expectation thanks to \cite[Theorem~4.1]{GMS_Tutte}. Therefore, we can just focus our attention to the sum appearing in the last line of \eqref{eq_first_0_cone}. Basically, the fact that this sum has finite expectation follows from the combination of several results in \cite{GMS_Tutte} and \cite{GMS22}.

\medskip
\noindent 
Let $F = F(\HH)$ denote the sum appearing in the last line of \eqref{eq_first_0_cone}. We will show that $\E[F] < \infty$ using an ergodic theory result from \cite[Section~2.2]{GMS22}. Let $\{S_k\}_{k \in \Z}$ be the bi-infinite sequence of origin-containing squares of a uniform dyadic system independent from $\HH$, as defined in \cite[Section 2.1]{GMS22}. We will not need the precise definition of these sequence here, but rather we only need to know $S_{k-1} \subset S_k$ for each $k \in \Z$ and $\cup_{k = 1}^{\infty} S_k = \C$. As shown in \cite[Proposition~3.1]{GMS_Tutte}, the cell configuration $\HH$ satisfies a suitable translation invariance modulo scaling assumption, and so, we can apply \cite[Lemma~2.7]{GMS22} to $\HH$ to find that the following is true. If we let $F_z$ for $z \in \C$ be defined in the same manner as $F$ but with the translated cell configuration $\HH-z$ in place of $\HH$, then it holds almost surely that
\begin{equation*}
\E[F] = \lim_{k \to \infty} \frac{1}{\Area(S_k)} \int_{S_k} F_z d z.
\end{equation*}
To bound the right-hand side of the above expression, we can proceed as follows
\begin{equation}
\label{eq_second_0_cone}
\begin{alignedat}{1}
\lim_{k \to \infty} \frac{1}{\Area(S_k)} \int_{S_k} F_z d z & = \lim_{k \to \infty} \frac{1}{\Area(S_k)} \int_{S_k} \sum_{\{H \in \HH \, : \, H \sim  H_z\}} \frac{\diam(H)^2}{\Area(H_z)} d z \\
& \leq \limsup_{k \to \infty} \frac{1}{\Area(S_k)} \sum_{H \in \HH(S_k)} \sum_{\{H' \in \HH \, : \, H' \sim  H\}} \diam(H')^2. 
\end{alignedat}
\end{equation}
Since the maximal size of the cells in $\HH(S_k)$ is almost surely of strictly smaller order than the side length of $S_k$ as $k \to \infty$ (see \cite[Lemma 2.9]{GMS22}), we find that almost surely for large enough $k \in \N$, each $H' \in \HH$ with $H' \sim H$ for some $H \in \HH(S_k)$ is contained in $\HH(S_k(1))$, where $S_k(1)$ is the square with the same center as $S_k$ and three times the side length of $S_k$. Therefore, the last line in \eqref{eq_second_0_cone} is almost surely bounded above by
\begin{equation*}
	2 \limsup_{k \to \infty} \frac{1}{\Area(S_k)} \sum_{H \in \HH(S_k(1))} \diam(H)^2 \deg(H), 
\end{equation*}
which is finite by \cite[Lemmas~2.8~and~2.10]{GMS22}. 

\medskip
\noindent
Therefore, summing up, we proved that the embedded lattice $\MM$ satisfies assumptions~\ref{a_1},~\ref{a_2},~\ref{a_3}, ad so, we can apply Proposition~\ref{pr_invariance_general}. Furthermore, we notice that, thanks to the rotational invariance of the law of $(h^{\rm{c}}, \theta)$, the limiting covariance matrix $\Sigma$ is given by a positive scalar multiple of the identity matrix. Hence, in order to conclude the proof of Proposition~\ref{pr_invariance_dual_mated}, we need to transfer the result to the setting of mated-CRT maps with the sphere topology. This can be done by means of absolute continuity arguments as in \cite[Subsection~3.3]{GMS_Tutte}. For the sake of brevity, we will not repeat such arguments here and we refer to \cite{GMS_Tutte}.
\end{proof}

%%%%%%%%%%%%%%%%%%%%%%%%%%%%%%%%%%%%%%%%%%
\subsection{Convergence to LQG}
\label{sub_conv_LQG}
In this subsection, we see how the proof of Theorem~\ref{th_main_2} is almost an immediate consequence of the results proved above. More specifically, for $n \in \N$ and $\gamma \in (0, 2)$, let $(\GG^n, v_0^n, v_1^n)$ be the doubly marked $n$-mated-CRT map with the sphere topology under the a priori SLE/LQG embedding as specified in Subsection~\ref{sub_SLE_LQG}.
We observe that there exists a conformal map from $\C$ to $\CC_{2 \pi}$ sending $\infty \mapsto +\infty$ and $0 \mapsto -\infty$. This mapping is unique up to horizontal translation and rotation. The horizontal translation can be fixed by specifying that the volume of $\R/2\pi\Z\times [0, \infty)$ under the $\gamma$-LQG measure $\mu_h$ induced by the embedding is precisely $1/2$. Furthermore, the rotation on $\R/2\pi\Z$ can be chosen uniformly at random. Therefore, using this conformal map, we can define an embedding of $(\GG^n, v^n_0, v^n_1)$ into $\CC_{2\pi}$ in such a way that $\smash{v_0^n}$ and $\smash{v_1^n}$ are mapped to $-\infty$ and $+\infty$, respectively. Using the same conformal map, we can also embedded the associated dual graph $\smash{\hat{\GG}^n}$ into $\smash{\CC_{2\pi}}$. This put us exactly in the setting of Theorem~\ref{th_main_1} from which we can deduce the following result. 
\begin{proposition}
\label{pr_close_mated}
Fix $\gamma \in (0, 2)$ and let $\{(\GG^n, v_0^n, v_1^n)\}_{n \in \N}$ be the sequence of doubly marked $n$-mated CRT map with the sphere topology embedded in the infinite cylinder $\CC_{2\pi}$ as specified above. For each $n \in \N$, let $\dotSS_n : \VV\GG^{n} \to \CC_{2\pi}$ be the Smith embedding associated to $(\GG^n, v_0^n, v_1^n)$ as specified in Definition \ref{def_dotted_Smith}. There exists a sequence of random affine transformations $\{T_n\}_{n \in \N}$ from $\CC_{\eta_n}$ to $\CC_{2\pi}$ of the form specified in the statement of Theorem~\ref{th_main_1} such that, for all compact sets $K \subset \CC_{2\pi}$, the following convergence holds in probability 
\begin{equation*}
\lim_{n \to \infty} \sup_{x \in \VV\GG^{n}(K)} \d_{2\pi}\l(T_n \dotSS_n(x), x\r) = 0,
\end{equation*}
where $\d_{2\pi}$ denotes the Euclidean distance on the cylinder $\CC_{2\pi}$.
\end{proposition}
\begin{proof}
By construction the sequence $\{(\GG^n, v_0^n, v_1^n)\}_{n \in \N}$ and the associated sequence of dual graphs $\{\hat{\GG}^n\}_{n \in \N}$ satisfy almost surely assumption~\ref{it_embedding}. Moreover, Proposition~\ref{pr_invariance_mated} guarantees the convergence in probability of the random walk on the sequence of primal graphs to Brownian motion and so assumption~\ref{it_invariance} is satisfied. Furthermore, Proposition~\ref{pr_invariance_dual_mated} guarantees the convergence in probability of the random walk on the sequence of dual graphs to Brownian motion and so also assumption~\ref{it_invariance_dual} is satisfied. Therefore, the desired result follows from Theorem~\ref{th_main_1}.
\end{proof}

\noindent
Using the same procedure specified at the beginning of this subsection, we can also consider the parameterization of the unit-area quantum sphere by the infinite cylinder $\CC_{2\pi}$. Hence, with a slight abuse of notation, we let $(\CC_{2\pi}, h, -\infty, +\infty)$ be the unit-area quantum sphere parametrized by $\CC_{2\pi}$ and we denote by $\mu_{h}$ the associated $\gamma$-LQG measure. We are now ready to prove Theorem~\ref{th_main_2}.

\begin{proof}[Proof of Theorem~\ref{th_main_2}]
The result on the measure convergence \ref{mated_a} follows from Proposition~\ref{pr_close_mated} and the fact that $\theta$ is parameterized by $\mu_h$-mass. The uniform convergence statement for curves \ref{mated_b} is also an immediate consequence of Proposition~\ref{pr_close_mated}. The claimed random walk convergence \ref{mated_c} follows from Propositions~\ref{pr_invariance_mated}~and~\ref{pr_close_mated}.
\end{proof}
%%%%%%%%%%%%%%%%%%%%%%%%%%%%%%%%%%%%%%%%%%

%%%%%%%%%%%%%%%%%%%%%%%%%%%%%%%%%%%%%%%%%%
\appendix
\section{Some standard estimates for planar Brownian motion}
Throughout the whole appendix, $B$ denotes a standard planar Brownian motion. For $\x \in \R^2$, we use the notation $B^\x$ to denote a planar Brownian motion started from $\x$. We recall that $\sigma_{2\pi} : \R^2 \to \CC_{2 \pi}$ is defined in \eqref{eq_covering_map} and denotes the covering map of the infinite cylinder $\CC_{2 \pi}$. In particular, if $B^\x$ is a as above, then $\sigma_{2\pi}(B^\x)$ is a Brownian motion on $\CC_{2 \pi}$ started from $\sigma_{2 \pi}(\x)$.

\begin{lemma}
\label{lm_std_BM_1}
Fix $R' > 1$. For $\x \in \R^2$, consider the following stopping times
\begin{equation*}
\tau := \inf \l\{t \geq 0 \, : \, \Im(B_t^\x) = - 2R' \text{ or } \Im(B_t^\x) =  R'\r\}.
\end{equation*}
Then it holds that
\begin{equation*}
\P_\x\l(\sigma_{2 \pi}(B|_{[0, \tau]}) \text{ does not wind around the cylinder below height } - R'\r) \lesssim \frac{1}{R'}, \quad  \forall \x \in \R \times \{-R'\},
\end{equation*}
where the implicit constant is independent of everything else. 
\end{lemma}
\begin{proof}
Fix $\x \in \R \times \{-R'\}$. Let $\sigma_0 : = 0$, and for $k \in \N_0$ we define inductively 
\begin{equation*}
\tau_{k} := \inf\l\{t \geq \sigma_{k} \, : \, \Im(B^\x_t) = -R'-1\}, \quad \sigma_{k+1} := \inf\l\{t \geq \tau_{k} \, : \, \Im(B^\x_t) = -R' \r\}.
\end{equation*}
Moreover, for $k \in \N_0$ consider the events
\begin{equation*}
A_k := \l\{|\Re(B^\x_{\sigma_k}) - \Re(B^\x_{\tau_k})| \geq 1\r\} , \quad F_{k} := \l\{\tau_{k} > \tau\r\}.
\end{equation*}
Let $K$ be the smallest $k \in \N_0$ such that $F_k$ occurs. Then the probability that the event in the lemma statement happens is less or equal to the probability that none of the events $\{A_k\}_{k \in [K]_0}$ happen. Thanks to the strong Markov property of the Brownian motion, the events $\{A_k\}_{k \in \N_0}$ are independent and identically distributed. 
Moreover, thanks to well-known properties of Brownian motion, the event $A_0$ happens with uniformly positive probability $p$ independent of $R'$ and $\x$.  Therefore, we obtain that 
\begin{equation*}
\P_{\x}\left(\bigcap_{i = 0}^K\bar{A_k}\right) = \sum_{k \in \N_0} \P_{\x}\left(\bigcap_{i = 0}^k \bar{A_k}\right)  \P_\x(K=k) = \sum_{k \in \N_0} (1-p)^{k} \P_\x(K = k) \lesssim \frac{1}{R'}, 
\end{equation*}
for all $\x \in \R \times \{R'\}$, where the implicit constant is independent of everything else, and the last inequality follows from standard Brownian motion estimates. Hence, this concludes the proof.
\end{proof}

\begin{lemma}
\label{lm_std_BM_2}
Fix $R > 1$, $R' > R$. For $\x \in \R \times [-R, R]$, let $\tau :=\inf \{t \geq 0 \, : \, |\Im(B^\x_t)| = R'\}$, and define the events 
\begin{equation*}
\begin{alignedat}{1}
& W^{+} := \l\{\sigma_{2 \pi} (B^\x|_{[0, \tau]}) \text{ does a loop around the cylinder between heights } R \text{ and } R'\r\}, \\
& W^{-} := \l\{\sigma_{2 \pi} (B^\x|_{[0, \tau]}) \text{ does a loop around the cylinder between heights } -R' \text{ and } -R \r\}.
\end{alignedat}
\end{equation*}
Then it holds that
\begin{equation*}
\P_\x\l(\bar{W^{+}} \cup \bar{W^{-}}\r) \lesssim \frac{R}{R'}, \quad \forall \x \in \R \times [-R, R],
\end{equation*}
where the implicit constant is independent of everything else.
\end{lemma}
\begin{proof}
Fix $\x \in \R \times [-R, R]$. It is sufficient to prove that $\P_\x(\bar{W^{+}}) \lesssim R/R'$, and the same with $W^{-}$ in place of $W^{+}$. We will proceed similarly to the proof of the previous lemma. Let $\sigma_0 : = \inf\{t \geq 0 \, : \, \Im(B^\x_t) = R\}$ and, for $k \in \N_0$ we define inductively 
\begin{equation*}
\tau_{k} := \inf\l\{t \geq \sigma_{k} \, : \, \Im(B^\x_t) = R+1\r\}, \quad \sigma_{k+1} := \inf\l\{t \geq \tau_{k} \, : \, \Im(B^\x_t) = R\r\}.
\end{equation*}
Moreover, for $k \in \N_0$ consider the events 
\begin{equation*}
A_k := \l\{|\Re(B^\x_{\sigma_k}) - \Re(B^\x_{\tau_k})| = 1 , \, \Im(B^\x_{\sigma_k}) = \Im(B^\x_{\tau_k})\r\},  \quad F_{k} := \l\{\tau_{k} > \tau\r\} .
\end{equation*}
Let $K$ be the smallest $k \in \N_0$ such that $F_k$ occurs. Then the probability that the event $W^{+}$ does not happen is less or equal to the probability that none of the events $\{A_k\}_{k \in [K]_0}$ happen. Thanks to the strong Markov property of the Brownian motion, the events $\{A_k\}_{k \in \N_0}$ are independent and identically distributed. Moreover, since the event $A_0$ happens with uniformly positive probability $p$ independent of $R$ and $\x$, we have that 
\begin{equation*}
\P_\x\l(\bar{W^{+}}\r) \leq \P_{\x}\left(\bigcap_{i = 0}^{K}\bar{A_i}\right) = \sum_{k \in \N_0} \P_{\x}\left(\bigcap_{i = 0}^{k}\bar{A_i}\right) \P_\x(K = k)  = \sum_{k \in \N_0} (1-p)^{k}  \P_\x(K = k) \lesssim \frac{R}{R'}, 
\end{equation*}
for all $\x \in \R \times [-R, R]$, where the implicit constant is independent of everything else, and the last inequality follows from standard Brownian motion estimates. Proceeding in a similar way, one can also prove that $\P_\x(\bar{W^{-}}) \lesssim R/R'$. Therefore, the desired result follows from the fact that $\P_\x(\bar{W^{+}} \cup \bar{W^{-}}) \leq \P_\x(\bar{W^{+}}) + \P_\x(\bar{W^{-}})$. 
\end{proof}

%%%%%%%%%%%%%%%%%%%%%%%%%%%%%%%%%%%%%%%%%%
\small
\bibliography{ref}{}

\newcommand{\etalchar}[1]{$^{#1}$}
\def\cprime{$'$}
\begin{thebibliography}{DKRV16}
\def\myhref#1#2{\href{#2}{\nolinkurl{#1}}}

\bibitem[AB99]{AB99}
\textsc{M.~Aizenman} and \textsc{A.~Burchard}.
\newblock H\"{o}lder regularity and dimension bounds for random curves.
\newblock \emph{Duke Math. J.} \textbf{99}, no.~3, (1999), 419--453.
\newblock
  \myhref{doi:10.1215/S0012-7094-99-09914-3}{https://dx.doi.org/10.1215/S0012-7094-99-09914-3}.

\bibitem[ABL16]{Berry_square}
\textsc{L.~Addario-Berry} and \textsc{N.~Leavitt}.
\newblock Random infinite squarings of rectangles.
\newblock \emph{Ann. Inst. Henri Poincar\'{e} Probab. Stat.} \textbf{52},
  no.~2, (2016), 596--611.
\newblock
  \myhref{doi:10.1214/14-AIHP661}{https://dx.doi.org/10.1214/14-AIHP661}.

\bibitem[Ber15]{Ber_LBM}
\textsc{N.~Berestycki}.
\newblock Diffusion in planar {L}iouville quantum gravity.
\newblock \emph{Ann. Inst. Henri Poincar\'{e} Probab. Stat.} \textbf{51},
  no.~3, (2015), 947--964.
\newblock
  \myhref{doi:10.1214/14-AIHP605}{https://dx.doi.org/10.1214/14-AIHP605}.

\bibitem[Ber17]{Berestycki_Elementary}
\textsc{N.~Berestycki}.
\newblock An elementary approach to {G}aussian multiplicative chaos.
\newblock \emph{Electron. Commun. Probab.} \textbf{22}, (2017), Paper No. 27,
  12.
\newblock \myhref{doi:10.1214/17-ECP58}{https://dx.doi.org/10.1214/17-ECP58}.

\bibitem[Ber23]{BerMulti}
\textsc{F.~Bertacco}.
\newblock Multifractal analysis of {G}aussian multiplicative chaos and
  applications.
\newblock \emph{Electron. J. Probab.} \textbf{28}, (2023), Paper No. 2, 36.
\newblock \myhref{doi:10.1214/22-ejp893}{https://dx.doi.org/10.1214/22-ejp893}.

\bibitem[BG22]{GB_LBM}
\textsc{N.~Berestycki} and \textsc{E.~Gwynne}.
\newblock Random walks on mated-{CRT} planar maps and {L}iouville {B}rownian
  motion.
\newblock \emph{Comm. Math. Phys.} \textbf{395}, no.~2, (2022), 773--857.
\newblock
  \myhref{doi:10.1007/s00220-022-04482-y}{https://dx.doi.org/10.1007/s00220-022-04482-y}.

\bibitem[BP24]{BerPow}
\textsc{N.~Berestycki} and \textsc{E.~Powell}.
\newblock Gaussian free field and liouville quantum gravity (2024).
\newblock \myhref{arXiv:2404.16642}{https://arxiv.org/abs/2404.16642}.

\bibitem[BS96]{BS96}
\textsc{I.~Benjamini} and \textsc{O.~Schramm}.
\newblock Random walks and harmonic functions on infinite planar graphs using
  square tilings.
\newblock \emph{Ann. Probab.} \textbf{24}, no.~3, (1996), 1219--1238.
\newblock
  \myhref{doi:10.1214/aop/1065725179}{https://dx.doi.org/10.1214/aop/1065725179}.

\bibitem[BSST40]{BSST40}
\textsc{R.~L. Brooks}, \textsc{C.~A.~B. Smith}, \textsc{A.~H. Stone}, and
  \textsc{W.~T. Tutte}.
\newblock The dissection of rectangles into squares.
\newblock \emph{Duke Math. J.} \textbf{7}, (1940), 312--340.
\newblock
  \myhref{doi:10.1215/S0012-7094-40-00718-9}{https://dx.doi.org/10.1215/S0012-7094-40-00718-9}.

\bibitem[CG20]{Carm_Smith}
\textsc{J.~Carmesin} and \textsc{A.~Georgakopoulos}.
\newblock Every planar graph with the {L}iouville property is amenable.
\newblock \emph{Random Structures Algorithms} \textbf{57}, no.~3, (2020),
  706--729.
\newblock \myhref{doi:10.1002/rsa.20936}{https://dx.doi.org/10.1002/rsa.20936}.

\bibitem[DDDF20]{DDDF_first}
\textsc{J.~Ding}, \textsc{J.~Dub\'{e}dat}, \textsc{A.~Dunlap}, and
  \textsc{H.~Falconet}.
\newblock Tightness of {L}iouville first passage percolation for {$\gamma \in
  (0,2)$}.
\newblock \emph{Publ. Math. Inst. Hautes \'{E}tudes Sci.} \textbf{132}, (2020),
  353--403.
\newblock
  \myhref{doi:10.1007/s10240-020-00121-1}{https://dx.doi.org/10.1007/s10240-020-00121-1}.

\bibitem[DFG{\etalchar{+}}20]{Weak_Metric}
\textsc{J.~Dub\'{e}dat}, \textsc{H.~Falconet}, \textsc{E.~Gwynne},
  \textsc{J.~Pfeffer}, and \textsc{X.~Sun}.
\newblock Weak {LQG} metrics and {L}iouville first passage percolation.
\newblock \emph{Probab. Theory Related Fields} \textbf{178}, no. 1-2, (2020),
  369--436.
\newblock
  \myhref{doi:10.1007/s00440-020-00979-6}{https://dx.doi.org/10.1007/s00440-020-00979-6}.

\bibitem[DKRV16]{LQG_Sphere}
\textsc{F.~David}, \textsc{A.~Kupiainen}, \textsc{R.~Rhodes}, and
  \textsc{V.~Vargas}.
\newblock Liouville quantum gravity on the {R}iemann sphere.
\newblock \emph{Comm. Math. Phys.} \textbf{342}, no.~3, (2016), 869--907.
\newblock
  \myhref{doi:10.1007/s00220-016-2572-4}{https://dx.doi.org/10.1007/s00220-016-2572-4}.

\bibitem[DMS21]{DMS21}
\textsc{B.~Duplantier}, \textsc{J.~Miller}, and \textsc{S.~Sheffield}.
\newblock Liouville quantum gravity as a mating of trees.
\newblock \emph{Ast\'{e}risque} , no. 427, (2021), viii+257.
\newblock \myhref{doi:10.24033/ast}{https://dx.doi.org/10.24033/ast}.

\bibitem[DS11]{DS11}
\textsc{B.~Duplantier} and \textsc{S.~Sheffield}.
\newblock Liouville quantum gravity and {KPZ}.
\newblock \emph{Invent. Math.} \textbf{185}, no.~2, (2011), 333--393.
\newblock
  \myhref{doi:10.1007/s00222-010-0308-1}{https://dx.doi.org/10.1007/s00222-010-0308-1}.

\bibitem[Geo16]{Geo16}
\textsc{A.~Georgakopoulos}.
\newblock The boundary of a square tiling of a graph coincides with the
  {P}oisson boundary.
\newblock \emph{Invent. Math.} \textbf{203}, no.~3, (2016), 773--821.
\newblock
  \myhref{doi:10.1007/s00222-015-0601-0}{https://dx.doi.org/10.1007/s00222-015-0601-0}.

\bibitem[GGJN19]{Nach_Circle}
\textsc{O.~Gurel-Gurevich}, \textsc{D.~C. Jerison}, and \textsc{A.~Nachmias}.
\newblock A combinatorial criterion for macroscopic circles in planar
  triangulations.
\newblock \emph{arXiv} (2019).
\newblock \myhref{arXiv:1906.01612}{https://arxiv.org/abs/1906.01612}.

\bibitem[GHS19]{GHS_mated}
\textsc{E.~Gwynne}, \textsc{N.~Holden}, and \textsc{X.~Sun}.
\newblock A distance exponent for {L}iouville quantum gravity.
\newblock \emph{Probab. Theory Related Fields} \textbf{173}, no. 3-4, (2019),
  931--997.
\newblock
  \myhref{doi:10.1007/s00440-018-0846-9}{https://dx.doi.org/10.1007/s00440-018-0846-9}.

\bibitem[GM20]{GM_Conf}
\textsc{E.~Gwynne} and \textsc{J.~Miller}.
\newblock Confluence of geodesics in {L}iouville quantum gravity for {$\gamma
  \in (0,2)$}.
\newblock \emph{Ann. Probab.} \textbf{48}, no.~4, (2020), 1861--1901.
\newblock
  \myhref{doi:10.1214/19-AOP1409}{https://dx.doi.org/10.1214/19-AOP1409}.

\bibitem[GM21]{GM_Uniq}
\textsc{E.~Gwynne} and \textsc{J.~Miller}.
\newblock Existence and uniqueness of the {L}iouville quantum gravity metric
  for {$\gamma\in(0,2)$}.
\newblock \emph{Invent. Math.} \textbf{223}, no.~1, (2021), 213--333.
\newblock
  \myhref{doi:10.1007/s00222-020-00991-6}{https://dx.doi.org/10.1007/s00222-020-00991-6}.

\bibitem[GMS19]{GMS_Harmonic}
\textsc{E.~Gwynne}, \textsc{J.~Miller}, and \textsc{S.~Sheffield}.
\newblock Harmonic functions on mated-{CRT} maps.
\newblock \emph{Electron. J. Probab.} \textbf{24}, (2019), Paper No. 58, 55.
\newblock \myhref{doi:10.1214/19-EJP325}{https://dx.doi.org/10.1214/19-EJP325}.

\bibitem[GMS20]{GMS_Voronoi}
\textsc{E.~Gwynne}, \textsc{J.~Miller}, and \textsc{S.~Sheffield}.
\newblock The {T}utte embedding of the {P}oisson-{V}oronoi tessellation of the
  {B}rownian disk converges to {$\sqrt{8/3}$}-{L}iouville quantum gravity.
\newblock \emph{Comm. Math. Phys.} \textbf{374}, no.~2, (2020), 735--784.
\newblock
  \myhref{doi:10.1007/s00220-019-03610-5}{https://dx.doi.org/10.1007/s00220-019-03610-5}.

\bibitem[GMS21]{GMS_Tutte}
\textsc{E.~Gwynne}, \textsc{J.~Miller}, and \textsc{S.~Sheffield}.
\newblock The {T}utte embedding of the mated-{CRT} map converges to {L}iouville
  quantum gravity.
\newblock \emph{Ann. Probab.} \textbf{49}, no.~4, (2021), 1677--1717.
\newblock
  \myhref{doi:10.1214/20-aop1487}{https://dx.doi.org/10.1214/20-aop1487}.

\bibitem[GMS22]{GMS22}
\textsc{E.~Gwynne}, \textsc{J.~Miller}, and \textsc{S.~Sheffield}.
\newblock An invariance principle for ergodic scale-free random environments.
\newblock \emph{Acta Math.} \textbf{228}, no.~2, (2022), 303--384.
\newblock
  \myhref{doi:10.4310/ACTA.2022.v228.n2.a2}{https://dx.doi.org/10.4310/ACTA.2022.v228.n2.a2}.

\bibitem[GP20]{GP_square}
\textsc{A.~Georgakopoulos} and \textsc{C.~Panagiotis}.
\newblock Convergence of square tilings to the {R}iemann map.
\newblock \emph{arXiv} (2020).
\newblock \myhref{arXiv:1910.06886}{https://arxiv.org/abs/1910.06886}.

\bibitem[GRV16]{GRV_LBM}
\textsc{C.~Garban}, \textsc{R.~Rhodes}, and \textsc{V.~Vargas}.
\newblock Liouville {B}rownian motion.
\newblock \emph{Ann. Probab.} \textbf{44}, no.~4, (2016), 3076--3110.
\newblock
  \myhref{doi:10.1214/15-AOP1042}{https://dx.doi.org/10.1214/15-AOP1042}.

\bibitem[HP17]{Hutch_Smith}
\textsc{T.~Hutchcroft} and \textsc{Y.~Peres}.
\newblock Boundaries of planar graphs: a unified approach.
\newblock \emph{Electron. J. Probab.} \textbf{22}, (2017), Paper No. 100, 20.
\newblock \myhref{doi:10.1214/17-EJP116}{https://dx.doi.org/10.1214/17-EJP116}.

\bibitem[HS21]{HS_Cardy}
\textsc{N.~Holden} and \textsc{X.~Sun}.
\newblock {C}onvergence of uniform triangulations under the {C}ardy embedding.
\newblock \emph{arXiv} (2021).
\newblock \myhref{arXiv:1905.13207}{https://arxiv.org/abs/1905.13207}.

\bibitem[Kah85]{Kahane}
\textsc{J.-P. Kahane}.
\newblock Sur le chaos multiplicatif.
\newblock \emph{Ann. Sci. Math. Qu\'{e}bec} \textbf{9}, no.~2, (1985),
  105--150.

\bibitem[LP16]{LP16}
\textsc{R.~Lyons} and \textsc{Y.~Peres}.
\newblock \emph{Probability on trees and networks}, vol.~42 of \emph{Cambridge
  Series in Statistical and Probabilistic Mathematics}.
\newblock Cambridge University Press, New York, 2016,  xv+699.
\newblock
  \myhref{doi:10.1017/9781316672815}{https://dx.doi.org/10.1017/9781316672815}.

\bibitem[MS17]{IM_four}
\textsc{J.~Miller} and \textsc{S.~Sheffield}.
\newblock Imaginary geometry {IV}: interior rays, whole-plane reversibility,
  and space-filling trees.
\newblock \emph{Probab. Theory Related Fields} \textbf{169}, no. 3-4, (2017),
  729--869.
\newblock
  \myhref{doi:10.1007/s00440-017-0780-2}{https://dx.doi.org/10.1007/s00440-017-0780-2}.

\bibitem[Nac20]{Nac20}
\textsc{A.~Nachmias}.
\newblock \emph{Planar maps, random walks and circle packing}, vol. 2243 of
  \emph{Lecture Notes in Mathematics}.
\newblock Springer, Cham, 2020,  xii+118.
\newblock \'{E}cole d'\'{e}t\'{e} de probabilit\'{e}s de Saint-Flour
  XLVIII---2018, \'{E}cole d'\'{E}t\'{e} de Probabilit\'{e}s de Saint-Flour.
  [Saint-Flour Probability Summer School].
\newblock
  \myhref{doi:10.1007/978-3-030-27968-4}{https://dx.doi.org/10.1007/978-3-030-27968-4}.

\bibitem[Pol81]{Pol81}
\textsc{A.~M. Polyakov}.
\newblock Quantum geometry of bosonic strings.
\newblock \emph{Phys. Lett. B} \textbf{103}, no.~3, (1981), 207--210.
\newblock
  \myhref{doi:10.1016/0370-2693(81)90743-7}{https://dx.doi.org/10.1016/0370-2693(81)90743-7}.

\bibitem[RS87]{RS87}
\textsc{B.~Rodin} and \textsc{D.~Sullivan}.
\newblock The convergence of circle packings to the {R}iemann mapping.
\newblock \emph{J. Differential Geom.} \textbf{26}, no.~2, (1987), 349--360.
\newblock
  \myhref{doi:10.4310/jdg/1214441375}{https://dx.doi.org/10.4310/jdg/1214441375}.

\bibitem[RV11]{RV_KPZ}
\textsc{R.~Rhodes} and \textsc{V.~Vargas}.
\newblock K{PZ} formula for log-infinitely divisible multifractal random
  measures.
\newblock \emph{ESAIM Probab. Stat.} \textbf{15}, (2011), 358--371.
\newblock
  \myhref{doi:10.1051/ps/2010007}{https://dx.doi.org/10.1051/ps/2010007}.

\bibitem[RV14]{RVReview}
\textsc{R.~Rhodes} and \textsc{V.~Vargas}.
\newblock Gaussian multiplicative chaos and applications: a review.
\newblock \emph{Probab. Surv.} \textbf{11}, (2014), 315--392.
\newblock \myhref{doi:10.1214/13-PS218}{https://dx.doi.org/10.1214/13-PS218}.

\bibitem[Sch00]{Sh00}
\textsc{O.~Schramm}.
\newblock Scaling limits of loop-erased random walks and uniform spanning
  trees.
\newblock \emph{Israel J. Math.} \textbf{118}, (2000), 221--288.
\newblock
  \myhref{doi:10.1007/BF02803524}{https://dx.doi.org/10.1007/BF02803524}.

\bibitem[She16]{She_Zipper}
\textsc{S.~Sheffield}.
\newblock Conformal weldings of random surfaces: {SLE} and the quantum gravity
  zipper.
\newblock \emph{Ann. Probab.} \textbf{44}, no.~5, (2016), 3474--3545.
\newblock
  \myhref{doi:10.1214/15-AOP1055}{https://dx.doi.org/10.1214/15-AOP1055}.

\end{thebibliography}
\bibliographystyle{Martin}
%%%%%%%%%%%%%%%%%%%%%%%%%%%%%%%%%%%%%%%%%%

\end{document}